%% file: qc.fox.tex
\documentclass[11pt, reqno, a4paper]{amsart}
\usepackage{graphicx, epsfig, psfrag}
\usepackage{amsfonts,amsmath,amssymb,amsbsy,amsthm}
\usepackage{bm}
\usepackage{color}
\usepackage{float}
\usepackage{hyperref}
\usepackage{mathrsfs}
\usepackage{subfigure}
\usepackage{bbm}
\usepackage{longtable}

\usepackage{pdfsync}

\usepackage{environments}

\def\figs{.}

\renewcommand{\paragraph}[1]{\subsubsection{#1}}

\input{notation}

\definecolor{cocol}{rgb}{0.7, 0, 0}

\begin{document}

\title[The role of the patch test in 2D atomistic-to-continuum
coupling
methods]{The role of the patch test in 2D \\
  atomistic-to-continuum coupling methods}

\author{C. Ortner}
\address{C. Ortner\\ Mathematical Institute\\
  24-29 St Giles' \\ Oxford OX1 3LB \\ UK}
\email{ortner@maths.ox.ac.uk}

\date{\today}

\thanks{This work was supported by the EPSRC Critical Mass Programme
  ``New Frontiers in the Mathematics of Solids'' (OxMoS) and by the
  EPSRC grant EP/H003096 ``Analysis of atomistic-to-continuum coupling
  methods''.}



\begin{abstract}
  For a general class of atomistic-to-continuum coupling methods,
  coupling multi-body interatomic potentials with a $\PI$-finite
  element discretisation of Cauchy--Born nonlinear elasticity, this
  paper adresses the question whether {\em patch test consistency}
  (or, absence of {\em ghost forces}) implies a first-order error
  estimate.
  
  In two dimensions it is shown that this is indeed true under the
  following additional technical assumptions: (i) an energy
  consistency condition, (ii) locality of the interface correction,
  (iii) volumetric scaling of the interface correction, and (iv)
  connectedness of the atomistic region. The extent to which these
  assumptions are necessary is discussed in detail.
\end{abstract}

\maketitle


\section{Introduction}
\label{sec:intro}
Defects in crystalline materials
interact through their elastic fields far beyond their atomic
neighbourhoods. An accurate computation of such defects requires the
use of atomistic models; however, the size of the atomistic systems
that are often required to accurately represent the elastic far-field
makes atomistic models infeasible or, at the very least, grossly
inefficient. Indeed, atomistic accuracy is often only required in a
small neighbourhood of the defect, while the elastic far field may be
approximated using an appropriate continuum elasticity model.

Atomistic-to-continuum coupling methods (a/c methods) aim to exploit
this fact by retaining atomistic models only in small neighbourhoods
of defects, and coupling these neighbourhoods to finite element
discretisations of continuum elasticity models; see Figure
\ref{fig:intro:meshes}(c,d). By employing a coarse discretisation of
the continuum model, such a process can achieve a considerable
reduction in computational complexity, however, some of the first a/c
methods suffered from the so-called ``ghost force problem'': While
homogeneous deformations are equilibria of both the pure atomistic and
the pure continuum model, they are not equilibria of certain a/c
models \cite{Ortiz:1995a, Shenoy:1999a, Miller:2003a} due to spurious
forces --- the ``ghost forces'' --- that can arise at the interface
between the atomistic and continuum regions.

Much of the literature on a/c methods has focused on constructing a/c
methods that did not exhibit, or reduced the effect of the ``ghost
forces'' \cite{Shenoy:1999a, cadd, Shimokawa:2004, E:2006, badi08,
  Fish:2007, Dobson:2008a, Makr:2010a, Shapeev:2010a, KlZi:2006,
  XiBe:2004, IyGa:2011}. A straightforward solution was the
introduction of {\em force-based} (i.e, non-conservative) methods
\cite{Kohlhoff:1989, Shenoy:1999a, cadd, Dobson:2008a, Fish:2007,
  Makr:2010a}. The construction of accurate {\em energy-based}
coupling mechanisms turned out to be more challenging. Several
creative approaches providing partial solutions to the problem were
suggested \cite{Shimokawa:2004, E:2006, Shapeev:2010a, IyGa:2011},
however, no general solution exists so far.

The inconsistency of early a/c methods is reminiscent of the
inconsistency problems encountered in the early history of finite
element methods. A simple criterion to test consistency of finite
element methods is the {\em patch test} introduced by Irons {\it et
  al} \cite{Irons:1966a}; see also
\cite{StrangFix2008,Belytschko:2000a}. The ``ghost force problem''
discussed above is precisely the failure of such a patch test. It is
well known that, in general, the patch test is neither necessary nor
sufficient for convergence of finite element methods; see, e.g.,
\cite{StrangFix2008,Belytschko:2000a} where several variants of patch
tests are discussed. The same is true for a/c methods: It was shown in
\cite{Luskin:clusterqc} that a particular flavour of force-based a/c
coupling typically has a consistency error of nearly $100\%$, even
though it does pass the patch test.

Although a growing numerical analysis literature exists on the subject
of a/c methods \cite{LinP:2003a, OrtnerSuli:2008a, Dobson:2008a,
  Dobson:2008b, Ortner:qnl.1d, Dobson:qce.stab, emingyang, E:2005a,
  MiLu:2011}, it has so far focused primarily on one-dimensional model
problems. (A notable exception is the work of Lu and Ming on
force-based hybrid methods \cite{MiLu:2011}. However, the techniques
used therein require large overlaps and cannot accommodate sharp
interfaces.) Only specific methods are analyzed; the question whether
absence of ``ghost forces'' (or, {\em patch test consistency}, as we
shall call it) in general implies satisfactory accuracy has neither
been posed nor addressed so far. The purpose of the present paper is
to fill precisely this gap. After introducing a general atomistic
model and a general class of abstract a/c methods, and establishing
the necessary analytical framework, it will be shown in Theorem
\ref{th:2d:main_result}, which is the main result of the paper, that
in two dimensions patch test consistency together with additional
technical assumptions implies first-order consistency of energy-based
a/c methods.

\subsection{Outline}
\label{sec:intro:outline}
\S\ref{sec:pre} gives a detailed introduction to the construction of
a/c methods. This section also develops a new notation that is well
suited for the analysis of 2D and 3D a/c methods. In
\S\ref{sec:intro:patch_test} we give a precise statement of the patch
test consistency condition.

\S\ref{sec:fram} contains a general framework for the a priori error
analysis of a/c methods in $\WW^{1,p}$-norms, similar to an error
analysis of Galerkin methods with variational crimes. Several new
technical results are presented in this section, such as the
introduction of an oscillation operator to measure local smoothness of
discrete functions (\S\ref{sec:fram:osc}), an interpolation error
estimate for piecewise affine functions (\S\ref{sec:fram:interp}), and
making precise the assumption made in much of the a/c numerical
analysis literature that it is sufficient to estimate the modelling
error without coarsening (\S\ref{sec:fram:cons}). The main ingredient
left open in this analysis is a stability assumption, which requires a
significant amount of additional work and is beyond the scope of this
paper.

\S\ref{sec:e1d} presents two 1D examples for modelling error
estimates, which motivate the importance of patch test consistency,
and discusses the modelling error estimates that can at best be
expected if a method is patch test consistent.

\S\ref{sec:aux} introduces the two main auxiliary results used in the
2D analysis of \S\ref{sec:2d}: (i) Shapeev's bond density lemma, which
allows the translation between bond integrals and volume integrals;
and (ii) a representation theorem for discrete divergence-free
$\PO$-tensor fields.

Finally, \S\ref{sec:2d} establishes the main result of this paper,
Theorem \ref{th:2d:main_result}: if an a/c method is patch test
consistent and satisfies other natural technical assumptions, then it
is also first-order consistent. The proof depends on a novel
construction of stress tensors for atomistic models, related to the
virial stress \cite{AdTa:2010} (generalizing the 1D construction in
\cite{MakrOrtSul:qcf.nonlin, Makr:2010a}), and a corresponding
construction for the stress tensor associated with the a/c
energy. Moreover, we discuss in detail to what extent the technical
conditions of Theorem \ref{th:2d:main_result} are required. For
example, we show that, if the atomistic region is finite (i.e.,
completely surrounded by the continuum region) then the condition of
global energy consistency already follows from patch test consistency.

\subsection{Sketch of the main result}
\label{sec:intro:main_result}
In this section we discuss the main result in non-rigorous terms. Let
$\Ea$ be an atomistic energy functional and let $\Eac$ be an a/c
energy functional, which uses the atomistic description in part of the
computational domain, and couples it to a finite element
discretisation of Cauchy--Born nonlinear elasticity
(cf. \S\ref{sec:intro:qce} for its definition), with suitable
interface treatment. Suppose that the following conditions are
satisfied:
\begin{itemize}
\item[(i)] {\it $\Eac$ is patch test consistent:} Every homogeneous
  deformation is a critical point of $\Eac$
  (cf. \S~\ref{th:intro:a_patchtest}).
\item[(ii)] {\it $\Eac$ is globally energy consistent:} $\Eac$ is exact for
  homogeneous deformations (cf. \S~\ref{sec:intro:qce}).
\item[(iii)] {\it Locality and scaling of the interface correction:}
  The interface correction has (roughly) the same interaction range as
  the atomistic model, and has volumetric scaling (cf.
  \S\ref{sec:intro:loc_scal}).
\item[(iv)] The atomistic region $\Oma$ is connected.
\item[(v)] {\it Stability: } For some $p \in [1, \infty]$, and for
  deformations $y$ in a neighbourhood of the atomistic solution, the
  second variation $\ddel\Eac(y)$ is stable, when understood as a
  linear operator from (discrete variants of) $\WW^{1,p}$ to
  $\WW^{-1,p}$ (cf. \S\ref{sec:fram:stab}).
\item[(vi)] The finite element mesh in the continuum region is shape
  regular. \cite{Ciarlet:1978}
\end{itemize}

Under conditions (v) and (vi), we show in \S\ref{sec:fram} that the
error between the atomistic solution $y_\a$ and the a/c solution
$y_\ac$ can be bounded by
\begin{equation}
  \label{eq:idea:errest_1}
  \| \D y_\a - \D y_\ac \|_{\LL^p(\Om)} \lesssim \Emodeleps + \Eext +
  \b\| {\rm h} \D^2 y_\a \b\|_{\LL^p(\Omc)}^2,
\end{equation}
where $\Om$ is the computational domain, $\Omc$ the continuum region,
${\rm h}$ a local mesh size function, $\Eext$ is the consistency
error for the treatment of external forces, and $\Emodeleps$ is the
modelling error, which we describe next. Since $y_\a$ is not a
continuous deformation, but a deformation of a discrete lattice, the
various terms appearing above need to be interpreted with care. This
is the focus of \S\ref{sec:fram}. For example, in the rigorous version
of \eqref{eq:idea:errest_1}, we will replace $\D^2 y_\a$ by an
oscillation of a suitably defined gradient.

The step outlined above does not distinguish different variants of a/c
methods. The error introduced by the coupling mechanism and the
continuum model is contained in the modelling error
\begin{displaymath}
  \Emodeleps = \b\| \del\Eac(y_\a) - \del\Ea(y_\a)
  \b\|_{\WW^{-1,p}_\eps},
\end{displaymath}
where $\WW^{-1,p}_\eps$ is a suitably defined negative Sobolev norm on
an atomistic grid. If $\Eac$ is not patch test consistent, then,
typically, $\Emodeleps = O(1)$. However, if conditions (i)--(iv) hold,
and if the problem is set in either one or two space dimensions, then
we will prove in \S\ref{sec:2d} that
\begin{equation}
  \label{eq:idea:errest_2}
  \Emodeleps \lesssim \eps \, \b\| \D^2 y_\a \b\|_{\LL^p(\Omc \cup \Omi)},
\end{equation}
where $\eps$ is the atomistic spacing and $\Omi$ an interface
region. A rigorous statement of \eqref{eq:idea:errest_2} is given in
Theorem \ref{th:2d:main_result}, which is the main result of the
paper. 

\subsection{Basic notational conventions} 
\label{sec:intro:basic_notation}
Vectors are denoted by lower case roman symbols, $x, y$, $a, b$, and
so forth. Matrices are denoted by capital sans serif symbols $\mA, \mB,
\mF, \mG$, and so forth. We will not distinguish between row and
column vectors. If two vectors are multiplied, then it will be
specified whether the operation is the dot product or the tensor
product: for $a, b \in \R^k$, we define
\begin{displaymath}
  a \cdot b := \sum_{j = 1}^k a_j b_j, \quad \text{and} \quad
  a \otimes b := \b( a_i b_j \b)_{\substack{i = 1,\dots, k \\ j = 1,
      \dots, k}},
\end{displaymath}
where, in $a \otimes b$, $i$ is the row index and $j$ the column
index.

When matrix fields take the role of stress tensors, we will also call
them tensor fields and use greek letters $\sigma, \Sigma$, and so
forth.

The Euclidean norm of a vector and the Frobenius norm of a matrix are
denoted by $|\cdot|$. The $\ell^p$-norms, $p \in [1, \infty]$, of a
vector or matrix are denoted by $|\cdot|_p$, and sometimes by $\|\cdot
\|_{\ell^p}$. For $\eps > 0$, the weighted $\ell^p$-norms, on an index
set $\mathscr{S}$, are defined as
\begin{displaymath}
  \| a \|_{\ell^p_\eps(\mathscr{S})} := \eps^{1/p} \| a\|_{\ell^p(\mathscr{S})}.
\end{displaymath}

The topological dual of a vector space $\Us$ is denoted by $\Us^*$,
with duality pairing $\< \cdot, \cdot\>$.

If $A \subset \R^d$ is a measurable set then $|A|$ denotes its
$d$-dimensional volume. The symbol $\mint_A$ denotes $|A|^{-1}
\int_A$, provided that $|A| > 0$.  If $A$ has a well-defined area or
length, then these are denoted, respectively, by ${\rm area}(A)$ and
${\rm length}(A)$.

For a measurable function $f : A \to \R$, $\|f\|_{\LL^p(A)}$ denotes
the standard $\LL^p$-norm. If $f : A \to \R^{k \times m}$, then
$\|f\|_{\LL^p(A)} := \| \,|f|_p \|_{\LL^p(A)}$.

Partial derivatives with respect to a variable $x_j$, say, are denoted
by $\pp / \pp {x_j}$. The Jacobi matrix of a differentiable function
$f : A \to \R^k$ is denoted by $\pp f$.  The symbol $\pp$ will also be
redefined in some contexts, but used with essentially the same meaning
as here.

When $f$ is a deformation or a displacement, then we will also write
$\D f = \pp f$. For $r \in \R^d$, the uni-directional derivative is
denoted by
\begin{displaymath}
  \Dc{r} f(x) := \lim_{t \searrow 0} \frac{f(x+tr) - f(x)}{t},
\end{displaymath}
whenever this limit exists. The symbol $\Da{r}$ denotes a finite
difference operator, which will be defined in
\S\ref{sec:intro:defn_Ea}.

Additional notation will be defined throughout. A list of symbols,
with references to their definitions is given in \S\ref{sec:notation}.

\section{Introduction to Atomistic/Continuum Model Coupling}
\label{sec:pre}
In this section we introduce a general multi-body interaction model
with periodic boundary conditions. This choice of boundary condition
is crucial since the non-locality of the atomistic interactions makes
an analysis of Dirichlet or Neumann boundaries challenging.  The
periodic boundary conditions can also be understood as ``artificial
boundary conditions'' for infinite crystals.

Next, we describe the construction of energy-based a/c methods that
couple the atomistic interaction potential with a $\PI$-finite element
discretization of the Cauchy--Born continuum model.  We motivate the
patch test (``ghost forces''), and why an interface correction is
required to obtain accurate coupling schemes.

\subsection{An atomistic model with periodic boundary condition}
\label{sec:intro:at_model}

\paragraph{Periodic deformations}
\label{sec:intro:at_model:perdef}
Let $d \in \{1,2,3\}$ denote the space dimension. We will, in
subsequent sections, restrict our analysis to $d \in \{1,2\}$,
however, the introduction to a/c coupling methods is independent of
the dimension.

For some $N \in \N$, and $\eps := 1/N$, we define the periodic
reference cell
\begin{displaymath}
  \L := \eps \b\{ -N+1, \dots, N \b\}^{d}.
\end{displaymath}
The space of $2\Z^d$-periodic displacements of $\Lext := \eps \Z^d$ is
given by
\begin{displaymath}
  \Us := \b\{ u :\Lext \to \R^d \bsep u(x+\xi) = u(x)
  \text{ for all } \xi \in 2 \Z^d, x \in \L \b\}.
\end{displaymath}
A homogeneous deformation (or, {\em Bravais lattice}) of $\Lext$ is a
map $\yA : \Lext \to \R^d$, where $\mA \in \R^{d \times d}$ and
$\yA(x) = \mA x$ for all $x \in \Lext$. We denote the space of {\em
  periodic deformations} of $\Lext$ by
\begin{displaymath}
  \Ys := \b\{ y : \Lext \to \R^d \bsep y - \yA \in \Us
  \text{ for some } \mA \in \R^{d \times d} \b\},
\end{displaymath}
and the space of deformations with prescribed {\em macroscopic strain}
$\mA$ by
\begin{displaymath}
  \Ys_\mA := \{ y \in \Ys \sep y - \yA \in \Us \}.
\end{displaymath}
$\Ys$ is a linear space, while $\Ys_\mA$ is an affine subspace of
$\Ys$.

For future reference we define notation that extends sets
periodically: For any set $A \subset \R^d$ we define $A^\per :=
\bigcup_{\xi \in 2 \Z^d} (\xi + A)$. This notation is consistent with
the definition of $\Lext$.  If $\mathscr{A}$ is a family of sets, then
$\mathscr{A}^\per := \{ A^\per \sep A \in \mathscr{A}\}$.

\paragraph{The atomistic energy}
\label{sec:intro:defn_Ea}
For a map $v : \Lext \to \R^k$, $k \in \N$, and $r \in \Z^d \setminus
\{0\}$, we define the finite difference operator
\begin{displaymath}
  \Da{r} v(x) := \frac{v(x+\eps r) - v(x)}{\eps}.
\end{displaymath}

We assume that the {\em stored elastic energy} of a deformation $y \in
\Ys$ is given in the form
\begin{equation}
  \label{eq:intro:Ea}
  \Ea(y) := \eps^d \sum_{x \in \L} V\b(\Da{\Rg} y(x)\b),
\end{equation}
where $\Rg \subset \Z^d \setminus \{0\}$ is a finite {\em interaction
  range}, $\Da{\Rg}y(x) := (\Da{r} y(x))_{r \in \Rg}$, and where $V
\in \CC^2((\R^d)^\Rg)$ is a multi-body interaction potential. Under
these conditions, $\Ea \in \CC^2(\Ys)$.

The scaling of the lattice, of the finite difference operator, and of
the energy were chosen to highlight the natural connection between
molecular mechanics and continuum mechanics. For example, $\eps^d
\sum$ resembles an integral (or, a Riemann sum), while $\Da{r}$
resembles a directional derivative. It should be stressed, however,
that $\eps$ is a fixed parameter of the problem, which is small but
does {\em not} tend to zero.

The formulation \eqref{eq:intro:Ea} includes all commonly employed
{\em classical} interatomic potentials (see, e.g., \cite{Finnis,
  Pettifor}): pair potentials such as the Lennard-Jones or Morse
potential, bond-angle potentials, embedded atom potentials, bond-order
potentials, or any combination of the former, provided that they have
a finite interaction range.  The generality of the interaction
potential also includes effective potentials obtained for in-plane or
anti-plane deformations of 3D crystals.

No major difficulties should be expected in generalizing the analysis
to infinite interaction ranges, provided the interaction strength
decays sufficiently fast.
A generalization of the analysis to genuine long-range interactions
such as Coulomb interactions is not obvious.

\paragraph{Assumptions on the interaction potential}
\label{sec:intro:assV}
For ${\bf g} = (g_{r})_{r \in \Rg} \in (\R^d)^\Rg$ we denote the first
and second partial derivatives of $V$ at ${\bf g}$, respectively, by
\begin{displaymath}
  \pp_r V({\bf g}) := \frac{\pp V({\bf g})}{\pp g_r} \in \R^d,
  \quad \text{and} \quad
  \pp_{r,s} V({\bf g}) := \frac{\pp^2 V({\bf g})}{\pp g_r \pp g_s} \in
  \R^{d \times d}, \qquad \text{for } r, s, \in \Rg.
\end{displaymath}

Throughout this work we assume the following {\em global} bound: there
exist constants $M_{r,s}^\a \geq 0$, $r, s \in \Rg$, such that
\begin{equation}
  \label{eq:intro:bound_DrsV}
  \sup_{{\bf g} \in (\R^d)^\Rg}  \b\| \pp_{r,s} V({\bf g}) \b\| \leq
  M_{r,s}^\a,
  \qquad \text{for all } r, s \in \Rg,
\end{equation}
where $\|\cdot\|$ denotes the $\ell^2$-operator norm. This assumption
contradicts realistic interaction models and is made to simplify the
notation; see \S\ref{sec:intro:invert} for further discussion of this
issue.

In the subsequent analysis we will, in fact, never make direct use of
the second partial derivatives $\pp_{r,s} V$ directly, but only use
the resulting Lipschitz property,
\begin{equation}
  \label{eq:intro:LipV}
  \b| \pp_r V({\bf g}) - \pp_r V({\bf h}) \b| \leq \sum_{s \in \Rg}
  M_{r,s}^\a \,|g_s - h_s| \qquad \forall {\bf g}, {\bf h} \in
  (\R^d)^\Rg, \quad r \in \Rg.
\end{equation}
The proof is straightforward.  For future reference, we also define
the constant
\begin{equation}
  \label{eq:intro:defn_Ma}
  M^\a := \sum_{r, s \in \Rg} |r| |s| M_{r,s}^\a.
\end{equation}

\paragraph{The variational problem}
\label{sec:intro:minA}
Let $\Pa \in \CC^2(\Ys; \R)$ be the {\em potential of external forces}
modelling, for example, a substrate or an indenter. As explained in
\S\ref{sec:intro:defects}, one may also use such a potential to
model simple point defects such as vacancies or impurities. 

Given a potential of external forces $\Pa$ and a macroscopic strain
$\mA$, we consider the problem of finding local minimizers of the
total energy $\Eatot := \Ea + \Pa$ in $\Ys_\mA$, in short,
\begin{equation}
  \label{eq:intro:minEa}
  y_\a \in \underset{y \in \Ys_\mA}{\argmin} \, \Eatot(y),
\end{equation}
where $\argmin$ denotes the set of {\em local} minimizers.

If $y_\a$ solves \eqref{eq:intro:minEa}, then $y_\a$ is a {\em
  critical point} of $\Eatot$, that is,
\begin{equation}
  \label{eq:intro:Ea_crit}
  \b\< \del\Ea(y_\a) + \del\Pa(y_\a), u \b\> = 0 \qquad \forall u
  \in \Us,
\end{equation}
where, for a functional $\mathscr{E} \in \CC^1(\Ys)$, we define the
{\em first variation} of $\mathscr{E}$ at $y$, as
\begin{displaymath}
  \<\del\mathscr{E}(y), z \> :=
  \frac{\dd}{\dd t} \mathscr{E}(y+t z)\b|_{t = 0}
  \qquad \text{for } y, z \in \Ys.
\end{displaymath}
If $\E \in \CC^2(\Ys)$ then the second variation is defined
analogously as
\begin{displaymath}
  \<\ddel\mathscr{E}(y) z_1, z_2 \> :=
  \frac{\dd}{\dd t} \b\< \del\E(y+t z_1), z_2 \b\> \b|_{t = 0}
  \qquad \text{for } y, z_1, z_2 \in \Ys.
\end{displaymath}
The same notation will be used for functionals defined on different
spaces.

\paragraph{The patch test for the atomistic model}
The following proposition can be understood as the {\em patch test}
for the atomistic model: {\it In the absence of external forces and
  defects a homogeneous lattice is always a critical point of $\Ea$.}

\begin{proposition}
  \label{th:intro:a_patchtest}
  $\<\del\Ea(\yA), u\> = 0$ for all $u \in \Us$ and for all $\mA \in
  \R^{d \times d}$.
\end{proposition}
\begin{proof}
  Let $y, z \in \Ys$, then
  \begin{equation}
    \label{eq:intro:delEa}
    \< \del\Ea(y), z\> = \eps^d \sum_{x \in \L} \sum_{r \in \Rg} 
    \pp_r V\b(\Da{\Rg}y(x)\b) \cdot \Da{r} z(x).
  \end{equation}

  Fix $\mA \in \R^{d \times d}$ and $u \in \Us$, then
  \begin{align*}
    \< \del \Ea(\yA), u \> =~& \sum_{r \in \Rg} \eps^d \sum_{x \in \L}
    \pp_r V\b( \mA \Rg \b) \cdot \Da{r} u(x) \\
    =~&  \sum_{r \in \Rg} \pp_r V\b( \mA \Rg \b) \cdot \bg\{
    \eps^d \sum_{x \in \L} \Da{r} u(x) \bg\} = 0,
  \end{align*}
  where we have used the fact that $u$ is periodic. Above, and
  throughout, we use the notation $\mA \Rg = (\mA r)_{r \in \Rg} =
  \Da{\Rg} \yA(x)$.
\end{proof}

\subsection{Remarks on the atomistic model} \quad
\label{sec:intro:remarks}
\paragraph{Invertibility of deformations}
\label{sec:intro:invert}
In \S\ref{sec:intro:defn_Ea} and \S\ref{sec:intro:assV} we have
assumed that $\Ea$ is twice differentiable at {\em all} deformations $y
\in \Ys$, and that the second partial derivatives of the interaction
potential are {\em globally bounded}. 

However, realistic interaction potentials $V$ take the value $+\infty$
if two atoms occupy the same position in space, and hence can only be
differentiable at deformations that are one-to-one (i.e., ``true''
deformations).  With only minor additional technicalities such
potentials can be admitted in the analysis. The global bound
\eqref{eq:intro:bound_DrsV} would then be replaced by a local bound
and certain explicit bounds on $\Da{\Rg}y$; see, e.g.,
\cite{Ortner:qnl.1d, OrtShap:2011a}.

\paragraph{Reference cutoff}
\label{sec:intro:ref_cutoff}
Another aspect of the atomistic energy \eqref{eq:intro:Ea}, which
makes it inappropriate for realistic applications is that the
interaction potential $V$ has a cut-off radius in the reference
configuration. In atomistic models, atoms are unconstrained in their
position and hence two atoms that are far apart in the reference
configuration may be arbitrarily close, and hence interact, in the
deformed configuration.

The reference cutoff in \eqref{eq:intro:Ea} is assumed only for the
sake of brevity of the notation. One can, similarly as discussed in
\S\ref{sec:intro:invert}, take a more general form of the interaction
potential that does not suffer from this drawback and make suitable
assumptions on deformations under consideration that control the
interaction neighbourhood.

\paragraph{Modelling crystal defects}
\label{sec:intro:defects}
Some simple crystal defects can be modelled via the potential of the
external forces, $\Pa$. The simplest example is an impurity, where a single
atom is replaced with an atom from a different species. For
\eqref{eq:intro:Ea} this means that the interaction potential is
changed from $V(\Da{\Rg} y(x))$ to $V^{\rm mod}(x; \Da{\Rg} y(x))$ in a
neighbourhood of the impurity. Alternatively, one may keep the
original form of $\Ea$ and define
\begin{displaymath}
  \Pa(y) = \eps^d \sum_{x \in \L} \b[
  V^{\rm mod}(x; \Da{\Rg} y(x)) - V(\Da{\Rg} y(x)) \b].
\end{displaymath}

Similarly, a vacancy can be modelled by simply removing all
interactions with a given atom. This would yield a difficulty with the
``unused'' degrees of freedom for the position of the vacancy atom,
which could simply be removed from the system \cite{OrtShap:2011a}. An
interstitial (an additional atom) can also be modelled fairly easily,
but one would need to augment the variable $y$ by additional degrees
of freedom for the position of the interstitial atom.

Dislocations, which possibly represent the most important class of
crystal defects are, in general, more difficult to describe. In the
atomistic minimization problem \eqref{eq:intro:minEa} they simply
represent a special class of local minimizers, however, in the coupled
atomistic/continuum models we discuss below most classes of
dislocations less straightforward to embed; however, see
\cite{LiLuOrVK:2011a, Miller:2008} for simple examples.


\subsection{Construction of a/c coupling methods}
\label{sec:intro:ac}
The atomistic model problem \eqref{eq:intro:minEa} is a
finite-dimensional optimisation problem and is therefore, in
principle, solvable using standard optimisation algorithms. However,
typical applications where atomistic models are employed require of
the order $10^9$ to $10^{12}$ atoms or more \cite{Miller:2003a,
  Ortiz:1995a}. It is therefore desirable to construct computationally
efficient {\em coarse grained} models.

\paragraph{Galerkin projection}
\label{sec:intro:galerkin}
To motivate the idea of atomistic-to-continuum coupling we consider a
crystal with a localized defect. Figure \ref{fig:intro:meshes}(a)
shows a deformed 2D crystal with an impurity that repels its
neighbouring atoms, causing a large local deformation. We observe
that, except in a small neighbourhood of the defect, the atoms are
arranged as a ``smooth'' deformation of the reference lattice
$\Lext$. It is therefore possible to approximate the atomistic
configurations from a low-dimensional subspace constructed, for
example, using a $\PI$-finite element method.

\begin{figure}
  \begin{center}
    \includegraphics[width=5cm]{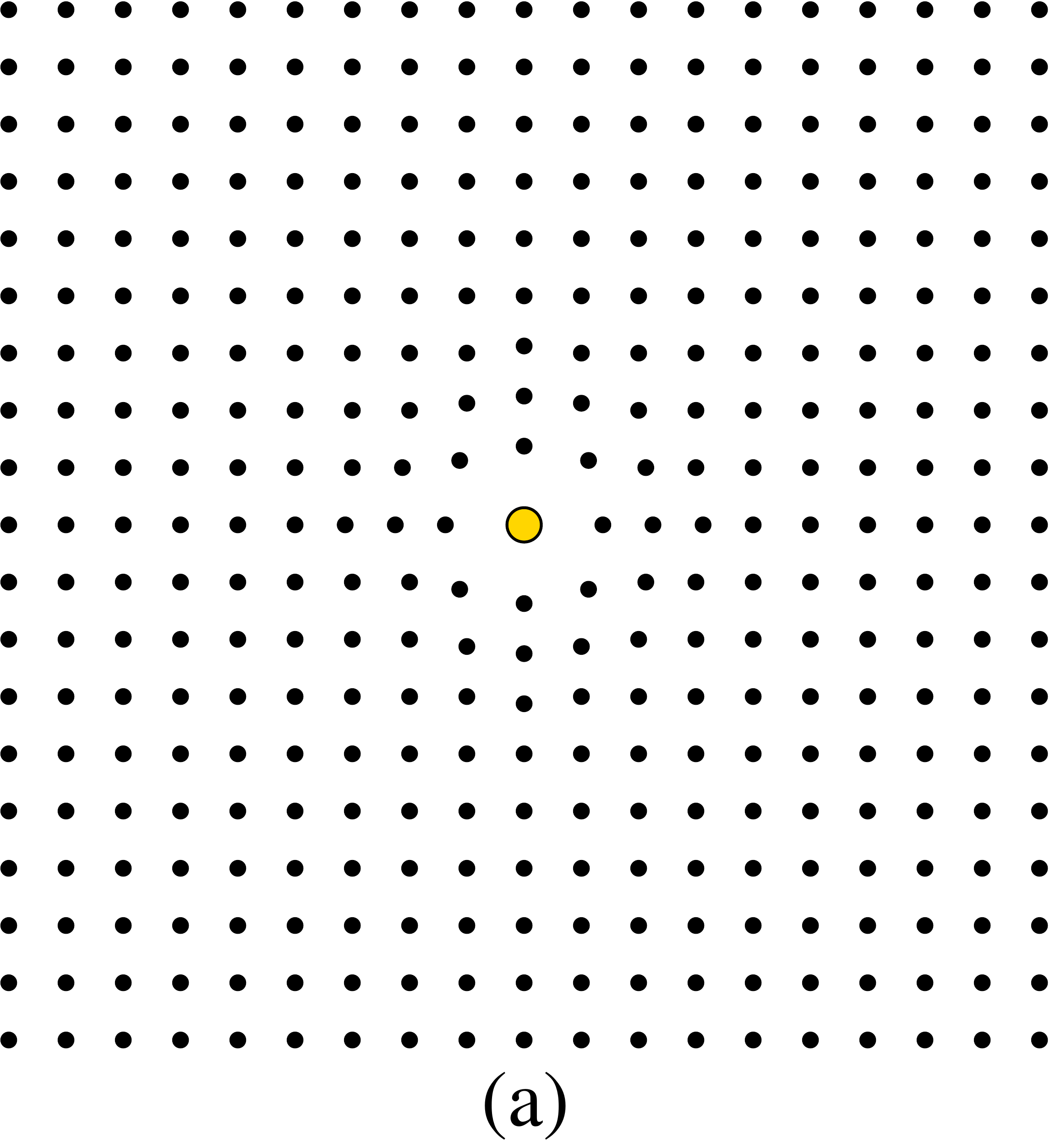} \quad
    \includegraphics[width=5cm]{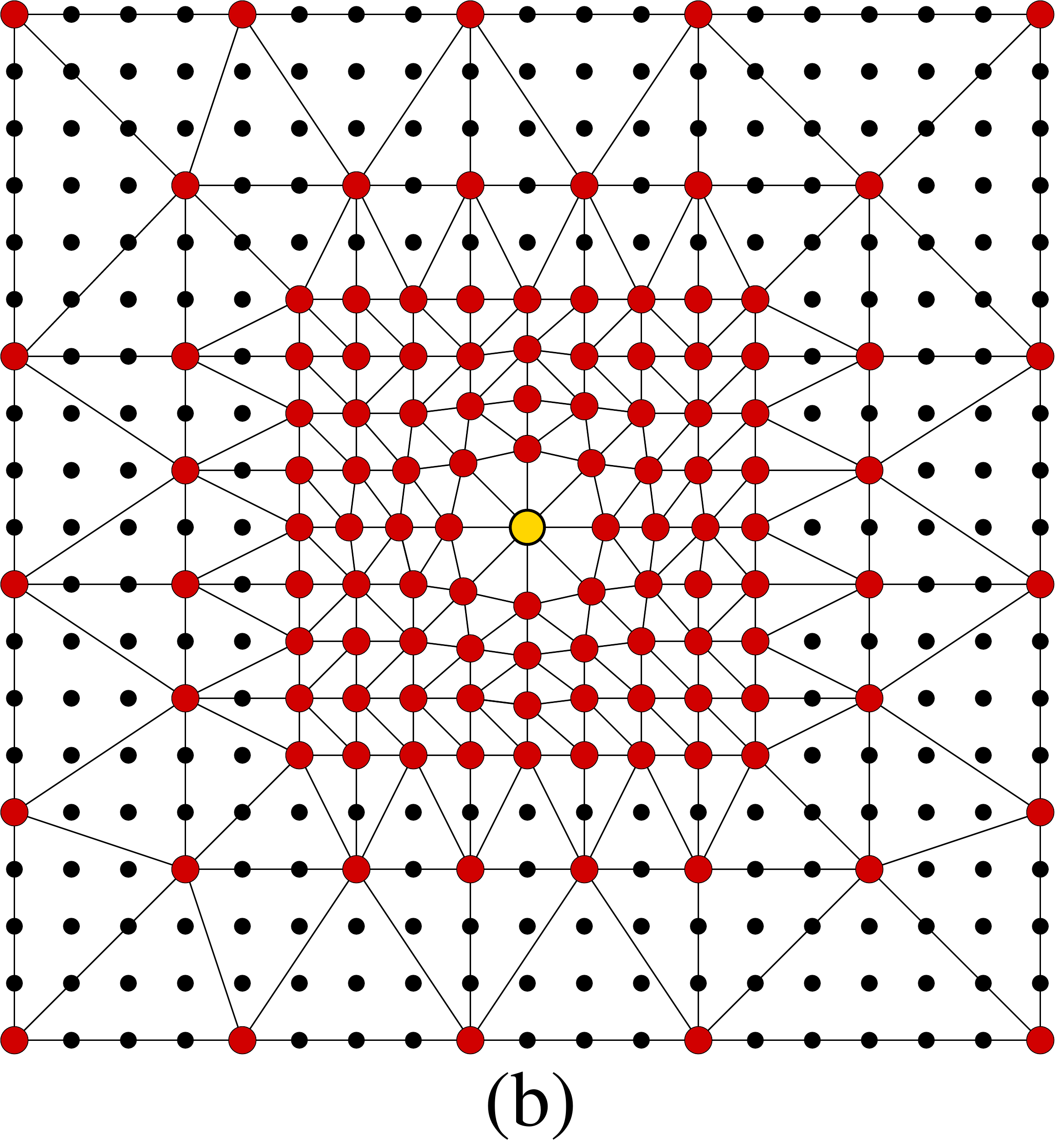} \\[2mm]
    \includegraphics[width=5cm]{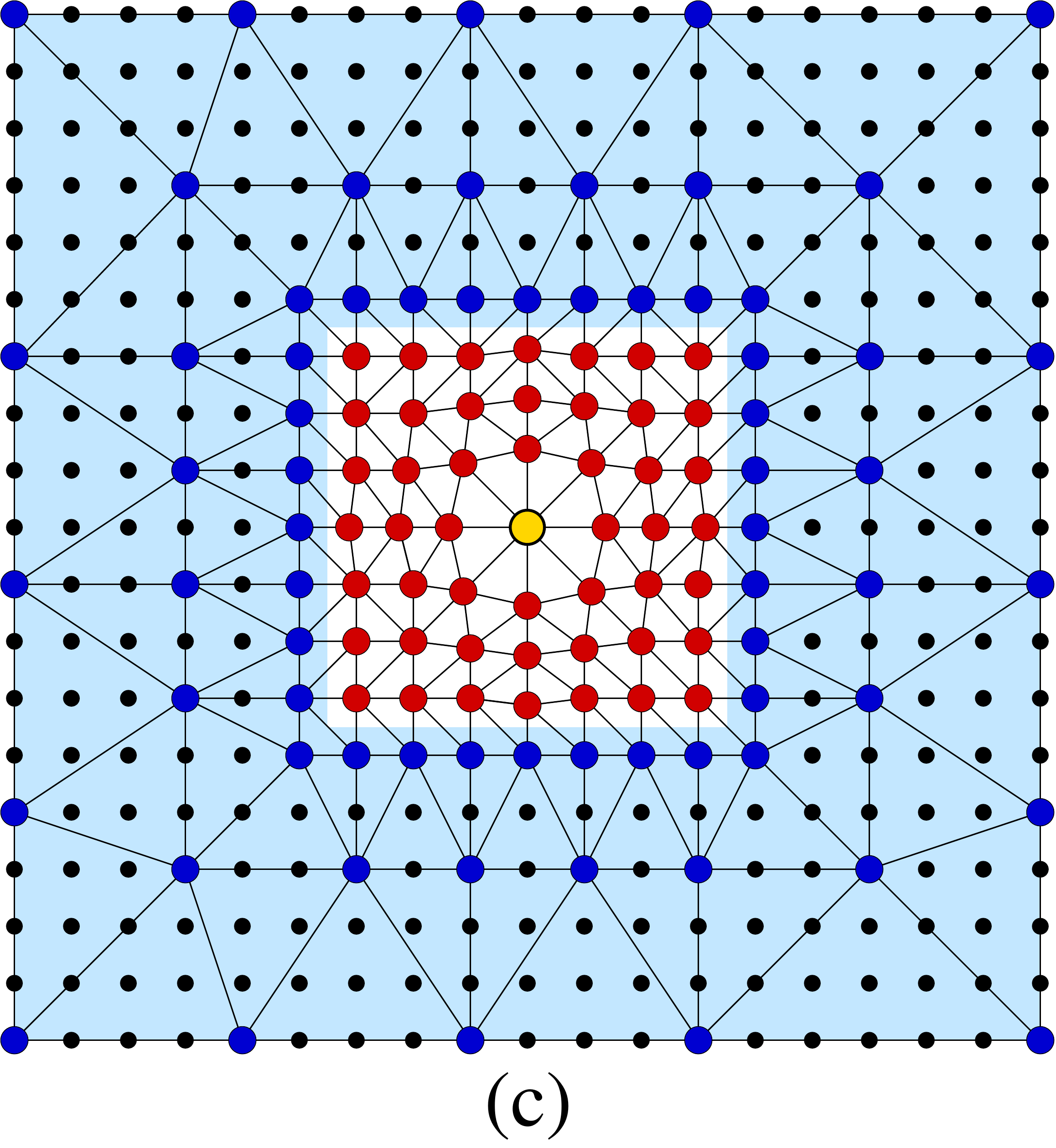} \quad
    \includegraphics[width=5cm]{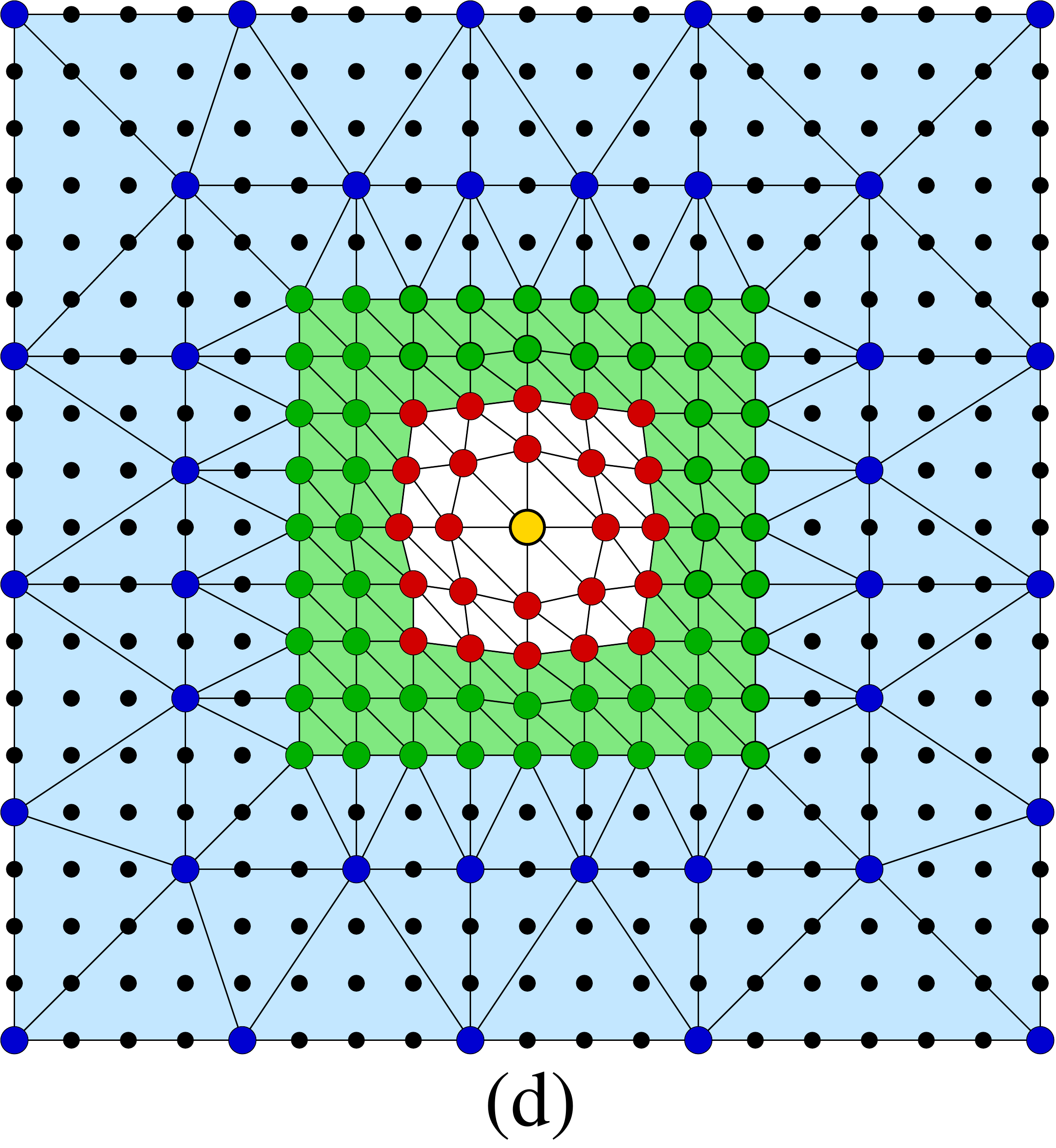}
    \caption{\label{fig:intro:meshes} 
      {\bf (a)} A 2D atomistic configuration
        with an impurity that causes a large local deformation from the
        reference lattice $\eps \Z^2$. \qquad
        {\bf (b)} Triangulation of the deformed atomistic configuration to
        visualize the Galerkin projection described in
        \S\ref{sec:intro:galerkin}. The positions of the large red
        atoms are free, while the positions of the small black atoms are
        constraint by the motion of the free atoms.  \qquad
        {\bf (c)} Visualization of the QCE method described in
        \S\ref{sec:intro:qce}. The blue shaded region is the set
        $\Omcqce$, the red atoms inside the white region are the set
        $\Laqce$ (both after deformation).   \qquad
        {\bf (d)} Visualization of an interface correction as
        described in \S\ref{sec:intro:int_corr}. The blue shaded
        regions is the continuum region $\Omc$, the green shaded
        region is the interface region $\Omi$, and the white region is
        the atomistic region $\Oma$. The mesh in $\Omi \cup \Oma$ is
        chosen so that it coincides with $\Teps$
        (cf. \S\ref{sec:fram:disc_cts}).  }
  \end{center}
\end{figure}

Let $\Om := (-1, 1]^d$, and let $\Th$ be a regular \cite{Ciarlet:1978}
triangulation of $\bar\Om$, with vertices belonging to $\Lext$, that
can be extended periodically to a regular triangulation $\Thext$ of
$\R^d$. We make the convention that elements $T \in \Th$ are closed
sets. For $T \in \Th$ we define $h_T := \diam(T)$, and for each $x \in
\R^d$ we define $h(x) := \max \{ h_T \sep T \in \Th \text{ s.t. } x \in
T \}$.

We define the $\PI$ finite element space
\begin{displaymath}
  \Sh := \b\{ v_h : \Lext \to \R  \bsep  v_h \text{ is piecewise affine w.r.t. }
  \Thext \b\},
\end{displaymath}
and we denote the spaces of piecewise affine displacements and
deformations, respectively, by
\begin{displaymath}
  \Ush := \Us \cap \Sh^d, \quad \Ysh := \Ys \cap \Sh^d,
  \quad \text{and} \quad
  \YshA := \Ys_\mA \cap \Sh^d.
\end{displaymath}

For future reference, let $I_h : \Ys \to \Ysh$ denote the nodal
interpolation operator. We note that $I_h : \YsA \to \YshA$ as well as
$I_h : \Us \to \Ush$. We also define $\PIper(\Th)$ to be the set of
all $2\Z^d$-periodic functions $v_h \in \PI(\Thext)$; i.e., $\Ush =
\PIper(\Th)^d$. Let $\Edghext$ denote the set of closed edges of the
extended triangulation $\Thext$, and let $\Edgh$ denote the set of all
edges $f \in \Edghext$ such that ${\rm area}(f \cap \Om) \neq
0$. Finally, we introduce the spaces of piecewise constant functions
$\PO(\Th)$, $\PO(\Thext)$, and $\POper(\Th)$, defined in a similar
manner.

The {\em Galerkin approximation} of \eqref{eq:intro:minEa} is the
coarse-grained minimization problem
\begin{equation}
  \label{eq:intro:min_Gal}
  y_{\a, h} \in \underset{y_h \in \YshA}{\argmin} \, \Eatot(y_h),
\end{equation}
where we recall that $\Eatot = \Ea + \Pa$.

For example, if we choose the triangulation $\Th$ as in Figure
\ref{fig:intro:meshes}(b), then we obtain full atomistic resolution in
the neighbourhood of the defect, while considerably reducing the
overall number of degrees of freedom by coarsening the mesh away from
the defect. In this way it is possible to obtain highly accurate
approximations to nontrivial atomistic configurations. Under suitable
technical assumptions it is not too difficult to employ the classical
techniques of finite element error analysis in this context. Such
analyses, including {\it a posteriori} grid generation, are given in
\cite{LinP:2003a, LinP:2006a, OrtnerSuli:2008a}.

\begin{remark}[Regularity of Atomistic Solutions]
  In order for the Galerkin projection method, or the subsequent a/c
  coupling methods, to be accurate we require ``regularity'' of
  atomistic solutions. Such a regularity theory does not exist at
  present, however, most numerical experiments performed on atomistic
  models for {\em simple lattices} indicate ``smoothness'' of
  atomistic deformations away from defects. The situation would be
  different for so-called {\em multi-lattices}, which require a
  homogenisation step and represent a far greater challenge.
\end{remark}

\paragraph{Continuum region \& Cauchy--Born approximation}
\label{sec:intro:qce}
The Galerkin projection \eqref{eq:intro:min_Gal} reduces the number of
degrees of freedom considerably, however, the complexity of computing
$\Ea|_{\Ysh}$ is not reduced in the same manner. Due to the
non-locality of the atomistic interaction $\Ea|_{\Ysh}$ cannot be
evaluated as easily as in the case of finite element methods for
continuum mechanics. Several attempts have been made to use quadrature
ideas to approximate $\Ea$ and render \eqref{eq:intro:min_Gal}
computationally efficient
\cite{Eidel:2008a,Gunzburger:2008a,Knap:2001a}, however, it was shown
in \cite{Luskin:clusterqc} that these approximations yield
unacceptable consistency errors.

An alternative approach, proposed in \cite{Ortiz:1995a}, is to keep
the full atomistic description for atomistically fine elements, while
employing the {\em Cauchy--Born approximation} for coarser elements as
well as an interface region. Following the terminology of
\cite{Dobson:2008b} we call the resulting method the {\em QCE method}
(the original energy-based quasicontinuum method).

To formulate this method we choose a set $\Laqce \subset \L$ of atoms
that we wish to treat atomistically (the red atoms in Figure
\ref{fig:intro:meshes}(c)). Let $Q_\eps(x) := x + \eps (-\smfrac12,
\smfrac12]^d$, and define
\begin{displaymath}
  \text{and} \quad
  \Omcqce := \Om \setminus \bigcup_{x \in \Laqce} Q_\eps(x).
\end{displaymath}
With this notation, the QCE energy functional is defined, for $y_h
\in \Ysh$, as
\begin{equation}
  \label{eq:intro:E_qce}
  \Eqce(y_h) := \eps^d \sum_{x \in \Laqce} V\b(\Da{\Rg} y_h(x)\b)
  + \int_{\Omcqce} W(\D y_h) \dx,
\end{equation}
where $W : \R^{d \times d} \to \R$ is the {\em Cauchy--Born stored
  energy function},
\begin{equation}
  \label{eq:intro:defnW}
  W(\mF) := V(\mF \Rg) = V\b( (\mF r)_{r \in \Rg} \b).
\end{equation}
Note that $W(\mF)$ is the energy of a single atom in the Bravais
lattice $\mF \Lext$.

One may readily check that the complexity of evaluating $\Eqce$ (or
its derivatives) is of the order $O(\#\Th)$, that is, of the same
order as the number of degrees of freedom. Moreover, it is easy to see
that $\Eqce(\yA) = \Ea(\yA)$ for all $\mA \in \R^{d \times d}$. For
future reference, we give a formal definition of this property.

\begin{definition}[Global Energy Consistency]
  We say that an energy functional $\E \in \CC(\Ys)$ is {\em globally
    energy consistent} if
  \begin{equation} 
    \label{eq:intro:econs}
    \E(\yA) = \Ea(\yA) \qquad \forall \mA \in \R^{d \times d}.
  \end{equation}
\end{definition}

More generally one can show that $\Eqce(y_h)$ is a good approximation
to $\Ea(y_h)$ if $\D y_h$ varies only moderately. Despite these facts,
it turns out, as we discuss in \S\ref{sec:intro:patch_test} and
\S\ref{sec:e1d:qce}, that minimizers of $\Eqce$ are poor
approximations to minimizers of $\Ea$.

\begin{remark}[The Cauchy--Born Model]
  \label{rem:intro:cb}
  The Cauchy--Born model is a standard continuum model, for large
  deformations of single crystals. In the absence of defects,
  solutions of a pure Cauchy--Born model (no atomistic region) can
  provide excellent approximations to solutions of the atomistic model
  \eqref{eq:intro:minEa}. For example, in \cite{E:2007a} it is shown
  that, for smooth and small dead load external forces, and under
  realistic stability assumptions on $\Ea$, there exist solutions
  $y_\a$ of \eqref{eq:intro:minEa} and $y_\c$ of the Cauchy--Born
  model, such that
  \begin{displaymath}
    \bg(\eps^d \sum_{x \in \L} \sum_{j = 1}^d \b| \Da{e_j}y_\a(x) -
    \Da{e_j} {y}_\c(x)\b|^2 \bg)^{1/2} \leq C \eps^2,
  \end{displaymath}
  where $C$ depends on higher partial derivatives of $V$ and on the
  regularity of ${y}_\c$. Hence, in the regime of ``smooth elastic''
  deformations, the Cauchy--Born model can be considered an excellent
  approximation to the atomistic model \eqref{eq:intro:Ea}.
\end{remark}

\paragraph{The patch test}
\label{sec:intro:patch_test}
The {\em patch test} is often employed in the theory of finite element
methods \cite{StrangFix2008,Belytschko:2000a,Irons:1966a} as a simple
test for consistency. The test also plays an important role in the
design of a/c methods.

\begin{definition}[Patch Test Consistency]
  We say that an energy functional $\E \in \CC^1(\Ysh; \R)$ is {\em
    patch test consistent} if it satisfies
  \begin{equation}
    \label{eq:intro:patchtest}
    \b\< \del\E(\yF), u_h \b\> = 0 \qquad \forall\, u_h \in \Ush,
    \quad \forall\, \mF \in \R^{d \times d}.
  \end{equation}
\end{definition}

The terminology ``patch test consistency'' is motivated by Proposition
\ref{th:intro:a_patchtest}, where we have shown that the exact energy
$\Ea$ does satisfy the patch test \eqref{eq:intro:patchtest}.

However, the QCE energy $\Eqce$ is {\em not} patch test consistent
\cite{Shenoy:1999a}. This result will be reviewed for a
one-dimensional model problem in \S\ref{sec:e1d:qce}, where it will
also be shown how failure of the patch test affects the consistency
error.

\begin{remark}
  \label{rem:intro:pt_frc}
  In most of the a/c coupling literature the patch test is stated as
  the condition that
  \begin{displaymath}
    \frac{\pp \E(y_h)}{\pp y_h(p)}\B|_{y_h = \yF} = 0
    \quad \text{ for all finite element nodes } p.
  \end{displaymath}
  It is straightforward to see that this condition is equivalent to
  the variational formulation given in \eqref{eq:intro:patchtest}.
\end{remark}

\paragraph{Interface correction}
\label{sec:intro:int_corr}
In the engineering literature (see, e.g., \cite{Shenoy:1999a,
  Miller:2003a}) the non-zero forces under homogeneous strains of
patch test inconsistent a/c energies are usually dubbed ``ghost
forces''. The discovery that $\Eqce$ is not patch test consistent has
resulted in a number of works constructing new a/c methods that
removed or reduced the ``ghost forces'' \cite{badi08, E:2006,
  Shapeev:2010a, Shimokawa:2004, XiBe:2004, KlZi:2006, XHLi:genqnl1d,
  BvK:blend1d, IyGa:2011}. In some cases, this is achieved through
sacrificing a variational (i.e., conservative, or, energy-based)
formulation \cite{Dobson:2008a, Fish:2007, Kohlhoff:1989, Makr:2010a,
  Parks:2008a, Shenoy:1999a, cadd}.

In the present paper we will focus only on energy-based a/c methods
that are patch test consistent, i.e., that remove the ``ghost forces''
altogether. None of the methods presently available in the literature
have resolved this problem in its full generality, however,
\cite{E:2006, Shapeev:2010a, Shimokawa:2004, IyGa:2011} present
several interesting approaches and partial solutions. In the
following, we present a generalization of the {\em geometrically
  consistent coupling method} \cite{E:2006}, which is the most general
approach but leaves some questions concerning its practical
construction open.

Let $\Thc \subset \Th$ and $\Omc = \cup \Thc$. We assume that all
atoms belonging to $\L \setminus {\rm int}(\Omc)$ are vertices of
$\Th$, and we define two sets $\Lagc, \Ligc$ such that
\begin{displaymath}
  \Lagc \cup \Ligc = \L \setminus {\rm int}(\Omc), \quad \text{and} \quad
  \Lagc \cap \Ligc = \emptyset.
\end{displaymath}
For each $x \in \Ligc$, we define a modified interaction potential
$\modV(x; \cdot) \in \CC^2( (\R^d)^\Rg )$, and we define the a/c
energy functional as
\begin{equation}
  \label{eq:intro:Eac_v1}
  \E_\gc(y_h) := \eps^d \sum_{x \in \Lagc} V\b(\Da{\Rg} y(x)\b)
  + \eps^d \sum_{x \in \Ligc} \modV\b(x; \Da{\Rg} y(x) \b)
  + \int_{\Omc} W(\D y_h) \dx.
\end{equation}
For the modified potential $\modV$ one takes a general ansatz with
several free parameters, which are then fitted to remove or minimize
the ghost force. For example, following the ideas of the quasinonlocal
coupling method \cite{Shimokawa:2004} and the geometrically consistent
coupling method \cite{E:2006} one may define
\begin{align*}
  \modV\b(x; {\bf g}\b) = V\b( (\tilde{g}_r)_{r \in \Rg} \b), \quad
  \text{where} \quad
  \tilde{g}_r = \sum_{s \in \Rg} C_{x, r,s} g_s.
\end{align*}
The constants $C_{x,r,s}$ can then be determined analytically as in
\cite{E:2006,OrtZha:2011a}, or, as proposed in \cite{OrtZha:2011b},
numerically in a preprocessing step. The 2D numerical experiments
performed in \cite{OrtZha:2011b} suggest that it is always possible to
determine parameters $C_{x,r,s}$ such that $\Eac$ becomes patch test
consistent, however, a proof of this fact is still missing.

The purpose of the present work is to investigate the question whether
patch test consistency is in fact a sufficient condition for
first-order consistency of an a/c coupling method. If this would turn
out to be false in general, then it would be necessary to develop new
approaches for constructing accurate a/c methods.

\paragraph{General assumptions on the interface correction. }
\label{sec:intro:gen_int_corr}
For the subsequent analysis we assume an even more general form of the
a/c functional than \eqref{eq:intro:Eac_v1}. We choose $\Thc, \Thi,
\Tha \subset \Th$, mutually disjoint, such that $\Th = \Thc \cup \Thi
\cup \Tha$, and we define the continuum, interface, and atomistic
regions
\begin{displaymath}
  \Omc := \cup \Thc, \quad \Omi := \cup \Thi, \quad \text{and}
  \quad \Oma := \cup \Tha.
\end{displaymath}
(Note that $\Omc, \Omi, \Oma$ are closed sets.) Next, we define the
set of all nodes $\La \subset \L$ that interact with the atomistic
region:
\begin{displaymath}
  \La := \B\{ x \in \L \Bsep  (x, x+\eps r) \cap \Oma^\per \neq
  \emptyset \text{ for some } r \in \Rg \B\},
\end{displaymath}
where the ordered pair $(x, x') \in \Lext \times \Lext$ is called a
{\em bond}; here, and throughout, the symbol $(x, x')$ is also
identified with the segment ${\rm conv}\{ x, x' \}$ (in particular, it
is closed). To avoid interaction between $\La$ and $\Omc$, we assume
throughout that 
\begin{equation}
  \label{eq:intro:noint_La_Omc}
  \b\{ x + t r \bsep x \in \La, t \in [0, 1], r \in \Rg \b\} \subset
  \Oma \cup \Omi.
\end{equation}

Next, we define the set of interface bonds
\begin{displaymath}
  \Bi := \B\{ b = (x, x+\eps r) \Bsep x \in \L, r \in \Rg, (x, x+\eps r)
  \subset \Omi^\per \B\}.
\end{displaymath}
Finally, we define an interface functional $\Ei \in \CC^2(\Ys)$ such
that
\begin{equation}
  \label{eq:intro:defn_Fi_1}
  \Ei(y) = \eps^d \Fi\B( \b( \Da{r} y(x) ; (x, x+\eps r) \in \Bi \b) \B),
\end{equation}
that is, the interface functional $\Ei$ is given as a function of the
finite differences $\Da{r} y(x)$ of bonds $(x, x+\eps r)$ that are
contained in the interface region. Note also the volumetric scaling
$\eps^d$.

With this notation, we set
\begin{equation}
  \label{eq:intro:Eac}
  \Eac(y_h) := \eps^d \sum_{x \in \La} V\b(\Da{\Rg} y_h(x)\b)
  + \int_{\Omc} W(\D y_h) \dx + \Ei(y_h).
\end{equation}

The interface functional $\Ei(y_h)$ specifies the different variants
of a/c methods. It is easy to see that the functionals $\Eqce$ and
$\E_{\rm gc}$, discussed above, fit this framework (in the case of
$\Eqce$ we have to drop the assumption \eqref{eq:intro:noint_La_Omc}).

If we define the total a/c energy as $\Eatot := \Eac + \Pac$, where
$\Pac$ is a suitable a/c approximation to $\Pa$ , the a/c
approximation to \eqref{eq:intro:minEa} is
\begin{equation}
  \label{eq:intro:min_ac}
  y_\ac \in \underset{y_h \in \YshA}{\rm argmin} \, \Eactot(y_h).
\end{equation}
If $y_\ac$ solves \eqref{eq:intro:min_ac}, then it is a critical point
of $\Eactot = \Eac + \Pac$:
\begin{equation}
  \label{eq:intro:crit_ac}
  \b\< \del\Eac(y_\ac) + \del\Pac(y_\ac), u_h \b\> = 0
  \qquad \forall u_h \in \Ush.
\end{equation}

\paragraph{The locality and scaling conditions}
\label{sec:intro:loc_scal}
We define notation for first and second partial derivatives of $\Fi$
as follows: for ${\bf g} = (g_b)_{b \in \Bi}$ let
\begin{displaymath}
  \pp_b \Fi\b( (g_b)_{b \in \Bi} \b) := \frac{\pp \Fi\b(
    (g_b)_{b \in \Bi} \b)}{\pp g_b}, \quad \text{and} \quad
  \pp_b \pp_{b'} \Fi\b( (g_b)_{b \in \Bi} \b) := \frac{\pp^2 \Fi\b(
    (g_b)_{b \in \Bi} \b)}{\pp g_b \pp g_{b'}}.  
\end{displaymath}
We extend the definition periodically: if $b \in \Bi$ and $\xi \in
2\Z^d$ then $\pp_{\xi+b} \Fi := \pp_{b} \Fi$, and we make a similar
definition for the second partial derivatives.
%
%
In our analysis we will require two crucial properties on $\Fi$, which
we call the {\em locality and scaling conditions}:

The {\em locality condition}
\begin{equation}
  \label{eq:intro:Fi_loc}
   \pp_{(x,x+\eps r)} \pp_{(x', x'+\eps s)} \Fi(y) = 0 \qquad
  \begin{array}{l}
    \text{for all bonds } (x, x+\eps r), (x', x' + \eps s) \in \Bi \\
    \text{such that } x \neq x',
  \end{array}
\end{equation}
implies that the same bonds interact through $\Ei$ as in the atomistic
model. This condition can be weakened, by requiring that only bonds
within an $O(\eps)$ distance interact, however, such a more general
condition would add additional notational complexity.

In the {\em scaling condition} we assume that there exist constants
$M_{r,s}^\i \geq 0$, $r, s \in \Rg$, such that
\begin{equation}
  \label{eq:intro:Fi_scal}
   \b\| \pp_{(x, x+\eps r)} \pp_{(x, x+\eps s)} \Fi(y) \b\| \leq~
   \cases{
     M_{r,s}^\i, &
     \quad \forall (x, x+\eps r), (x, x+\eps s) \in \Bi, \\[2mm]
     \smfrac12 M_{r,s}^\i, &
     \quad \text{if } {\rm length}\b( \pp\Omi^\per \cap (x, x+\eps r)\b) >
     0.
   }
\end{equation}
This condition effectively yields an $O(1)$ Lipschitz bound for
$\del\Ei$ in the function spaces we will use. The scaling aspect
enters through an implicit assumption on the magnitude of the
constants $M_{r,s}^\i$, namely, we will assume throughout that the
constant
\begin{equation}
  \label{eq:intro:defn_Mi}
  M^\i := \sum_{r \in \Rg} \sum_{s \in \Rg} |r| |s| M_{r,s}^\i
\end{equation}
is of the same order of magnitude as the constant $M^\a$ defined in
\eqref{eq:intro:defn_Ma}.

\begin{remark}
  \label{rem:12_factor}
  The factor $\smfrac12$ for bonds on the boundary of the interface
  region is not strictly necessary, since it can be removed by simply
  replacing $M_{r,s}^\i$ with $2 M_{r,s}^\i$, however, if stated as
  above it makes the statements of the results in \S\ref{sec:2d}
  slightly sharper, and moreover simplifies the argument in
  \eqref{eq:2d:Rac1est_Ti_pre}.

  The necessity of this factor is related to the fact that we allow
  $\Fi$ to depend on bonds that lie on the boundary of $\Omi^\per$;
  this is made clear in Proposition \ref{th:2d:Sac} where we construct
  a stress function for $\Eac$. Note that if we did not allow $\Fi$ to
  depend on these boundary bonds, then it would in fact be impossible
  to construct patch test consistent a/c methods for non-flat a/c
  interfaces.
\end{remark}


\section{A Framework for the a priori Error Analysis of a/c
  Methods}
\label{sec:fram}
When analyzing the error of a numerical method, one should first of
all determine the main quantities of interest. For a/c methods, one is
usually interested in energy differences between homogeneous lattices
and lattices with defects, or critical loads at which defects form or
move (i.e., bifurcation points). Since the present paper is mostly
theoretical, we will simply focus on the error in the deformation
gradient. We note, however, that many aspects of this analysis are
crucial ingredients for the analysis of energy differences (see, e.g.,
\cite{OrtShap:2011a}) and would usually also enter an analysis of
bifurcation points.

We assume from now on that $d \in \{1, 2\}$. To execute the abstract
framework of this section also in 3D, several technical tools as well
as the central consistency result need to be developed first.

\subsection{Discrete and continuous functions}
\label{sec:fram:disc_cts}
In the following analysis it will be important to extend the a/c
functional $\Eac$ to all functions $y \in \Ys$. To that end, we first
define piecewise affine interpolants with respect to an atomistic mesh
$\Teps$.

We take a subdivision of the scaled unit cube $\eps (0, 1)^d$ into
$d$-simplices (in 1D the interval $\eps (0, 1)$; in 2D two symmetric
triangles; compare with the triangulation of the atomistic region in
Figure \ref{fig:intro:meshes}(d)), which we extend periodically to a
triangulation $\Tepsext$ of $\R^2$ with vertex set $\Lext$. The
restriction of $\Tepsext$ to $\Om$ is denoted by $\Teps$. Each discrete
function $v : \Lext \to \R^k$ will from now on be identified in a
canonical way with its continuous piecewise affine interpolant $v \in
\PI(\Tepsext)^k$.

For future reference we denote the sets of edges of $\Teps$ and
$\Tepsext$, corresponding to the definitions of $\Edgh$ and $\Edghext$
in \S\ref{sec:intro:galerkin}, by $\Edg$ and $\Edgext$.

\paragraph{Ambiguity of continuous interpolants}
If $y_h \in \Ysh$ then $y_h$ can also be interpreted as a member of
$\Ys$ and therefore has two, possibly different, continuous
interpolants. To distinguish them, we make the convention that the
symbol $y_h$ always denotes the interpolant in $\PI(\Th)^d$, while a
symbol $y$ always denotes the interpolant in $\PI(\Teps)^d$. If we
wish to evaluate the $\PI(\Teps)^d$-interpolant of a function $y_h \in
\Ysh$ then we will write $I_\eps y_h$.

To compare a $\PI(\Teps)^d$-interpolant with a
$\PI(\Th)^d$-interpolant, we use the following lemma. In 1D the result
is easy to establish; in 2D it depends on a technical tool that we
introduce in \S\ref{sec:aux:bdl}. The proof is given in the appendix.

\begin{lemma}
  \label{th:fram:Dbarvh_Dvh}
  Let $d \in \{1,2\}$; then, for all $y_h \in \Ysh$ and $p \in
  [1,\infty]$, we have 
  \begin{displaymath}
    \| \D I_\eps y_h \|_{\LL^p(\Om)} \leq \| \D y_h \|_{\LL^p(\Om)},
  \end{displaymath}
  where we recall that we have defined $\| \D v \|_{\LL^p} = \|\,|\D
  v|_p \|_{\LL^p}$.
\end{lemma}

\paragraph{Extension of the a/c energy}
Before we extend the a/c energy $\Eac$ we make one last technical
assumption, which considerably simplifies the subsequent analysis. We
shall assume from now on, that
\begin{displaymath}
  \Tha \cup \Thi \subset \Teps.
\end{displaymath}
If all atoms in $\L \cap (\Oma \cup \Omi)$ are vertices of
$\Tha\cup\Thi$, which is not uncommon, then this is no restriction.

With these conventions the a/c energy $\Eac$ defined in
\eqref{eq:intro:Eac} can be defined canonically for functions $\tily =
y + u_h \in \Ys + \Ush$ by the same formula:
\begin{displaymath}
  \Eac(\tily) = \int_{\Omc} W(\D \tily) \dx 
  + \eps^d \sum_{x \in \La} V\b(\Da{\Rg} \tily(x)\b) + \Ei(\tily)
  \qquad \text{for } \tily \in \Ys + \Ush.
\end{displaymath}
We also assume that $\Pac$ has a suitable extension to $\Ys$. It
should be stressed that for general $y_h \in \Ysh$, $\Eac(y_h) \neq
\Eac(I_\eps y_h)$.

\subsection{Measuring smoothness; An interpolation error estimate}
\label{sec:fram:smoothness_interperr}
The three main ingredients in the a priori error analysis of
Galerkin-like approximations are (i) consistency, (ii) stability, and
(iii) an interpolation error estimate. We begin by establishing the
latter. To that end, we first need to find a convenient {\em measure
  of smoothness} for discrete functions $y \in \Ys$.

\paragraph{Measuring smoothness in terms of local oscillation}
\label{sec:fram:osc}
There are several possibilities to measure the ``smoothness'' of a
discrete function. The most obvious is possibly the use of higher
order finite differences, e.g., $\Da{e_i}\Da{e_j} y(x)$. If, in
\S\ref{sec:fram:disc_cts}, we had chosen continuous interpolants
belonging to $\WW^{2,\infty}$, then we would be able to simply
evaluate the second derivatives $\D^2 y$. However, since the
interpolants we use are piecewise affine, the second derivative of $y$
is the measure $\jl\D y\jr \otimes \nu \ds\b|_{\Edgext}$, where $\jl\D
y\jr$ denotes the jump of $\D y$ across an element edge, and $\ds$ the
surface measure.

This last observation motivates the idea to measure smoothness of $y$
by the local oscillation of $\D y$. We define the {\em oscillation
  operator}, for measurable sets $\omega \subset \R^d$, and for $y \in
\Ys$, as
\begin{equation}
  \label{eq:fram:defn_osc}
  {\rm osc}(\D y; \omega) := \underset{x, x' \in \omega}{\rm
    ess~sup}\, \frac{\b| \D y(x) - \D y(x') \b|}{\eps}.
\end{equation}

The sets $\omega$ that arise naturally in our analysis will always
have $O(\eps)$ diameter, which is the reason for the
$\eps^{-1}$-scaling in the definition of ${\rm osc}$.

Note, in particular, that if $y$ were twice differentiable, and if
${\rm diam}(\omega) \leq C \eps$, then we would obtain
\begin{displaymath}
   {\rm osc}(\D y; \omega) \leq \frac{{\rm diam}(\omega)}{\eps} \,
   \| \D^2 y \|_{\LL^\infty(\omega)} 
   \leq C \| \D^2 y \|_{\LL^\infty(\omega)},
\end{displaymath}
which further illustrates that the oscillation operator is a
reasonable replacement for $\D^2 y$ to measure the local smoothness of
a piecewise affine function.

\paragraph{Interpolation error estimate}
\label{sec:fram:interp}
The smoothness measure we defined in \S\ref{sec:fram:osc} yields a
simple proof of an interpolation error estimate; see Appendix
\ref{sec:app_proofs}.

\begin{lemma}
  \label{th:fram:interp_err}
  Let $d \in \{1,2\}$ and suppose that $\Tha \cup \Thi \subset \Teps$;
  then there exists a constant $C_I$, which depends only on the shape
  regularity of $\Th$, such that, for all $y \in \Ys$, $p \in [1,
  \infty)$,
  \begin{displaymath}
    \b\| \D (y - I_h y) \b\|_{\LL^p(\Om)} \leq C_I \bg\{
    \sum_{T \in \Tepsc} |T|
    \B[ {\rm h}_T \, {\rm osc}(\D y; \omega_T^\c) \B]^p \bg\}^{1/p},
  \end{displaymath}
  where, for $T \in \Teps$,
  \begin{equation}
    \label{eq:fram:defn_omTc_hTeps}
    \omega_T^\c := \Omc^\per \cap \bigcup \b\{T' \in \Teps^\per \bsep T \cap T' \neq
    \emptyset \b\}, \quad \text{and} \quad 
    {\rm h}_T := \max_{x \in T} |h(x)|.
  \end{equation}
  Similarly, for $p = \infty$, we have 
  $\b\| \D (y - I_h y) \b\|_{\LL^\infty(\Om)} \leq C_I \max_{T \in
      \Tepsc} 
    \b[ {\rm h}_T \, {\rm osc}(\D y; \omega_T^\c) \b]$.    
\end{lemma}

\subsection{The stability assumption}
\label{sec:fram:stab}
The stability of a/c methods relies, firstly, on the stability of
atomistic models as well as their Cauchy--Born approximations. It
requires a thorough understanding of the physics of a model and in
particular more specific information about the interaction
potential. Since the focus of the present work is the consistency of
a/c methods, we will formulate stability as an {\em assumption}.

The simplest notion of stability one may use, which is also closely
connected to local minimality, is coercivity of the second variation:
\begin{equation}
  \label{eq:fram:stab_W12}
    \< \ddel \Eactot(y_h) u_h, u_h \> 
    \geq c_0 \| \D u_h \|_{\LL^2(\Om)}^2 \qquad \forall\, u_h \in \Ush,
\end{equation}
where $c_0 > 0$, and $y_h \in \Ysh$ is a suitable deformation in a
neighbourhood of the atomistic solution $y_\a$, e.g., $y_h = I_h
y_\a$. The choice of norm is motivated by the fact that the
Cauchy--Born model, and hence the atomistic model, are closely related
to second order elliptic differential equations.


Examples of sharp stability estimates for a/c methods in 1D can be
found in \cite{Dobson:qce.stab, Ortner:qnl.1d, XHLi:genqnl1d,
  XHLi:eam1d, BvK:blend1d}. For pair interactions in 2D the stability
of Shapeev's method \cite{Shapeev:2010a} is established in
\cite{OrtShap:2011a}.

More generally, for some $p \in [1, \infty], p' = p / (p-1)$, we may
assume an inf-sup condition of the form
\begin{equation}
  \label{eq:fram:stab_W1p}
  \inf_{\substack{u_h \in \Ush \\ \| \D u_h \|_{\LL^{p}(\Om)} = 1}}
  \sup_{\substack{v_h \in \Ush \\ \| \D v_h \|_{\LL^{p'}(\Om)} = 1}}
  \b\< \ddel \Eactot(y_h) u_h, v_h \b\> \geq c_0,
\end{equation}
for some constant $c_0 > 0$. The condition \eqref{eq:fram:stab_W1p} is
usually difficult to prove, especially for $p \neq 2$, and may indeed
be false in general. We will only use it to demonstrate how such a
stability result motivates consistency estimates in different negative
norms. Examples of 1D inf-sup stability estimates for a/c methods can
be found in \cite{Dobson:2008a, Dobson:arXiv0903.0610,
  MakrOrtSul:qcf.nonlin, OrtnerSuli:2008a}.

\subsection{Outline of an a priori error analysis}
\label{sec:fram:err}
The following outline of an a priori error analysis depends on a
stability assumption that we will not prove. Moreover, since it
primarily serves to motivate the {\em consistency problem}, and since
a rigorous derivation would be more involved without yielding much
additional insight, some steps will be kept vague. Most of these steps
are easily made rigorous; the main assumption we make below, which is
in fact very difficult to justify rigorously, is that $I_h y_\a$ and
$y_\ac$ are ``sufficiently close''. See \cite{MakrOrtSul:qcf.nonlin,
  Ortner:qnl.1d, OrtnerSuli:2008a, BvK:blend1d} for similar analyses
in 1D where all steps are rigorously justified, and
\cite{OrtnerWang:2009a,OrtShap:2011a} for a similar semi-rigorous
framework, where a proof of this step is replaced by an assumption.

%

Let $y_\a$ satisfy \eqref{eq:intro:Ea_crit}, $y_\ac$ satisfy
\eqref{eq:intro:crit_ac}, and suppose that the stability assumption
\eqref{eq:fram:stab_W1p} holds with $y_h = I_h y_\a$. Let $e_h :=
y_\ac - I_h y_\a$. Moreover, suppose that $\| \D I_h y_\a - \D y_\ac
\|_{\LL^\infty}$ is sufficiently small so that the following
approximation can be made precise:
\begin{align*}
  \b< \ddel\Eactot(I_h y_\a) e_h, v_h \b\>
  \approx~& \int_0^1 \b\< \ddel \Eactot(I_h y_\a + t e_h) e_h, v_h \b\>
  \dt \\
  =~& \b\< \del(\Eac+\Pac)(y_\ac) - \del(\Eac+\Pac)(I_h y_\a), v_h \b\>.
\end{align*}
Taking the supremum over all $v_h \in \Ush$, and invoking the inf-sup
condition \eqref{eq:fram:stab_W1p} and the criticality condition
\eqref{eq:intro:crit_ac}, we obtain
\begin{displaymath}
  c_0 \b\| \D e_h \b\|_{\LL^p(\Om)} \lesssim \b\| \del(\Eac+\Pac)(I_h
  y_\a) \b\|_{\WW^{-1,p}_h},
\end{displaymath}
where we define
\begin{displaymath}
  \| \Phi \|_{\WW^{-1,p}_h} := \sup_{\substack{v_h \in \Ush \\ \| \D
      v_h \|_{\LL^{p'}} = 1}} \b\< \Phi, v_h \b\>,
  \qquad \text{for } \Phi \in \Ush^*.
\end{displaymath}

We split the consistency error $\| \del(\Eac+\Pac)(I_h y_\a)
\|_{\WW^{-1,p}_h}$ into three separate contributions:
\begin{align}
  \notag
  c_0 \b\| \D e_h \b\|_{\LL^p(\Om)} \lesssim~& 
  \b\| \del(\Eac+\Pac)(I_h y_\a) - \del(\Eac+\Pac)(y_\a)
  \b\|_{\WW^{-1,p}_h}  \\
  \notag
  & + \b\| \del\Eac(y_\a) - \del\Ea(y_\a) \b\|_{\WW^{-1,p}_h}
  + \b\| \del\Pac(y_\a) - \del\Pa(y_\a) \b\|_{\WW^{-1,p}_h} \\
  \label{eq:fram:splitting}
  =:~& \Ecoarse + \Emodel + \Eext.
\end{align}
where we have used \eqref{eq:intro:Ea_crit}, and the extension of
$\Eac$ and $\Pac$ for all deformations $y \in \Ys$ constructed in
\S\ref{sec:fram:disc_cts}.

The {\em coarsening error}, $\Ecoarse$, can be bounded by Lipschitz
estimates for $\del(\Eac+\Pac)$ and an interpolation error
estimate. Using our assumption that $\Tha \cup \Thi \subset \Teps$ it
is not difficult to derive Lipschitz estimates of the form
\begin{equation}
  \label{eq:fram:Ecoarse_1}
  \Ecoarse \leq (M^\a + M_{\Pac}) \b\| \D I_h y_\a - \D y_\a \b\|_{\LL^p(\Om)},
\end{equation}
where $M^\a$ is a Lipschitz constant for $\pp W$
(cf. \eqref{eq:intro:defnW}), and $M_{\Pac}$ is a Lipschitz constant
for $\del\Pac$.

Combining \eqref{eq:fram:splitting}, and \eqref{eq:fram:Ecoarse_1},
the inequality
\begin{displaymath}
  \b\| \D y_\a - \D y_\ac \b\|_{\LL^p} \leq \b\| \D y_\a - \D I_h y_\a
  \b\|_{\LL^p} + \b\| \D e_h \b\|_{\LL^p},
\end{displaymath}
and the interpolation error estimate of Lemma
\ref{th:fram:interp_err}, we arrive at the following basic error
estimate
\begin{equation}
  \label{eq:fram:errest}
  \b\| \D y_\a - \D y_\ac \b\|_{\LL^p(\Om)}
  \leq \frac{\Emodel + \Eext}{c_0}
  + \frac{c_1}{c_0}   \bg\{
    \sum_{T \in \Tepsc}
    \B[ {\rm h}_T \, {\rm osc}(\D y_\a; \omega_T^\c) \B]^p \bg\}^{1/p},
\end{equation}
where $c_1 = C_I (c_0 + M^\a + M_{\Pac})$.

The {\em consistency error for the external forces}, $\Eext$, depends
on the form of $\Pa$ and $\Pac$ and cannot be discussed at this level
of abstraction. The {\em modelling error}, $\Emodel$, is the focus of the
remainder of the present paper. 

\begin{remark}[Choice of Splitting]
  Suppose, for simpliciy, that $\Pac = \Pa = 0$. In a typical finite
  element error analysis of continuum mechanics problems one would
  usually choose a different splitting of the consistency error:
  \begin{align*}
    \b\| \del\Eac(I_h y_\a) \b\|_{\WW^{-1,p}_h} \leq
    \b\| \del\Eac(I_h y_\a) - \del\Ea(I_h y_\a) \b\|_{\WW^{-1,p}_h}
    + \b\| \del\Ea(I_h y_\a) - \del\Ea(y_\a) \b\|_{\WW^{-1,p}_h}.
  \end{align*}
  This splitting was used in the analysis in \cite{OrtShap:2011a} and
  led to a suboptimal estimate of the modelling error, since it still
  contains some coarsening error.
\end{remark}

\subsection{The consistency problem}
\label{sec:fram:cons}
The main step that remains in obtaining an {\it a priori} error
estimate from \eqref{eq:fram:errest} is the estimation of the {\em
  modelling error}
\begin{displaymath}
  \Emodel = \b\| \del\Eac(y_\a) - \del\Ea(y_\a) \b\|_{\WW^{-1,p}_h}
  = \sup_{\substack{v_h \in \Ush \setminus \{0\}}}
  \frac{\b\< \del\Eac(y_\a) - \del\Ea(y_\a), v_h \b\>}{\| \D v_h \|_{\LL^{p'}(\Om)}}.
\end{displaymath}

Most of the numerical analysis literature on a/c methods estimates
this modelling error only for the case when $\Th = \Teps$. In 1D it is
easy to see that this is sufficient, since $I_\eps v_h = v_h$ in that
case; see also \cite{OrtnerWang:2009a}. The following lemma provides
the main technical step to explain why it is also sufficient in 2D to
consider the case $\Th = \Teps$. Its proof uses arguments similar to
those in the a posteriori error analysis of continuum finite element
methods and is given in Appendix \ref{sec:app_proofs}.

\begin{lemma}
  \label{th:fram:cons_reduceEmodel}
  Assume that $\Thi \cup \Tha \subset \Teps$.  Let $\Phi \in \Us^*$
  and $\Phi_h \in \Us_h^*$ be given in the form
  \begin{align*}
    \b\< \Phi, u \b\> = \int_{\Omc} \sigma : \D u \dx, 
    \quad \text{and} \quad
    \< \Phi_h, u_h \b\> = \int_{\Omc} \sigma : \D u_h \dx,
  \end{align*}
  for all $u \in \Us$, $u_h \in \Ush$, where $\sigma \in
  \POper(\Teps)^{d \times d}$; then there exists a universal constant
  $C_M$ such that, for all $p \in [1, \infty)$,
  \begin{equation}
    \label{eq:fram:cons_reduceEmodel}
    \begin{split}
      \B| \b\< \Phi, I_\eps u_h \b\> - \b\< \Phi_h, u_h \b\> \B|
      \leq~& C_M \eps \, \bg( \sum_{T \in \Tepsc} |T| {\rm osc}(\sigma;
      \omega_T^\c)^p \bg)^{1/p} \b\| \D u_h \b\|_{\LL^{p'}(\Omc)},
      \quad \text{and} \\
      \B| \b\< \Phi, I_\eps u_h \b\> - \b\< \Phi_h, u_h \b\> \B|
      \leq~& C_M \eps \, \b[\max_{T \in \Tepsc} {\rm osc}(\sigma;
      \omega_T^\c) \b] \, \b\| \D u_h \b\|_{\LL^{1}(\Omc)}.
    \end{split}
  \end{equation}
\end{lemma}

The estimate \eqref{eq:fram:cons_reduceEmodel}, together with Lemma
\ref{th:fram:Dbarvh_Dvh}, implies the following theorem, where we
use the notation
\begin{equation}
  \label{eq:fram:defn_negnrmeps}
  \| \Phi \|_{\WW^{-1,p}_\eps} := 
  \sup_{\substack{v \in \Us \setminus \{0\} \\ \| \D v \|_{\LL^{p'}} =
      1}}
  \b\< \Phi, v \b\> \qquad \text{for } \Phi \in \Us^*.
\end{equation}

\begin{theorem}
  \label{th:fram:cons_reduceEmodel_2}
  Suppose that $\Tha \cup \Thi \subset \Teps$ and that $y \in \Ys$;
  then, for all $p \in [1, \infty)$,
  \begin{equation}
    \label{eq:fram:cons_reduceEmodel_2}
    \begin{split}
      \b\| \del\Eac(y) - \del \Ea(y) \b\|_{\WW^{-1,p}_h}
      \leq~& M^\a C_M \eps \bg( \sum_{T \in \Tepsc} |T| {\rm osc}(\D y(T);
      \omega_T^\c)^p \bg)^{1/p} \\
      &+ \b\| \del\Eac(y) - \del\Ea(y) \b\|_{\WW^{-1,p}_\eps},
    \end{split}
  \end{equation}
  with corresponding statement for $p = \infty$.
\end{theorem}
\begin{proof}
  Since $\Ea(I_\eps y_h)$ uses only point values of $I_\eps y_h$,
  which are the same as for $y_h$, we have
  \begin{displaymath}
    \< \del\Ea(y), u_h
  \> = \< \del\Ea(y), I_\eps u_h \> \qquad \forall\, u_h \in \Ush.
  \end{displaymath}
  Using this fact, we can estimate
  \begin{displaymath}
    \b| \b\< \del \Eac(y) - \del\Ea(y), u_h \b\> \b| 
    \leq \b| \b\< \del\Eac(y), u_h \b\> - \b\< \del\Eac(y), I_\eps u_h \b\>
    \b| + \b| \< \del \Eac(y) - \del\Ea(y), I_\eps u_h \b\> \b|.
 \end{displaymath}
 Due to the assumption that $\Tha \cup \Thc \subset \Teps$, the first
 group can be estimated using Lemma \ref{th:fram:cons_reduceEmodel},
 with $\sigma = \pp W(\D y)$, which yields the first term in
 \eqref{eq:fram:cons_reduceEmodel_2}.

 Using Lemma \ref{th:fram:Dbarvh_Dvh}, the second group can be
 estimated by
 \begin{align*}
   \b| \< \del \Eac(y) - \del\Ea(y), I_\eps u_h \b\> \b| \leq~& \b\|
   \del \Eac(y) - \del\Ea(y) \b\|_{\WW^{-1,p}_\eps} \b\| \D I_\eps u_h
   \b\|_{\LL^{p'}} \\
   \leq~& \b\| \del \Eac(y) - \del\Ea(y)
   \b\|_{\WW^{-1,p}_\eps} \b\| \D u_h \b\|_{\LL^{p'}}.
 \end{align*}
 Taking the supremum over all $u_h \in \Ush$ with $\|\D u_h
 \|_{\LL^{p'}} = 1$ yields the stated result.
\end{proof}

\medskip \noindent Applying Theorem
\ref{th:fram:cons_reduceEmodel_2} to the modelling error $\Emodel$,
defined in \eqref{eq:fram:splitting}, we obtain that
\begin{equation}
  \label{eq:fram:Emodel_upperbnd}
  \Emodel \leq \Emodeleps +   M^\a C_M \eps \bg( \sum_{T \in \Tepsc} |T| {\rm osc}(\D y(T);
  \omega_T^\c)^p \bg)^{1/p},
\end{equation}
where
\begin{displaymath}
   \Emodeleps :=  \b\| \del\Eac(y_\a) - \del\Ea(y_\a) \b\|_{\WW^{-1,p}_\eps}.
\end{displaymath}
Even though $\Emodeleps$ is essentially an upper bound for $\Emodel$,
it is usually easier to estimate. The {\em consistency problem} is to
prove a sharp upper bound on $\Emodeleps$.

In \S\ref{sec:e1d} we will discuss two simple 1D examples to determine
what can be expected in more general situations. In Theorem
\ref{th:2d:main_result} we will prove that for an a/c method that is
patch test consistent, and satisfies various other technical
conditions, one obtains
\begin{displaymath}
  \Emodeleps \leq C \,\eps\, \bg\{ \sum_{T \in \Tepsc\cup\Tepsi}
  |T| \b[ {\rm osc}(\D y_\a; \omega_T) \b]^p \bg\}^{1/p},
\end{displaymath}
where $C$ is a constant that is independent of $y_\a$, but does depend
on the interface width, and $\omega_T \subset \Omc\cup\Omi$ is the
interaction neighbourhood defined in \eqref{eq:2d:defn_intnhd}.

Combined with \eqref{eq:fram:Emodel_upperbnd} and
\eqref{eq:fram:errest}, and using the fact that $\eps \leq {\rm h}_T$,
and that $\omega_T \supset \omega_T^\c$ for $T \in \Tepsc$, this bound
yields
\begin{equation}
  \label{eq:fram:errest_v2}
  \b\| \D y_\a - \D y_{\ac} \b\|_{\LL^p}
  \leq \frac{\Eext}{c_0} + \frac{c_2}{c_0} 
  \bg\{ \sum_{T \in \Tepsc\cup\Tepsi}
  |T| \B[ {\rm h}_T \, {\rm osc}(\D y_\a; \omega_T) \B]^p \bg\}^{1/p},
\end{equation}
where $c_2$ is a constant that is independent of $y_\a$. This estimate
closely resembles a typical first-order a priori error estimate for a
continuum mechanics finite element approximation; see also the
interpretation given in \S\ref{sec:intro:main_result}. 

It should be stressed again that \eqref{eq:fram:errest_v2} is not a
rigorous error estimate, but depends on various assumptions made in
the forgoing discussion, most prominently, the stability assumption
\eqref{eq:fram:stab_W1p}, and the assumption that $\|\D I_h y_\a - \D
y_\ac \|_{\LL^\infty}$ is ``sufficiently small''.

\begin{remark}
  The locality of the patches $\omega_T$ is crucial. If ${\rm
    diam}(\omega_T)$ is not of the order $O(\eps)$, then it is
  possible that ${\rm osc}(\D y_\a; \omega_T) \gg 1$ even if $y_\a$ is
  globally smooth; see also \S\ref{sec:2d:rem_locality}.
\end{remark}


\section{Examples in 1D}
\label{sec:e1d}
In the present section we review the consistency analyses of specific
a/c methods to point out the main features and to motivate what may be
expected in the general case. Throughout this section we assume that
$d = 1$, $\Rg = \{ \pm 1, \pm 2\}$, and that $V$ is given by
\begin{displaymath}
  V(\{g_{\pm 1}, g_{\pm 2}\}) = \smfrac12 \b[
  \phi_1(g_1) + \phi_1(g_{-1}) + \phi_2( g_2) + \phi_2(g_{-2}) \b],
\end{displaymath}
where $\phi_1, \phi_2 \in \CC^{2,1}(\R)$ are, respectively, the first
and second neighour interaction potentials, which are assumed to be
symmetric about the origin. We assume that their derivatives
$\phi_{i}'$ and $\phi_{i}''$ have global Lipschitz constants $m_i'$
and $m_i''$.

For the 1D analysis it is convenient to write $x_n = n \eps$, $v_n =
v(x_n)$, and to write all interactions in terms of the backward
difference operator
\begin{displaymath}
  v'_n = \frac{v_n - v_{n-1}}{\eps}.
\end{displaymath}
With this notation the atomistic energy can now be rewritten in the
form
\begin{equation}
  \label{eq:e1d:Ea_bonds}
  \Ea(y) = \eps \sum_{n = -N+1}^N \phi_1(y_n') + \eps \sum_{n = -N+1}^N
  \phi_2(y_n' + y_{n+1}'),
\end{equation}
where we note that $y_n' + y_{n+1}' = \eps^{-1} (y_{n+1} - y_{n-1})$
describes a second neighbour bond.

For future reference we also define the second and third finite
differences
\begin{displaymath}
  v_n'' = \frac{v_{n+1}' - v_n'}{\eps}, \quad \text{and}  \quad
  v_n'''(x) = \frac{v_{n+1}' - 2 v_n' + v_{n-1}'}{\eps^2}.
\end{displaymath}
It is also worth pointing out that $v_n' = \D v(s)$ for all $s \in
(x_{n-1}, x_n)$.

\subsection{Consistency of the QNL method}
\label{sec:e1d:qnl}
We begin with a modelling error analysis for the {\em quasinonlocal
  coupling method} (QNL method) of Shimokawa {\it et al}
\cite{Shimokawa:2004}. The geometrically consistent coupling scheme
\cite{E:2006} and the method proposed by Shapeev \cite{Shapeev:2010a}
reduce to the same method for 1D second neighbour interactions.

The following presentation follows largely \cite{Ortner:qnl.1d}, where
the QNL method is defined as follows: Let $\Laqnl = \{ -K, \dots,
K\}$ for some $K \geq 1$ and $\Lcqnl = \{ -N+1, \dots, N\}
\setminus \Laqnl$; then, for $y \in \Ys$, the QNL energy is defined by
\begin{equation}
  \label{eq:e1d:Eqnl}
  \begin{split}
    \Eqnl(y) = \eps \sum_{n = -N+1}^N \phi_1(y_n')
    ~& + \eps \sum_{n \in \Laqnl} \phi_2(y_n' + y_{n+1}') 
    + \eps \sum_{n \in \Lcqnl} \smfrac12 \b[ \phi_2(2 y_n') +
    \phi_2(2 y_{n+1}') \b].
  \end{split}
\end{equation}
We observe that we have not modified the first neighbour interactions,
but have ``split'' the non-local second neighbour interactions into
local first neighbour interactions in the continuum region.

It is straightforward to rewrite $\Eqnl$ in the form specified in
\eqref{eq:intro:Eac}, with
\begin{displaymath}
  \Omc = [\eps(K+1), 1] \cup (-1, \eps(-K-1)], \quad
  \La = \eps \{-K+1, \dots, K-1\},
\end{displaymath}
and a suitably defined interface functional $\Ei$; however, the form
\eqref{eq:e1d:Eqnl} is more convenient for the analysis.

The following modelling error estimate was first established in
\cite[Thm. 3.1]{Ortner:qnl.1d}. Dobson and Luskin \cite{Dobson:2008b}
treated a quadratic interaction case, using entirely different
analytical techniques that gave an even more detailed analysis of the
error; Ming and Yang \cite{emingyang} used related methods as
\cite[Thm. 3.1]{Ortner:qnl.1d}, but gave a qualitatively less precise
estimate of consistency error. An extension of the result to linear
finite range pair interactions is given in \cite{XHLi:genqnl1d}.

Note also that it is shown in \cite{Dobson:2008b} that the consistency
error of the QNL method in $\ell^p_\eps$-norms is of the order $O(1)$,
that is, the usage of negative norms cannot be avoided.

We will discuss the estimate in detail in
\S\ref{sec:e1d:discussion}. The proof of the following result serves
as a first guidance on how one may approach proofs of consistency of
a/c methods in more general situations. 

\begin{proposition}[Consistency of the QNL Method]
  \label{th:qnl}
  Let $y \in \Ys;$ then
  \begin{displaymath}
    \b\| \del\Eqnl(y) - \del\Ea(y) \b\|_{\WW^{-1,p}_\eps}
    \leq \eps m_2' \b\| y'' \b\|_{\ell^p_\eps(\{-K,K\})}
    + \eps^2 m_2' \b\| y''' \b\|_{\ell^p_\eps(\Lcqnl')}
    + \eps^2 m_2'' \b\| y'' \b\|_{\ell^{2p}_\eps(\Lcqnl)}^2,
  \end{displaymath}
  where $\Lcqnl' = \{-N+1,
  \dots, -K-1\} \cup \{K+2, \dots, N\}$,
\end{proposition}
\begin{proof}
  Throughout the proof we will make use of the fact that the boundary
  conditions are periodic without comment, treating the boundary as if
  it belonged to the ``interior'' of the continuum region.

  Since the first neighbour interactions as well as the second
  neighbour interactions in the atomistic region are treated
  identically in the atomistic model and the QNL method, we have
  \begin{align*}
    \b\< \del\Ea(y) - \del\Eqnl(y), u \b\> =
    \eps \sum_{n \in \Lcqnl} \B[ ~&\phi_2'(y_n'+y_{n+1}')
    \cdot (u_n' + u_{n+1}') \\[-3mm]
    & - \phi_2'(2y_n') \cdot u_n' - \phi_2'(2y_{n+1}') \cdot u_{n+1}' \B].
  \end{align*}
  Rearranging the sum in terms of the gradients $u_n'$, and using the
  fact that $\Laqnl = \{-K, \dots, K\}$, yields
  \begin{equation}
    \label{eq:e1d:20}
    \b\< \del\Ea(y) - \del\Eqnl(y), u_h \b\> = \eps \sum_{n = -N+1}^N
    \mR_n \cdot u_n',
  \end{equation}
  where $(\mR_n)_{n = 1}^N$ is defined as follows:
  \begin{displaymath} 
    \mR_n = \cases{
      0, 
      & n \in \{-K+1,\dots, K\}, \\
      \phi_2'(y_{n-1}' + y_{n}') - \phi_2'(2 y_n'), 
      & n = K+1, \\
      \phi_2'(y_n' + y_{n+1}') - \phi_2'(2 y_n'), 
      & n = -K, \\
      \phi_2'(y_n' + y_{n+1}') - 2 \phi_2'(2 y_n')
      + \phi_2'(y_{n-1}' + y_n'),
      & n \in \Lcqnl'. }
  \end{displaymath}
  At the interface, $n = K+1$, we have
  \begin{align*}
    \mR_n \leq m_2' \b|y_{n+1}' - y_n'\b| = \eps m_2' \b|y_{n-1}''\b|,
  \end{align*}
  with a similar estimate for $n = - K$.  In the continuum region, $n
  \geq K+2$, or $n \leq - K+1$, the terms $\mR_n$ have second order
  structure, and a second order Taylor expansion yields
  \begin{align*}
    |\mR_n| \leq~& \b|\phi_2''(2 y_n')\b|
    \b|y_{n+1}'- 2 y_n' + y_{n-1}'\b| 
    + \smfrac12 m_2'' \B[ \b|y_{n+1}' - y_n'\b|^2 + \b|y'_n -
    y'_{n_1} \b|^2\B] \\
    =~& \eps^2 m_2' \b| y_n'''\b| + \eps^2 \smfrac12 m_2'' \B[
    \b|y_n''\b|^2 + \b|y_{n-1}''\b|^2 \B].
  \end{align*}
  After inserting the two bounds into \eqref{eq:e1d:20}, and applying
  several weighted H\"{o}lder inequalities one obtains the stated
  estimate.
\end{proof}

\subsection{Inconsistency of the QCE method}
\label{sec:e1d:qce}
As in the previous section, let $\Laqnl = \{-K, \dots, K\}$, $\Lcqnl =
\{-N+1, \dots, N\} \setminus \Laqnl$, and $\La = \eps \Laqnl$; then,
under the assumptions and notation set out at the beginning of
\S\ref{sec:e1d}, the QCE energy functional defined in
\eqref{eq:intro:E_qce} reads
\begin{equation}
  \label{eq:e1d:defn_qce}
  \begin{split}
    \Eqce(y) =~& \eps \sum_{n \in \Laqnl} \smfrac{1}{2} \B[ \phi_1(y_n') +
    \phi_1(y_{n+1}') + \phi_2(y_{n-1}'+y_n') + \phi_2(y_{n+1}' +
    y_{n+2}') \B] \\
    & + \eps \sum_{n \in \Lcqnl} \smfrac{1}{2} \B[ \phi_1(y_n') +
    \phi_1(y_{n+1}') + \phi_2(2 y_{n}') + \phi_2(2 y_{n+1}') \B],
  \end{split}
\end{equation}
noting that $W(\mF) = \phi_1(\mF) + \phi_2(2 \mF)$.

The following result is a variant of
\cite[Thm. 3.2]{OrtnerWang:2009a}. Previous analyses of the QCE method
\cite{Dobson:2008c, Dobson:2008b, emingyang} computed the consistency
error contribution due to the ``ghost forces'' explicitly rather than
estimating their $\WW^{-1,p}_\eps$-residual contribution.

\begin{proposition}
  Let $y \in \Ys$ and $p \in [1,\infty]$; then
  \begin{align*}
    \b\| \del\Eqce(y) - \del\Ea(y) \b\|_{\WW^{-1,p}_\eps} \leq~&
    \eps^{1/p} G + \eps m_2' \b\| y''
    \b\|_{\ell^p_\eps(\mathscr{N}_\i)} \\
    & + \eps^2 m_2' \b\| y''' \b\|_{\ell^p_\eps(\Lcqnl')} + \eps^2
    m_2'' \b\| y'' \b\|_{\ell^{p}_\eps(\Lcqnl \cup \{-K,K\})}^{2p},
  \end{align*}
  where $\Lcqnl'$ is defined as in Proposition~\ref{th:qnl},
  $\mathscr{N}_\i = \{ -K-1, -K, K, K+1 \}$, and
  \begin{displaymath}
    G = \smfrac12 \b[ \b|\phi_2'(2 y_{-K-1}')\b| + \b|\phi_2'(2 y_{-K+1}')\b| +
    \b|\phi_2'(2 y_K')\b| + \b|\phi_2'(2 y_{K+2}')\b| \b].
  \end{displaymath}

  The estimate is sharp in the sense that, for some constant $C$,
  $\smfrac12 \leq C \leq 2$,
  \begin{displaymath}
    \b\| \del\Eqce(\yA) \b\|_{\WW^{-1,p}_\eps} 
    = \b\| \del \Eqce(\yA) - \del\Ea(\yA) \b\|_{\WW^{-1,p}_\eps}
    \geq  C \eps^{1/p} \b|\phi_2'(2 \mA)\b|
    \qquad \forall \mA \in \R.
  \end{displaymath}
\end{proposition}
\begin{proof}
  The first result can be proven in much the same way as Proposition
  \ref{th:qnl}, by rewriting the first variation $\del\Eqce$ in the
  form $\< \del\Eqce(y), u \> = \eps \sum_{n = -N+1}^N \mR_n \cdot
  u_n'$, and carefully estimating the coefficients $\mR_n$. See
  \cite{OrtnerWang:2009a} for the details of this computation.

  To obtain the opposite estimate for $y = \yA$, a brief computation
  gives
  \begin{displaymath}
    \b\< \del\Eqce(\yA) - \del\Ea(\yA), u \b\> =
    \eps \smfrac12 \phi_2'(2\mA) \b[u_{-K-1}' - u_{-K+1}'  - u_K' +
    u_{K+2}' \b].
  \end{displaymath}
  If we choose $u \in \Us$ such that
  \begin{displaymath}
    u_n' = {\rm sign}\b(\phi_2'(2\mA)\b) \cdot \cases{
      (4 \eps)^{-1/p'}, &  n = -K-1, K+2, \\
      - (4 \eps)^{-1/p'}, & n = -K+1, K, \\
      0, & \text{ otherwise},
    }
  \end{displaymath}
  then $\| \D u \|_{\LL^{p'}(-1,1)} = \| u' \|_{\ell^{p'}_\eps} = 1$, and we
  obtain that
  \begin{displaymath}
    \b\| \del\Eqce(\yA) - \del\Ea(\yA) \b\|_{\WW^{-1,p}_\eps}
    \geq \b\< \del\Eqce(\yA) - \del\Ea(\yA), u \b\> 
    =  \eps^{1/p} \b[2 \cdot 4^{-1/p'} \b|\phi_2'(2\mA)\b|\b]. \qedhere
  \end{displaymath}
\end{proof}

\subsection{Discussion}
\label{sec:e1d:discussion}
This discussion of the 1D consistency error estimates largely follows
the discussions in \cite{OrtnerWang:2009a,Ortner:qnl.1d}.

We have estimated the modelling errors for two prototypical a/c
methods. We see that the leading order terms in the upper bounds are
$O(\eps)$ and $O(\eps^{1/p})$ for the QNL and QCE methods,
respectively. However, a much finer distinction should be made.

First, we note that both methods reduce to the Cauchy--Born
approximation in the continuum region, and the corresponding
contributions are all of second order (see also Remark
\ref{rem:intro:cb}; note also that this requires point symmetry of
$V$, which we have not assumed in general in this paper).

Second, we see that the QCE method (and only the QCE method) has a
zeroth-order term $G \eps^{1/p}$ in the interface region. This term
occurs since the QCE method is not patch test consistent, that is,
homogeneous deformations are not equilibria of the QCE model:
\begin{displaymath}
  \del\Eqce(\yA) \neq 0.
\end{displaymath}
The origin and effect of these ``ghost forces'' are discussed in more
detail in \cite{Shenoy:1999a, Miller:2003a, Dobson:2008c,
  Dobson:2008b, emingyang}.

We should call this term zeroth order for several reasons: Firstly, it
is clearly of zeroth order if $p = \infty$, in which case the
consistency error is related to the error in the
$\WW^{1,\infty}$-norm. Secondly, the parameter $\eps$ is a constant of
the problem and does {\em not} tend to zero. As a matter of fact, the
accuracy of an a/c method should be related the smoothness of the
solution (as opposed to the atomistic scale), and the term $G
\eps^{1/p}$ is independent of the magnitude of $y''$ in the interface
region. The scaling $\eps^{1/p}$ relates only to the {\em width} of
the interface region.

Finally, it is worth remarking on the first-order consistency term in
the interface region for the QNL method. The reason this term is of
first order as opposed to second order is the loss of symmetry that is
introduced by changing the interaction law at the interface between
the atomistic and continuum regions.  A recent result of Dobson
\cite{Dob:pre} shows that no a/c method coupling an interatomic
potential to the Cauchy--Born continuum model can achieve better than
first-order accuracy in the interface region.

Note also, that the second finite differences $y_n''$ can in fact be
written in terms of the oscillation operator:
\begin{displaymath}
  \b|y_n'' \b| = {\rm osc}\b( \D y; [x_n-\eps, x_n+\eps]\b).
\end{displaymath}

In our analysis in \S\ref{sec:2d}, we will ignore the possibility of
proving a modelling error estimate that is of second order in the
continuum region, but we will be satisfied with an estimate that is
globally of first order. 




\section{Auxiliary Results}
\label{sec:aux}

\subsection{The bond density lemma}
\label{sec:aux:bdl}
The bond density lemma is a tool that allows a transition between
integrals over bonds, and volume integrals. It was first derived in
\cite{Shapeev:2010a} for the construction of patch test consistent a/c
methods for pair potential interactions. Related asymptotic results
were used previously in $\Gamma$-convergence analyses of atomistic
models \cite{AlCi:2004}; the achievement of Shapeev
\cite{Shapeev:2010a} was to obtain a formula that is {\em exact} for
any triangle.

Before we formulate the result we introduce some notation. Let $T
\subset \R^2$ be a triangle with vertices belonging to $\Lext$. We
define the characteristic function $\chi_T : \R^2 \to \R$ by
\begin{displaymath}
  \chi_T(x) := \lim_{t \searrow 0} \frac{ \b| T \cap B_t(x) \b|}{|B_t(x)|},
\end{displaymath}
where $B_t(x)$ denotes the closed ball with centre $x$ and radius
$t$. Note that $\chi_T = 1$ in ${\rm int}(T)$, and $\chi_T = 1/2$ on
the edges of $T$. The value of $\chi_T$ on the corners is not of
importance.

Let $x, x' \in \R^2$ and let $\varphi$ be a function that is measurable on
the line segment $(x, x')$; then we define the line integral, or {\em
  bond integral},
\begin{displaymath}
  \mint_x^{x'} \varphi \db := \int_{t = 0}^1 \varphi\b( (1-t) x + t x' \b) \dt.
\end{displaymath}

\begin{lemma}[Bond Density Lemma (\cite{Shapeev:2010a}, Lemma 2)]
  \label{th:aux:bdl}
  Let $T \subset \R^2$ be a non-degenerate triangle with vertices
  belonging to $\Lext = \eps \Z^2$, and let $r \in \Z^2$; then
  \begin{equation}
    \label{eq:aux:bdl}
    \eps^2 \sum_{x \in \Lext} \mint_{x}^{x+\eps r} \chi_T \db = |T|.
  \end{equation}
\end{lemma}
%
%

As a first application of the bond density lemma we present a proof of
Lemma \ref{th:fram:Dbarvh_Dvh} in the appendix.

\begin{remark}
  In the above form, the bond density lemma is false in 3D, which is
  one of the reasons why the present work is restricted to 2D.
  Moreover, the condition that the vertices of $T$ belong to lattice
  sites is also necessary. This is related to the assumption that the
  vertices of $\Th$ belong to $\Lext$, however, this is not crucial and
  could be removed with some additional work.
\end{remark}

\subsection{Discrete divergence-free tensor fields in 2D}
\label{sec:aux:helmh}
Our second auxiliary result concerns representations of discrete
divergence-free $\PO$-tensor fields. Following the construction given
by Polthier and Preu\ss~\cite{PolPre:2003}, we will give a proof for
the periodic setting. This proof also serves to motivate a crucial
argument in \S\ref{sec:2d:Est_wh}. Since we will only use the
atomistic finite element mesh $\Teps$ from now on, we will formulate
everything in terms of this mesh. However, all results hold for
general periodic triangulations.

\paragraph{The Crouzeix--Raviart finite element space}
\label{sec:aux:CR_space}
The representation of discrete divergence-free tensor fields requires
the use of the non-conforming Crouzeix--Raviart finite element
space. Recall from \S\ref{sec:fram:disc_cts}, \ref{sec:intro:galerkin}
the definition of the sets of edges $\Edg$ and $\Edgext$. The
Crouzeix--Raviart finite element space over $\Tepsext$ is defined as
\begin{align*}
  \CR(\Tepsext) = \b\{ w : \cup_{T \in \Tepsext} {\rm int}(T) \to \R
  \bsep~& w \text{ is piecewise affine w.r.t. } \Tepsext, \text{ and }\\
  & \text{ continuous in edge midpoints } q_f, f \in \Edgext \b\}.
\end{align*}
The degrees of freedom for functions $w \in \CR(\Tepsext)$ are the
values at edge midpoints, $w(q_f)$, $f \in \Edgext$, and the
corresponding nodal basis functions are denoted by $\zeta_f$.

The Crouzeix--Raviart finite element space of periodic functions is
defined as 
\begin{displaymath}
  \CRper(\Teps) = \b\{ w \in \CR(\Tepsext) \bsep w(\xi + x) =
  w(x) \text{ for $\dx$-a.e. } x \in \R^2, \xi \in 2 \Z^2 \b\}.
\end{displaymath}
The periodic nodal basis functions are defined, for $f \in \Edg$, by
$\zeta_f^\per = \sum_{\xi \in 2 \Z^2} \zeta_{\xi+f}$.

\paragraph{Path integrals}
\label{sec:aux:pathints}
For two edges $f, f' \in \Edgext$, let $\Gamma_{f,f'}$ denote the set
of all piecewise affine paths from $q_f$ to $q_{f'}$, crossing element
edges only in edge midpoints; see Figure \ref{fig:cr_paths}(a) for an
example.

\begin{figure}
  \begin{center}
    \includegraphics[height=4cm]{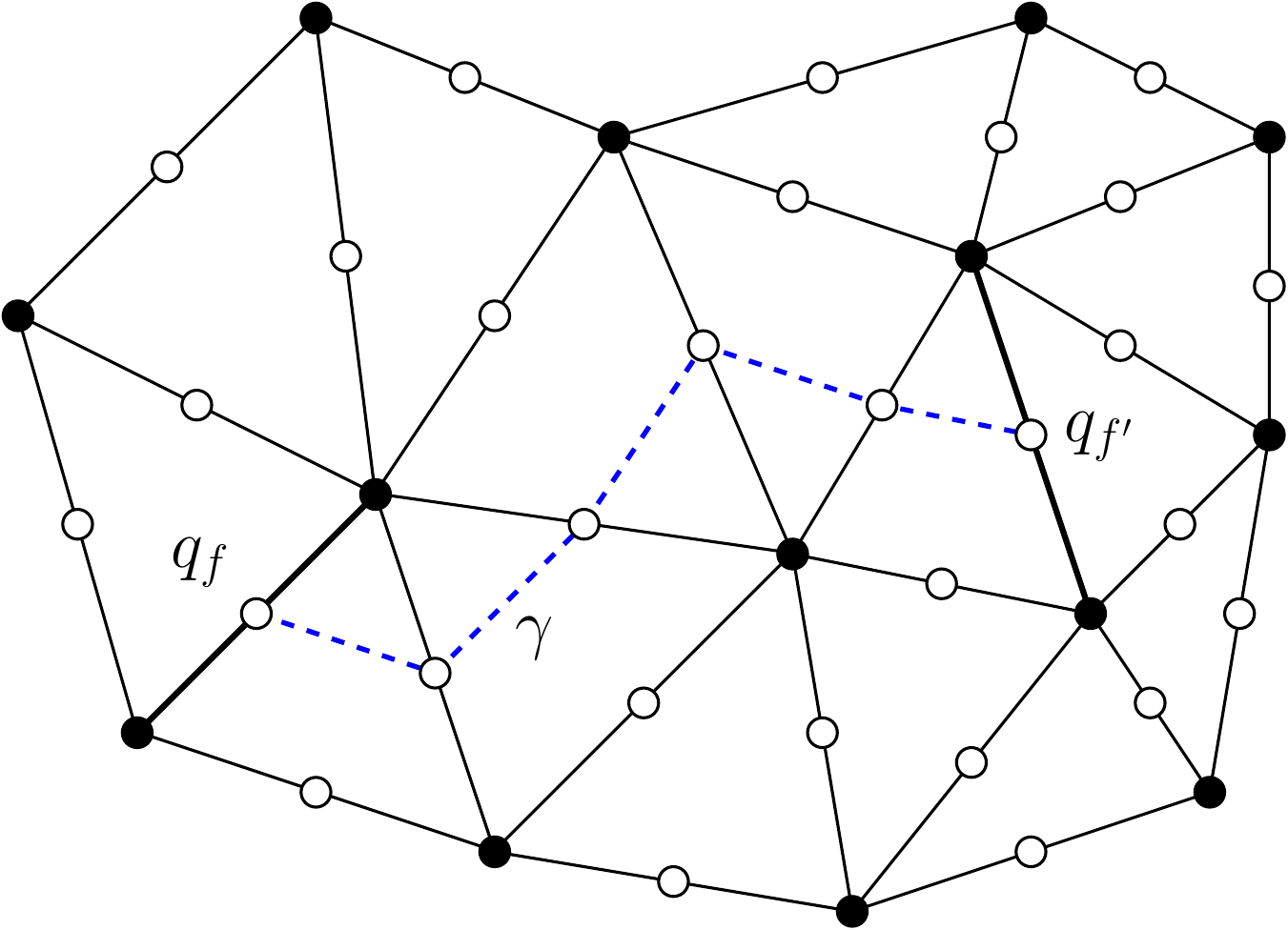}
    \quad \qquad  \includegraphics[height=4cm]{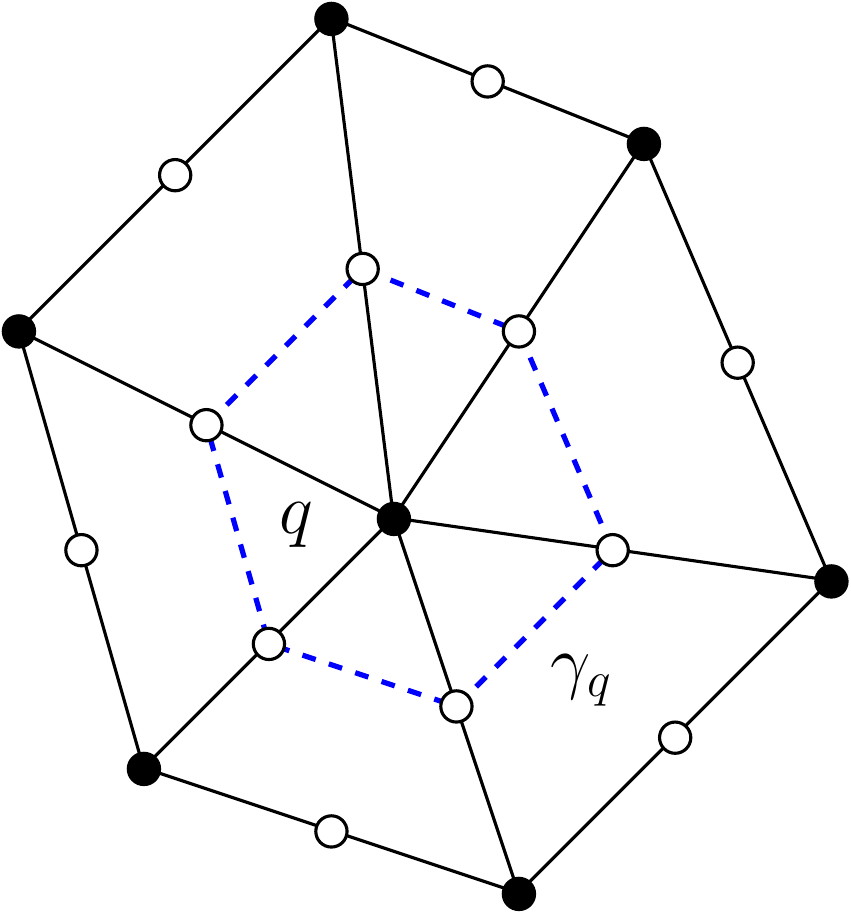}
    \\[-2mm]
    {\footnotesize \hspace{1cm} (a) \hspace{5cm} (b) }
  \end{center}
  \caption{\label{fig:cr_paths} (a) Illustration of a piecewise affine
    path $\gamma \in \Gamma_{f, f'}$. (b) Illustration of the path
    $\gamma_q$ used in the proof of Lemma \ref{th:aux:divfree}.}
\end{figure}

For any piecewise constant vector field $\sigma \in
\PO(\Tepsext)^2$ and for any path $\gamma \in \Gamma_{f,f'}$, $\gamma
= \{ x(t) \sep 0 \leq t \leq T \}$, we denote the standard path
integral by
\begin{displaymath}
  \int_{\gamma} \sigma \cdot \dx = \int_{t = 0}^T \sigma \cdot
  \dot{x}(t) \dt.
\end{displaymath}
For piecewise constant tensor fields $\sigma \in
\PO(\Tepsext)^{k \times 2}$ we define the path integral as
\begin{displaymath}
  \int_{\gamma} \sigma \cdot \dx = \int_{t = 0}^T \sigma
  \dot{x}(t) \dt.
\end{displaymath}

Since functions $w \in \CR(\Tepsext)^k$ have piecewise constant
gradients $\D w$, and since they are continuous in edge midpoints,
it is easy to see that
\begin{equation}
  \label{eq:aux:pathint_Dw}
  \int_\gamma \D w \cdot \dx = w(q_{f'}) - w(q_f) \qquad
  \forall \gamma \in \Gamma_{f,f'}.
\end{equation}

\paragraph{Discrete divergence-free tensor fields}
The following lemma characterizes discrete divergence-free tensor
fields.

\begin{lemma}
  \label{th:aux:divfree}
  A tensor field $\sigma \in \POper(\Teps)^{k \times 2}$ satisfies
  \begin{displaymath}
    \int_\Om \sigma : \D u \dx = 0 \qquad \forall u \in \PIper(\Teps)^{k}
  \end{displaymath}
  if and only if there exist a constant $\sigma_0 \in \R^{k \times 2}$
  and a function $w \in \CRper(\Teps)^k$ such that
  \begin{displaymath}
    \sigma = \sigma_0 + \D w \mJ,
    \qquad \text{where} \quad \mJ = \mymat{0 & -1 \\ 1 & 0} \in {\rm SO}(2).
  \end{displaymath}
\end{lemma}
\begin{proof}[Sketch of the proof]
  The reverse direction, that any tensor field of the form $\sigma =
  \sigma_0 + \D w \mJ$ has zero discrete divergence, can be checked
  using a straightforward calculation, using
  \eqref{eq:aux:pathint_Dw}.

  {\it Step 0. Outline: } 
  To simplify the notation, we define $\alpha = \sigma \mJ^T$ and,
  without loss of generality, we assume that $k = 1$.  Initially, we
  treat $\alpha$ as a piecewise constant tensor field on all of
  $\R^2$, ignoring periodicity. We construct $w \in \CR(\Tepsext)$ by
  explicitly specifying $w(q_f)$ for all edges $f \in \Edgext$. We
  will then show in the last step of the proof that $w$ can be
  written as the sum of an affine function and a periodic function.

  {\it Step 1. Construction of $w$: } 
  Fix a starting edge $\hat{f} \in \Edg$ and define $w(q_{\hat{f}}) =
  0$. For any edge $f \in \Edgext$ let $\gamma \in \Gamma_{\hat{f},
    f}$ and define $w(q_f)$ via the path integral
  \begin{displaymath}
    w(q_f) := \int_\gamma \alpha \cdot \dx.
  \end{displaymath}
  We need to show that this definition is independent of the
  path.
  
  Let $q$ be a vertex of the triangulation $\Tepsext$ and let
  $\varphi_q$ be the corresponding nodal basis function with support
  $\omega_q$; then a fairly straightforward calculation (see, e.g.,
  \cite{PolPre:2003} for the details) shows that
  \begin{align*}
    0 = \int_{\omega_q} \sigma \cdot \D \varphi_q \dx 
    = \frac12 \sum_{\substack{f \in \Edgext \\ f \subset {\rm
          int}(\omega_q)}} \B( \b(\sigma \cdot \nu_f\b)^+ + \b(\sigma
    \cdot \nu_f\b)^- \B) 
    = \int_{\gamma_q} \alpha \cdot \dx,
 \end{align*}
 where $\nu_f^\pm$ are the two unit normals to $f$, and $\gamma_q$ is
 the piecewise affine path through edge midpoints circling $q$; cf.
 Figure \ref{fig:cr_paths}(b). Note that the rotation $\mJ$ in the
 definition of $\alpha$ comes from the fact that tangent vectors are
 rotated normal vectors.

 Since all closed piecewise affine paths can be written as a sum over
 paths $\gamma_q$, this implies that $\int_\gamma \alpha \cdot \dx =
 0$ for all closed piecewise affine paths $\gamma$, and in particular
 that the definition of $w(q_f)$ is independent of the choice of path,
 that is, $w$ is well-defined.

 {\it Step 2. $\D w = \alpha$: } 
 From the definition of $w(q_f)$, $f \in \Edgext$, it follows that,
 for $f, f' \subset T \in \Tepsext$,
 \begin{displaymath}
   \nabla w(T) \cdot (q_f - q_{f'}) = w(q_f) - w(q_{f'}) =
   \alpha(T) \cdot (q_f - q_{f'}),
 \end{displaymath}
 which immediately implies that $\D w(T) = \alpha(T)$.

 {\it Step 3. Periodicity: }
 We are only left to show that $w(x) = a \cdot x + w_1(x)$ for some $a
 \in \R^2$, $a = (a_1,a_2)$, and $w_1 \in \CRper(\Teps)$. Let $a_j =
 w(q_{\hat{f}} + 2 e_j)$, $j = 1,2$, and define $w_1(x) = w(x) - a
 \cdot x$. Fix $j \in \{1,2\}$, let $f \in \Edgext$ and let $\gamma$
 be a path from $q_{\hat{f}}$ to $q_f$. Since $\alpha$ is
 $2\Z^2$-periodic, and since $w_1(q_{\hat{f}}+2 e_j) =
 w_1(q_{\hat{f}}) = 0$, we have
 \begin{displaymath}
   w_1(q_{f}+2e_j) = \int_{\gamma + 2e_j} \alpha \cdot \dx
   = \int_\gamma \alpha \cdot \dx = w_1(q_f).
 \end{displaymath}
 This shows that $w_1$ is periodic and thus concludes the proof. 
\end{proof}



\section{A General Consistency Result in 2D}
\label{sec:2d}
We are now finally in a position to make precise the statement that
{\em patch test consistent a/c methods are first-order
  consistent}. Motivated by the example of the QNL consistency result
(ignoring, as discussed in \S\ref{sec:e1d:discussion}, the
second-order consistency of the Cauchy--Born approximation), we would
like to prove a result of the form
\begin{displaymath}
  \b\| \del\Eac(y) - \del\Ea(y) \b\|_{\WW^{-1,p}_\eps}
  \leq C \eps \b\| \D^2 y \b\|_{\LL^p(\Omc \cup \Omi)}.
\end{displaymath}
As explained in \S\ref{sec:fram:osc} we will use the oscillation
operator \eqref{eq:fram:defn_osc} to replace the undefined second
derivative.

For each element $T \in \Tepsext$, let $\omega_T^\a$ be the {\em
  interaction neighbourhood} of $T$ in the atomistic model,
\begin{equation}
  \label{eq:2d:defn_intnhd_a}
  \omega_T^\a := \b\{ x + t_1 r_1 + t_2 r_2 \bsep
  x \in T, t_i \in [0, 1], r_i \in \Rg \b\},
\end{equation}
and $\omega_T$ its union with the set $\omega_T^\c$, restricted to the
continuum and interface regions,
\begin{equation}
  \label{eq:2d:defn_intnhd}
  \omega_T := \bg( \omega_T^\a \cup \bigcup_{\substack{T' \in \Tepsext
      \\ T \cap T' \neq \emptyset}} T' \bg)
  \setminus \Oma^\per.
\end{equation}
Note that assumption \eqref{eq:intro:noint_La_Omc} implies that
$\omega_T^\a \cup \omega_T^\c \subset \omega_T$ for all $T \in
\Tepsc$.

\begin{theorem}[First-order Consistency of Patch Test Consistent a/c
  Methods]
  \label{th:2d:main_result}
  \quad Suppose that $\Eac$ is patch test consistent
  (\S\ref{sec:intro:patch_test}) and globally energy consistent
  (\S\ref{sec:intro:qce}), that the locality condition
  \eqref{eq:intro:Fi_loc} and the scaling condition
  \eqref{eq:intro:Fi_scal} hold, that ${\rm int}(\Oma)$ is connected,
  that $\Tha\cup\Thi \subset \Teps$, and that
  \eqref{eq:intro:noint_La_Omc} holds. Then, for any $y \in \Ys,$ we
  have
  \begin{equation}
    \label{eq:2d:main_result}
    \b\| \del\Eac(y) - \del\Ea(y) \b\|_{\WW^{-1,p}_\eps}
    \leq \eps \bg\{ \sum_{T \in \Tepsc \cup \Tepsi} 
    |T| \B[ M_T \, {\rm osc}(\D y; \omega_T) \B]^p \bg\}^{1/p},
  \end{equation}
  where the oscillation measure ${\rm osc}$ is defined in
  \eqref{eq:fram:defn_osc}, the interaction neighbourhood $\omega_T$
  is defined in \eqref{eq:2d:defn_intnhd}, and the prefactors $M_T$
  are defined as follows:
  \begin{equation}
    \label{eq:2d:defn_barM}
    M_T = \cases{ 
      0, & T \in \Tepsa, \\
      \b(M^\i + M^\a\b) \b(1 + 7\, {\rm width}(\Omi)\b), & T
      \in \Tepsi, \\
      M^\a + 7 \b(M^\i + M^\a\b) {\rm width}(\Omi),
      & T \in \Tepsc.}
  \end{equation}
  In \eqref{eq:2d:defn_barM}, the constants $M^\a$ and $M^\i$ are
  defined in \eqref{eq:intro:defn_Ma} and \eqref{eq:intro:defn_Mi},
  and the ``interface width'' ${\rm width}(\Omi)$ is given by
  \begin{equation}
    \label{eq:2d:defn_widthOmi}
    {\rm width}(\Omi) := \max_{\substack{f \in \Edg \\ f \subset
        \Omi^\per}} \min_{\substack{f' \in \Edgext \\ f' \subset \Oma^\per}}
    \min_{\gamma \in \Gamma_{f,f'}} \, \frac{{\rm length}(\gamma)}{\eps}.
  \end{equation}
\end{theorem}
\begin{proof}[Outline of the proof]
  We will construct ``stress functions'' $\Sa(y), \Sac(y) \in
  \PO(\Teps)^{2 \times 2}$ such that
  \begin{displaymath}
    \b\< \del\Ea(y), u \b\> = \int_\Om \Sa(y) : \D u \dx, \quad
    \text{and} \quad
    \b\< \del\Eac(y); u \b\> = \int_\Om \Sac(y) : \D u \dx
    \qquad \forall u \in \Us.
  \end{displaymath}
  If we could prove an estimate of the form
  \begin{displaymath}
    \b|\Sa(y; T) - \Sac(y; T)\b| \lesssim \eps\, {\rm osc}(\D y; \omega_T),
  \end{displaymath}
  then the result would follow immediately. It turns out that this is
  not possible.

  Instead, we will use the fact that $\Eac$ is globally energy consistent and
  patch test consistent to construct a correction
  (cf. Corollary~\ref{th:2d:econs_cor}, Lemma \ref{th:2d:cons_of_pt},
  and \S\ref{sec:2d:defn_Sacm})
  \begin{displaymath}
    \Sacm(y) = \Sac(y) - \D \modw(y) \mJ,
  \end{displaymath}
  which still represents the first variation $\del\Eac$
  (cf. Lemma~\ref{th:aux:divfree}), and which has the property that
  \begin{displaymath}
    \Sa(\yF; T) = \Sacm(\yF; T) = \pp W(\mF) \qquad \forall\, T \in
    \Tepsc \cup \Tepsi, \quad \mF \in \R^{2 \times 2}.
  \end{displaymath}
  In addition, we show that $\Sacm(y; T) = \Sa(y; T)$ for all $T \in
  \Tepsa$.

  Lipschitz estimates for $\Sa$ and $\Sacm$, and careful modifications
  of the argument in the interface region, yield the following result
  (cf. Lemma \ref{th:2d:main_estimate}):
  \begin{align*}
    \b| \Sacm(y; T) - \Sa(y; T) \b| \leq~& \b| \Sacm(y; T) - \pp
    W(\D y(T)) \b| + \b| \Sa(y; T) - \pp W(\D y(T)) \b| \\
    \leq~& \eps M_T {\rm osc}(\D
    y ; \omega_T) \qquad \forall T \in \Teps,
  \end{align*}
  and, in particular,
  \begin{align*}
    \b\< \del\Eac(y) - \del\Ea(y); u \b\> =~&
    \sum_{T \in \Teps} |T| \B[ \Sacm(y; T) - \Sa(y; T) \B] : \D u(T)
    \\
    \leq~& \eps \bg\{ \sum_{T \in \Tepsc \cup \Tepsi} 
    |T| \B[ M_T \, {\rm osc}(\D y; \omega_T) \B]^p \bg\}^{1/p}
    \, \b\| \D u \|_{\LL^{p'}(\Om)},
  \end{align*}
  which yields the stated first order consistency estimate.
\end{proof}

\begin{remark}
  In 1D, a similar result can be proven using a similar framework but
  with significantly reduced technicalities. Note, in particular, that
  the 1D analogue of Lemma \ref{th:aux:divfree} is 
  \begin{displaymath}
    \int_\Om \sigma \cdot \D u \dx = 0 \quad \forall \, u \in \Us
    \qquad \text{if and only if} \qquad  \sigma \text{ is constant}.
  \end{displaymath}
  Hence, the corrector function $\w(\mF, \cdot)$ defined in
  \S\ref{sec:2d:cons_pt} is always identically equal to zero, which
  removes the interface width dependence from the modelling error
  (cf. \S\ref{sec:2d:Est_wh}). Hence, if $d = 1$, we obtain
  \eqref{eq:2d:main_result} again but with modified prefactors
  \begin{equation}
    \label{eq:2d:1d_result}
    M_T^{\rm 1D} := \cases{
      0, & T \in \Tepsa, \\
      M^\i + M^\a, & T \in \Tepsi, \\
      M^\a, & T \in \Tepsc.
    }
  \end{equation}
  Moreover, since $\w \equiv 0$ in 1D, and since symmetries are more
  easily exploited, it is not too difficult in 1D to prove second
  order consistency in the continuum region.
\end{remark}

\subsection{The atomistic stress function}
\label{sec:2d:Sa}
A natural ``weak'' representation of $\del\Ea(y)$, $y \in \Ys$, is
given by
\begin{equation}
  \label{eq:2d:delEa_v1}
  \< \del\Ea(y), z \> = \eps^2 \sum_{x \in \L} \sum_{r \in \Rg} \pp_r
  V(\Da{\Rg}) \cdot \Da{r} z(x), \qquad \text{for } z \in \Ys.
\end{equation}
Using bond integrals we rewrite this in a form that will be useful for
our subsequent analysis. We will then use the bond density lemma
whenever we need to transition between bond integrals and volume
integrals. This process yields a notion of stress for atomistic
models, which is related to the virial stress (see \cite{AdTa:2010}
for a recent reference; this connection will be discussed in detail
elsewhere). A variant of this result for pair interactions in 1D was
developed in \cite{MakrOrtSul:qcf.nonlin}.

\begin{proposition}
  \label{th:2d:Sa}
  Let $y, z \in \Ys$, then
  \begin{displaymath}
    \< \del\Ea(y), z \> = \sum_{T \in \Teps} |T| \, \Sa(y; T) : \D z(T)
    = \int_\Om \Sa(y) : \D z \dx,
  \end{displaymath}
  where the stress function $\Sa(y) \in \POper(\Teps)^{2 \times 2}$ is
  defined as follows:
  \begin{equation}
    \label{eq:2d:defn_Sa}
    \Sa(y; T) := \sum_{r \in \Rg} \frac{\eps^2}{|T|}
    \sum_{x \in \Lext} \B[\pp_r V\b(\Da{\Rg} y(x)\b) \otimes r \B] \mint_{x}^{x+\eps
      r} \chi_{T} \db.
  \end{equation}
\end{proposition}
\begin{proof}
  For the sake of brevity we will write $V_{x,r} = \pp_r V(\Da{\Rg}
  y(x))$.

  Recall that $T^\per = \bigcup_{\xi \in 2 \Z^2} (\xi+T)$, for $T \in
  \Teps$. It is easy to see that $\{ \chi_{T^\per} \sep  T \in \Teps \}$
  is a partition of unity for $\R^2$. Hence, using the identity
  \begin{equation}
    \label{eq:2d:13}
    \mint_{x}^{x+\eps r} \Dc{r} z \db = \Da{r} z(x),
  \end{equation}
  we can rewrite \eqref{eq:2d:delEa_v1} as
  \begin{displaymath}
    \b\< \del\Ea(y), z \b\> = \eps^2 \sum_{x \in \L} \sum_{r \in
      \Rg} V_{x,r} \cdot \mint_{x}^{x+\eps r} \B[ \sum_{T \in \Teps}
    \chi_{T^\per} \B] \Dc{r} z \db.
  \end{displaymath}
  Interchanging the order of summation and using the fact that $\Dc{r}
  z = (\D z) r$ holds $\db$-a.e. (note that, if the bond is aligned
  with an element edge, then $\Dc{r} z$ is continuous across that
  edge) yields
  \begin{displaymath}
    \< \del\Ea(y), z \> = \sum_{T \in \Teps} |T| \sum_{r \in \Rg}
    \frac{\eps^2}{|T|} \sum_{x \in \L} V_{x,r}
    \cdot \mint_{x}^{x+\eps r} \chi_{T^\per} \D z(T^\per) r \db.
 \end{displaymath}
 The term $\D z(T^\per) r = \D z(T) r$ can be taken outside the bond
 integral, and hence, employing the identity
  \begin{equation}
    \label{eq:2d:aGr_eq_arG}
    a \cdot (\mG r) = (a \otimes r) : \mG,
    \qquad \text{for } a, r \in \R^d, \mG \in \R^{d \times d},
 \end{equation}
 yields
 \begin{align*}
   \< \del\Ea(y), z \> =~& \sum_{T \in \Teps} |T| 
   \bg\{ \sum_{r \in \Rg} \frac{\eps^2}{|T|} \sum_{x \in \L}  
   \b[ V_{x,r} \otimes r \b]
   \mint_{x}^{x+\eps r} \chi_{T^\per} \db \bg\} : \D z(T) \\
   =:~& \sum_{T \in \Teps} |T| \, \Sa(y; T) : \D z(T).
 \end{align*}
 Finally, we use the fact that $\Da{\Rg} y$ is $2\Z^2$-periodic to deduce that
 \begin{align*}
   \Sa(y; T) =~& \sum_{r \in \Rg} \frac{\eps^2}{|T|} \sum_{x \in \L}  
   \b[ V_{x,r} \otimes r \b]
   \mint_{x}^{x+\eps r} \chi_{T^\per}  \db \\
   =~& \sum_{r \in \Rg} \frac{\eps^2}{|T|} \sum_{x \in \L}  
  \b[ V_{x,r} \otimes r \b] \sum_{\xi \in 2 \Z^2}
   \mint_{x}^{x+\eps r} \chi_{\xi+T} \db \\
   =~& \sum_{r \in \Rg} \frac{\eps^2}{|T|} \sum_{x \in \Lext}  
  \b[ V_{x,r} \otimes r \b] 
   \mint_{x}^{x+\eps r} \chi_{T} \db. \qedhere
 \end{align*}
\end{proof}

The terminology ``stress function'' for $\Sa$ is motivated by the fact
that $\Sa(y)$ takes precisely the same role as the first
Piola--Kirchhoff stress tensor in the continuum theory of elasticity.

In the next lemma we prove two useful properties of the atomistic
stress function $\Sa$. We show that $\Sa = \pp W$ under locally
homogeneous deformations and give a quantitative estimate for the
discrepancy between $\Sa$ and $\pp W$.

\begin{lemma}
  \label{th:2d:propSa}
  The stress function $\Sa$ defined in \eqref{eq:2d:defn_Sa} satisfies
  \begin{displaymath}
    \Sa(\yF; T) = \pp W(\mF) \qquad \forall \mF \in \R^{2 \times 2},
    \quad T \in \Teps.
  \end{displaymath}
  Moreover, we have the estimate
  \begin{equation}
    \label{eq:2d:Lip_Sa}
    \b| \Sa(y; T) - \pp W\b(\D y (T)\b) \b|
    \leq \eps M^\a {\rm osc}(\D y; \omega_T^\a)
    \qquad \forall y \in \Ys, \quad T \in \Teps,
  \end{equation}
  where $M^\a$ is defined in \eqref{eq:intro:defn_Ma} and
  $\omega_T^\a$ is defined in \eqref{eq:2d:defn_intnhd_a}.
\end{lemma}
\begin{proof}
  {\it Part 1: } Since $[\pp_r V(\mF \Rg) \otimes r]$ is independent
  of $x$, we can apply the bond density lemma to the sum in curly
  brackets, to deduce that
  \begin{displaymath}
   \Sa(\yF; T) = \sum_{r \in \Rg} \B[\pp_r V\b(\mF \Rg\b)
    \otimes r \B] \bg\{
    \frac{\eps^2}{|T|} \sum_{x \in \Lext} \mint_{x}^{x+\eps r} 
    \chi_{T} \db  \bg\} 
    = \sum_{r \in \Rg}  \B[\pp_r V\b(\mF \Rg\b)
    \otimes r \B].
  \end{displaymath}
  Recalling that $W(\mF) = V(\mF \Rg)$, it can be easily checked that
  the sum on the right-hand side of the second equality equals $\pp
  W(\mF)$.

  {\it Part 2: } Let $\mF = \D y(T)$, $V_{x,r} = \pp_r V(\Da{\Rg}
  y(x))$, and $V_{\mF, r} = \pp_r V(\mF \Rg)$. From part 1 we obtain
  that
  \begin{align}
    \notag
    \b| \Sa(y; T) - \pp W(\mF) \b| =~&
    \bg|\sum_{r \in \Rg} \frac{\eps^2}{|T|} \sum_{x \in \Lext} 
    \B[ \b( V_{x, r} - V_{\mF, r} \b) \otimes r \B]
    \mint_{x}^{x+\eps r} \chi_{T} \db \bg| \\
    \label{eq:2d:21}
    \leq~& \sum_{r \in \Rg} |r| \frac{\eps^2}{|T|} \sum_{x \in \Lext} 
    \B|V_{x, r} - V_{\mF, r}\B|
    \mint_{x}^{x+\eps r} \chi_{T} \db.
  \end{align}

  We use the Lipschitz property \ref{eq:intro:LipV} to estimate
  \begin{align}
    \notag
    \b|V_{x, r} - V_{\mF, r}\b| \leq~& \sum_{s \in \Rg} M_{r,s}^\a 
    \b| \Da{s} y(x) - \mF s \b| \\
    \notag
    =~& \sum_{s \in \Rg} M_{r,s}^\a
    \bg| \mint_{x}^{x+\eps s} \B(\Dc{s} y - \mF s \B) \db \bg| \\
    \label{eq:2d:23c}    
    \leq~& \sum_{s \in \Rg} M_{r,s}^\a |s|
    \max_{x' \in (x, x+\eps s)} \b| \D y(x') - \mF \b|.
  \end{align}
  Since $(x, x+\eps s) \subset \omega_T^\a$, and recalling that $\mF =
  \D y(T)$, we can further estimate
  \begin{equation}
    \label{eq:2d:25}    
    \max_{x' \in (x, x+\eps s)} \b| \D y(x') - \mF \b| \leq
    \eps {\rm osc}(\D y; \omega_T^\a).
  \end{equation}

  We combine \eqref{eq:2d:25} with \eqref{eq:2d:23c} and insert the
  resulting estimate into \eqref{eq:2d:21} to arrive at
  \begin{displaymath}
    \b| \Sa(y; T) - \pp W(\mF) \b| \leq \eps {\rm osc}(\D y; \omega_T^\a) 
    \sum_{r \in \Rg} \sum_{s \in \Rg} |r||s| M_{r,s}^\a
    \frac{\eps^2}{|T|} \sum_{x \in \Lext} \mint_{x}^{x+\eps r} \chi_{T} \db.
 \end{displaymath}
 An application of the bond density lemma, and referring to the
 definition of $M^\a$ in \eqref{eq:intro:defn_Ma}, yields the stated
 result.
\end{proof}

\subsection{The a/c stress function} \quad
\label{sec:2d:Sac}
We wish to derive a similar representation of $\del\Eac$ in terms of a
stress function $\Sac$, as we did in \S\ref{sec:2d:Sa} for
$\del\Ea$. A straightforward calculation along the same lines as the
proof of Proposition \ref{th:2d:Sa}, recalling first the definition of
the partial derivative $\pp_b \Fi$ from \S\ref{sec:intro:loc_scal},
yields the following result.

\begin{proposition}
  \label{th:2d:Sac}
  Suppose that \eqref{eq:intro:noint_La_Omc} holds, then, for all $y,
  z \in \Ys$,
  \begin{displaymath}
    \b\<\del\Eac(y), z \b\> 
    = \sum_{T \in \Teps} |T|\, \Sac(y; T) : \D z(T)
    = \int_\Om \Sac(y) : \D z \dx,
  \end{displaymath}
  where 
  \begin{align*}
    \Sac(y; T) := \cases{ 
      \Sa(y; T), & T \in \Tepsa, \\[2mm]
      \pp W(\D y(T)), & T \in
      \Tepsc, \\[2mm]
      \sum_{r \in \Rg} \frac{\eps^2}{|T|} 
      \sum_{x \in \Laext} \b[ \pp_r V\b(\Da{\Rg} y(x)\b) \otimes r \b]
      \mint_{x}^{x+\eps r} \chi_T \db\,\, & \\
      + \frac{\eps^2}{|T|} \sum_{(x, x+\eps r) \in \Biext} \b[\pp_{(x,x+\eps r)} \Fi(y)
      \otimes r\b] \mint_{x}^{x+\eps r} \chi_T^\i \db, & 
      T \in \Tepsi. }
 \end{align*}
 where $\chi_T^\i$ is a modified characteristic function:
 \begin{displaymath}
   \chi_T^\i(x) = \cases{
     1, & x \in \pp \Omi^\per, \\
     \chi_T(x), & \text{otherwise}.
   }
 \end{displaymath}
\end{proposition}
\begin{proof}
  Employing again the notation $V_{x, r} = \pp_r V(\Da{\Rg} y(x))$,
  the functional $\del\Eac(y)$ can be written as
  \begin{equation}
    \label{eq:2d:defnSac:10}
    \begin{split}
      \b\<\del\Eac(y), z \b\> =~& \int_{\Omc} \pp W(\D y) : \D z \dx
      + \eps^2 \sum_{x \in \La} \sum_{r \in \Rg} V_{x,r} \cdot \Da{r}
      u(x)  \\
      & + \eps^2 \sum_{(x, x+\eps r) \in \Bi} \pp_{(x, x+\eps r)} \Fi(y)
      \cdot \Da{r} u(x).
    \end{split}
  \end{equation}
  The first term on the right-hand side of \eqref{eq:2d:defnSac:10}
  gives rise to the definition of $\Sac(y; T)$ for $T \in
  \Tepsc$. (Note that \eqref{eq:intro:noint_La_Omc} guarantees that
  the bonds in the second group in \eqref{eq:2d:defnSac:10} do not
  contribute to $\Omc$.)

  After the same calculation as in the proof of Proposition
  \ref{th:2d:Sa}, the second group on the right-hand side of
  \eqref{eq:2d:defnSac:10} gives the definition of $\Sac(y; T)$ for $T
  \in \Tepsa$, as well as the first group in the definition of
  $\Sac(y; T)$ for $T \in \Tepsi$.

  The crucial modification to the previous argument is that the
  modified characteristic functions $\chi_T^\i$, $T \in \Tepsi$ form a
  partition of unity for $\Omi$ (except at a finite number of points,
  which do not contribute to bond integrals). Therefore, performing
  again a similar calculation as in the proof of Proposition
  \ref{th:2d:Sa} to ``distribute'' the third group on the right-hand
  side of \eqref{eq:2d:defnSac:10} between interface elements only, we
  obtain the second group in the definition of $\Sac(y; T)$ for $T \in
  \Tepsi$.

  Note that if we hadn't made the modification to the characteristic
  function, then $\Sac(y; T)$, $T \in \Tepsa\cup\Tepsc$ would contain
  contributions from $\Fi$.
\end{proof}

Since any discrete divergence-free tensor field may be added to $\Sac$
and still yield a valid representation of $\del\Eac$, it is not
surprising that, in general, $\Sac$ does not have the {\em necessary
  property} that $\Sac(\yF; T) = \pp W(\mF)$ for all $\mF \in \R^{2
  \times 2}$. This can already be observed in the nearest-neighbour,
flat interface constructions in \cite{OrtZha:2011a}. Hence, we need to
construct an alternative stress function $\Sacm$ representing
$\del\Eac$ that does have the desired properties. This construction
will be undertaken in the remainder of this section, using the
representation of discrete divergence-free vector fields as gradients
of Crouzeix--Raviart functions discussed in \S\ref{sec:aux:helmh}.

\paragraph{Consequences of global energy consistency} 
Recall from~\eqref{eq:intro:econs} that a functional $\E \in
\CC^1(\Ys)$ is called globally energy consistent if $\E(\yF) = \Ea(\yF)$ for
all matrices $\mF \in \R^{2 \times 2}$. The following lemma
establishes a simple but crucial consequence of this property.

\begin{lemma}
  \label{th:2d:econs}
  Suppose that $\E \in \CC^1(\Ys)$ is globally energy consistent, then
  \begin{displaymath}
    \< \del\E(\yF), \yG \> = |\Om|\, \pp W(\mF) : \mG
    \qquad \forall \, \mF, \mG \in \R^{2 \times 2}.
  \end{displaymath}
\end{lemma}
\begin{proof}
  From the assumption of global energy consistency, and
  \eqref{eq:intro:defnW}, we obtain that
  \begin{displaymath}
    \E(\yF) = \Ea(\yF) = |\Om| W(\mF) \qquad \forall \mF \in \R^{2
      \times 2}.
  \end{displaymath}
  Since $\yF + t \yG = y_{\mF + t \mG}$ this implies that
  \begin{displaymath}
    \b\< \del\E(\yF), \yG \b\> = |\Om| \lim_{t \to 0} \frac{W(\mF + t \mG)
      - W(\mF)}{t} = |\Om| \pp W(\mF) : \mG. \qedhere
  \end{displaymath}
\end{proof}

If we apply the foregoing lemma to an a/c functional $\Eac$ we obtain
the following corollary.

\begin{corollary}
  \label{th:2d:econs_cor}
  Suppose that $\Eac$ is globally energy consistent, then
  \begin{displaymath}
    \mint_{\Om} \Sac(\yF) \dx = \pp W(\mF)
    \qquad \forall \mF \in \R^{2 \times 2}.
  \end{displaymath}
\end{corollary}
\begin{proof}
  This result follows simply from Lemma \ref{th:2d:econs} and the fact
  that, for all $\mG \in \R^{2 \times 2}$,
  \begin{displaymath}
    \b\<\del\Eac(\yF), \yG\b\> = \int_\Om \Sac(\yF) : \mG \dx
    = \bg(\int_\Om \Sac(\yF) \dx \bg) : \mG. \qedhere
  \end{displaymath} 
\end{proof}

\paragraph{Consequences of patch test consistency} 
\label{sec:2d:cons_pt}
The examples given in \cite{OrtZha:2011a} show that patch test
consistency does not necessarily imply that $\Sac(\yF; T) = \pp
W(\mF)$. In the current paragraph we characterise the discrepancy
between $\Sac(\yF; T)$ and $\pp W(\mF)$.

First, we show that the test functions $u_h \in \Ush$ in the patch
test \eqref{eq:intro:patchtest} may be replaced by arbitrary
displacements $u \in \Us$.

\begin{lemma}
  \label{th:2d:pth_impl_pteps}
  Suppose that $\Eac$ is patch test consistent and that $\Tha \cup
  \Thi \subset \Teps$; then we also have
  \begin{equation}
    \label{eq:2d:pt_eps}
    \b\< \del\Eac(\yF), u \b\> = 0 \qquad \forall \, u \in \Us, \quad 
    \mF \in \R^{2 \times 2}.
  \end{equation}
\end{lemma}
\begin{proof}
  Fix $u \in \Us$; then, using the assumption that $\Tha \cup \Thi
  \subset \Teps$, we have
  \begin{displaymath}
    \int_{\Omc} \D u \dx = \int_{(\pp \Omc^\per)\cap\Om} u \otimes
    \nu \ds = \int_{(\pp \Omc^\per)\cap\Om} I_h u \otimes
    \nu \ds = \int_{\Omc} \D I_h u \dx.
  \end{displaymath}
  Since $\Sac(\yF) = \pp W(\mF)$ is constant in $\Omc$, and since $I_h
  u = u$ in $\Omi \cup \Oma$, we can therefore deduce that
  \begin{displaymath}
    \b\< \del\Eac(\yF), u \b\> = \int_{\Om} \Sac(\yF) : \D u \dx
    = \int_\Om \Sac(\yF) : \D I_h u \dx = \b\< \del\Eac(\yF), I_h u \b\> 
    =  0. 
  \end{displaymath}
  The penultimate equality requires some justification, but follows
  quite easily from the particular form of $\Eac$ assumed in
  \eqref{eq:intro:Eac} and the assumption that $\Tha \cup \Thi \subset
  \Teps$.
\end{proof}

\begin{lemma}
  \label{th:2d:cons_of_pt}
  Suppose that $\Eac$ is patch test consistent and globally energy consistent
  and that $\Thi\cup\Tha \subset \Teps$; then, for each $\mF \in \R^{2
    \times 2}$, there exists a function $\w(\mF; \cdot) \in
  \CRper(\Teps)^2$ such that
  \begin{displaymath}
    \Sac(\yF; T) = \pp W(\mF) + \D\w(\mF; T) \mJ \qquad \forall \, T
    \in \Teps,
  \end{displaymath}
  where $\mJ$ is a rotation matrix defined in Lemma
  \ref{th:aux:divfree}.  Moreover, if ${\rm int}(\Oma)$ is connected,
  then we may choose $\w(\mF) = 0$ in $\Oma$.
\end{lemma}
\begin{proof}
  If $\Eac$ is patch test consistent then, according to Lemma
  \ref{th:2d:pth_impl_pteps},
  \begin{displaymath}
    \int_\Om \Sac(\yF) : \D u \dx = \< \del\Eac(\yF), u \> = 0
    \qquad \forall u \in \Us.
  \end{displaymath}
  Hence, according to Lemma \ref{th:aux:divfree}, there exists a
  constant $\Sigma_0 \in \R^{2 \times 2}$, a vector-valued
  Crouzeix--Raviart function $\w = \w(\mF; \cdot) \in
  \CRper(\Teps)^2$, and a rotation matrix $\mJ$, such that
  \begin{displaymath}
    \Sac(\yF; T) = \Sigma_0 + \D \w(T) \mJ \qquad \forall T \in \Teps.
  \end{displaymath}

  Using global energy consistency of $\Eac$ and Corollary
  \ref{th:2d:econs_cor} we obtain that
  \begin{displaymath}
    \pp W(\mF) = \mint_\Om \Sac(\yF) \dx
    = \Sigma_0 + \mint_\Om \D \w\mJ \dx.
  \end{displaymath}
  If $\int_\Om \D \w \dx = 0$, then $\Sigma_0 = \pp W(\mF)$ and hence
  the result follows.

  To prove this, we integrate by parts separately in each element:
  \begin{align*}
    \int_\Om \D \w \dx =~& \sum_{T \in \Teps} \int_{\pp T} \w \otimes
    \nu \ds 
    = \sum_{f \in \Edg} \b( \w^+ \otimes \nu^+ + \w^- \otimes
    \nu^- \b) \ds = 0,
  \end{align*}
  where, in the last equality, we used the fact that $\int_f (\w^+ -
  \w^-) \ds = 0$ for all edges $f$, since $\w$ is continuous in the
  edge midpoints.

  Finally, since $\Sac(\yF; T) = \pp W(\mF)$ for all $T \in \Tepsa$
  (cf. Lemma \ref{th:2d:Sac} and Lemma \ref{th:2d:propSa}), it follows
  that $\D \w = 0$ in $\Oma$. Hence, if ${\rm int}(\Oma)$ is
  connected, then we can shift $\w$ by a constant so that $\w = 0$ in
  $\Oma$.
\end{proof}

\paragraph{The modified a/c stress function}
\label{sec:2d:defn_Sacm}
We wish to construct a modified a/c stress function $\Sacm$ that can
be used to represent $\del\Eac$, and satisfies the crucial property
that $\Sacm(\yF; T) = \pp W(\mF)$ for all $\mF \in \R^{2 \times 2}$,
$T \in \Tepsi$.

To this end, we generalize the Crouzeix--Raviart function $\w(\mF)$
defined in Lemma \ref{th:2d:cons_of_pt} to arbitrary deformations $y
\in \Ys$. Since we will use later on that the modified function
$\modw$ vanishes in $\Oma$, we require from now on that ${\rm
  int}(\Oma)$ is connected so that we can choose $\w(\mF, \cdot) = 0$
in $\Oma$ for all $\mF \in \R^{2 \times 2}$.

For each $y \in \Ys$ and each face $f \in \Edgext$, $f = T_1 \cap
T_2$, we define the patch $\omega_f = (T_1 \cup T_2) \setminus
\Oma^\per$, and the deformation gradient averages
\begin{displaymath}
  \mF_f(y) := \cases{
    \mint_{\omega_f} \D  y \dx, & \text{if } |\omega_f| > 0 \\
    0, & \text{otherwise.}
  }
\end{displaymath}
Note that $\omega_f$ was defined in such a way that $\omega_f \subset
\omega_{T_1} \cap \omega_{T_2}$. The value $\mF_f(y) = 0$ for $f
\subset \Oma^\per$ is of no importance, and could have been replaced
by any other value.

With this notation, and recalling the definitions of the edge
midpoints $q_f$ and the periodic nodal basis functions $\zeta_f^\per$
from \S\ref{sec:aux:CR_space}, we can define
\begin{equation}
  \label{eq:2d:defn_modwh}
  \modw(y; \cdot) = \sum_{f \in \Edg} \w\b( \mF_f(y); q_f \b) \zeta_f^\per.
\end{equation}
Note, in particular, that $\modw(\yF) = \w(\mF)$ for all $\mF \in
\R^{2 \times 2}$. It is therefore natural to define the modified
stress function
\begin{equation}
  \label{eq:2d:defn_Sacm}
  \Sacm(y; T) := \Sac(y; T) - \D \modw(y; T) \mJ, \qquad \text{for } T
  \in \Teps.
\end{equation}
In the following lemma we establish some elementary properties of
$\Sacm$.



\begin{lemma}
  \label{th:2d:prop_Sacm}
  Suppose that $\Eac$ is energy and patch test consistent, that ${\rm
    int}(\Oma)$ is connected, and that $\Tha\cup\Thi \subset
  \Teps$; then the modified a/c stress function $\Sacm$, defined in
  \eqref{eq:2d:defn_Sacm}, has the following properties:
  \begin{align}
    \label{eq:2d:prop_Sacm:delEac}
    \b\< \del\Eac(y), z \b\> =~& \int_\Om \Sacm(y) : \D z \dx \qquad
    \forall\, y, z \in \Ys; \\
    \label{eq:2d:prop_Sacm:TiTc}
    \Sacm(\yF; T) =~& \pp W(\mF) \qquad \forall\, \mF \in \R^{2 \times
      2}, \quad T \in \Tepsc \cup \Tepsi; \quad \text{and} \\
   \label{eq:2d:prop_Sacm:Ta}
    \Sacm(y; T) =~& \Sa(y; T) \qquad \forall\, y \in \Ys, \quad  T \in \Tepsa. 
 \end{align}
\end{lemma}
\begin{proof}
  To prove \eqref{eq:2d:prop_Sacm:delEac} let $z = \yB + u$ for
  some $\mB \in \R^{2 \times 2}$ and $u \in \Us$; then
  \begin{displaymath}
    \int_\Om \Sacm(y) : \D z \dx = \< \del\Eac(y), z \> -
    \int_\Om \b(\D \modw \mJ\b) : \b( \mB + \D u \b) \dx.
  \end{displaymath}
  Since $\int_\Om (\D \modw \mJ) : \D u \dx = 0$ by Lemma
  \ref{th:aux:divfree}, and since $\int_\Om \D \modw \dx = 0$ (see the
  proof of Lemma~\ref{th:2d:cons_of_pt}), the representation
  \eqref{eq:2d:prop_Sacm:delEac} follows.

  Property \eqref{eq:2d:prop_Sacm:TiTc} follows from Lemma
  \ref{th:2d:cons_of_pt} and the fact that $\modw(\yF) =
  \w(\mF)$:
  \begin{align*}
    \Sacm(\yF; T) = \Sac(\yF; T) - \D \w(\mF; T) \mJ  = \pp W(\mF)
    \quad \forall T \in \Teps.
  \end{align*}
  
  Property \eqref{eq:2d:prop_Sacm:Ta} follows from the fact that
  we constructed $\modw(y)$ to be zero in $\Oma$ for all $y \in
  \Ys$, and from Proposition \ref{th:2d:Sac}, which states that
  $\Sac(y; T) = \Sa(y; T)$ for all $T \in \Tepsa$.
\end{proof}

\subsection{The Lipschitz property}
\label{sec:2d:Lip_Sacm}
The final remaining ingredient for the proof of first-order
consistency, is a Lipschitz property for $\Sacm$, similar to the
Lipschitz property \eqref{eq:2d:Lip_Sa} of $\Sa$. In order to ensure
that there are no modelling error contributions from the atomistic
region it turns out to be most convenient to work directly with the
stress difference
\begin{equation}
  \label{eq:2d:defn:Rac}
  \Rac(y; T) := \Sacm(y; T) - \Sa(y; T).
\end{equation}
From \eqref{eq:2d:prop_Sacm:Ta} we immediately obtain that
\begin{equation}
  \label{eq:2d:Rac_Ta}
  \Rac(y; T) = 0 \qquad \forall \, T \in \Tepsa.
\end{equation}

In the remainder of the section we will estimate $\Rac(y; T)$ for $T
\in \Tepsc \cup \Tepsi$. To motivate the following result we note that,
from Lemma \ref{th:2d:propSa} and from \eqref{eq:2d:prop_Sacm:TiTc} we
see that
\begin{equation}
  \label{eq:2d:Rac_Tic}
  \Rac(\yF; T) = 0 \qquad \forall\, \mF \in \R^{2 \times 2}, \quad T \in
  \Tepsc \cup \Tepsi;
\end{equation}
hence, a suitable Lipschitz estimate for $\Rac$ yields the following
result.

\begin{lemma}
  \label{th:2d:main_estimate}
  Suppose that all conditions of Lemma \ref{th:2d:prop_Sacm} hold
  and, in addition, that $\Ei$ satisfies the locality and scaling
  conditions \eqref{eq:intro:Fi_loc} and \eqref{eq:intro:Fi_scal};
  then
  \begin{equation}
    \label{eq:3d:main_estimate}
    \b| \Rac(y; T) \b| \leq \eps M_T \, {\rm osc}(\D y; \omega_T)
    \qquad \forall y \in \Ys, \quad T \in \Teps,
  \end{equation}
  where $M_T$ is defined in \eqref{eq:2d:defn_barM}.
\end{lemma}

\medskip
The proof of this central lemma is split over the following
paragraphs.

\paragraph{Estimates in the continuum region} 
\label{sec:2d:Racest_Tc}
Let $T \in \Tepsc$, then 
\begin{align}
  \notag
  \b|\Rac(y; T)\b| =~& \B|\b[\Sac(y; T) - \D \modw(y; T) \mJ\b] -
  \Sa(y; T) \B| \\
  \notag
  \leq~& \b| \pp W(\D y) - \Sa(y; T) \b| + \b| \D \modw(y;
  T) \mJ \b| \\
  \label{eq:2d:Rac1_est_Tc}
  \leq~&  \eps M^\a {\rm osc}\b(\D y; \omega_T\b) + \b| \D \modw(y;
  T) \b|, 
\end{align}
where, in the last inequality, we used \eqref{eq:2d:Lip_Sa} and the
fact that $\omega_T^\a \subset \omega_T$ for $T \in \Tepsc$. We still
need to estimate $\D \modw(y; T)$, which we postpone until
\S\ref{sec:2d:Est_wh}.

\paragraph{Estimates in the interface region}
\label{sec:2d:Racest_Ti}
Let $T \in \Tepsi$, and let $V_{x, r} = \pp_r V(\Da{\Rg} y(x))$, then
\begin{align*}
  \Rac(y; T) =~& \Sac(y; T) - \Sa(y; T) - \D \modw(y; T) \mJ \\
  =~& \sum_{r \in \Rg} \frac{\eps^2}{|T|} \sum_{x \in \Laext} \B[
  V_{x, r} \otimes r \B] \mint_{x}^{x+\eps r} \chi_T \db 
  + \frac{\eps^2}{|T|} \sum_{\substack{b \in \Biext \\ b = (x, x+\eps
      r)}} \B[\pp_b \Fi(y) \otimes r \B] \mint_{x}^{x+\eps r} \chi_T^\i
  \db \\
  & - \sum_{r \in \Rg} \frac{\eps^2}{|T|} \sum_{x \in \Lext} \B[V_{x, r}
  \otimes r \B] \mint_{x}^{x+\eps r} \chi_T \db  
  - \D \modw(y; T) \mJ,
\end{align*}
which, after combining the first and third group, becomes
\begin{align*}
  \Rac(y; T) =~& \frac{\eps^2}{|T|} \sum_{\substack{b \in \Biext \\ b = (x, x+\eps
      r)}} \B[\pp_b \Fi(y) \otimes r \B] \mint_{x}^{x+\eps r} \chi_T^\i 
  \db  \\
  & - \sum_{r \in \Rg} \frac{\eps^2}{|T|} \sum_{x \in \Lext \setminus
    \Laext} \B[V_{x, r} \otimes r \B] \mint_{x}^{x+\eps r} \chi_T \db
  - \D \modw(y; T) \mJ \\
  =:~& \Rac^{(1)}(y; T) - \Rac^{(2)}(y; T) - \D \modw(y; T) \mJ. \\
\end{align*}
We will again postpone the estimation of $\D \modw(T)$ to
\S\ref{sec:2d:Est_wh}, and focus on the terms $\Rac^{(1)}(y; T)$ and
$\Rac^{(2)}(y; T)$.

Let $\mF = \D y(T)$. Using the {\em locality and scaling conditions}
\eqref{eq:intro:Fi_loc} and \eqref{eq:intro:Fi_scal}, we can estimate
\begin{align}
  \notag
  \b| \Rac^{(1)}(y; T) - \Rac^{(1)}(\yF; T) \b| \leq~&
  \frac{\eps^2}{|T|} \sum_{\substack{b \in \Biext \\ b = (x, x+\eps r)}}
  \B| \pp_b \Fi(y) - \pp_b \Fi(\yF) \B| |r| \mint_{x}^{x+\eps r}
  \chi_T^\i \db \\
  \label{eq:2d:Rac1est_Ti_pre}
  \leq~& \frac{\eps^2}{|T|} \sum_{\substack{b \in \Biext \\ b = (x,
      x+\eps r)}}
  \sum_{\substack{s \in \Rg \\ (x, x+\eps s) \in \Biext}} M^\i_{r,s}
  \b| \Da{s} y(x) - \mF s \b| |r| \mint_{x}^{x+\eps r}
  \chi_T \db.
\end{align}
In the transition from the first to the second line we have used the
fact that on bonds that lie on the boundary of $\Omi^\per$ the
constant $M_{r,s}^\i$ is replaced by $\smfrac12 M_{r,s}^\i$, which
effectively replaces $\chi_T^\i$ by $\chi_T$. Bounding $| \Da{s} y(x)
- \mF s |$ by the local oscillation, and applying the bond density
lemma (see Lemma \ref{th:2d:propSa} for a similar calculation), we
deduce that
\begin{equation}
  \label{eq:2d:Rac1est_Ti}
  \b| \Rac^{(1)}(y; T) - \Rac^{(1)}(\yF; T) \b|
  \leq \eps \sum_{r \in \Rg} \sum_{s \in \Rg} |r| |s| M_{r,s}^\i
  \, {\rm osc}(\D y; \omega_T)
  = \eps M^\i {\rm osc}(\D y; \omega_T).
\end{equation}

Following closely the proof of \eqref{eq:2d:Lip_Sa}, we obtain a
similar estimate for $\Rac^{(2)}$:
\begin{equation}
  \label{eq:2d:Rac2est_Ti}
  \b| \Rac^{(2)}(y; T) - \Rac^{(2)}(\yF; T) \b|
  \leq \eps \sum_{r \in \Rg} \sum_{s \in \Rg} |r| |s| M_{r,s}^\a
  \, {\rm osc}(\D y; \omega_T)
  = \eps M^\a {\rm osc}(\D y; \omega_T).  
\end{equation}
Note that it is enough to measure the oscillation over $\omega_T$
(which does not intersect with $\Oma$), since $(x, x + \eps s) \cap
\Oma = \emptyset$ for all $x \in \Lext \setminus \Laext, s \in \Rg$.

Combining \eqref{eq:2d:Rac1est_Ti} and \eqref{eq:2d:Rac2est_Ti}, and
using the fact that $\Rac(\yF; T) = 0$, we conclude that
\begin{align}
  \notag 
  \b| \Rac(y; T) \b| =~& \b| \Rac(y; T) - \Rac(\yF; T) \b| \\
  %
  \notag
  \leq~& \b| \Rac^{(1)}(y; T) - \Rac^{(1)}(\yF; T) \b|
  + \b| \Rac^{(2)}(y; T) - \Rac^{(2)}(\yF; T) \b| 
  \notag
   + \b| \D \modw(y; T) - \D \modw(\yF; T) \b|  \\
  \label{eq:2d:Racest_Ti_pre:c}
  \leq~& \eps \b( M^\i + M^\a \b) {\rm osc}(\D y; \omega_T)
  + \b| \D \modw(y; T) - \D \w(\mF; T) \b|.
\end{align}

\paragraph{Estimates on $\w$ and on $\modw$}
\label{sec:2d:Est_wh}
To finalize the estimates in \S\ref{sec:2d:Racest_Tc} and
\S\ref{sec:2d:Racest_Ti} we are left to establish a Lipschitz property
for $\modw$.  The following result is a fundamental technical lemma
that will allow us to achieve this. Its proof is deceptively simple,
however, it uses implicitly many of the foregoing
calculations. Moreover, some questions left open by Theorem
\ref{th:2d:main_result} may be answered through a better understanding
of this step.

\begin{lemma}
  \label{th:2d:Lipest_wh}
  Suppose that the conditions of Lemma \ref{th:2d:main_estimate}
  hold; then, for all $f \in \Edgext, f \subset \Omc \cup \Omi$, and
  for all $\mF, \mG \in \R^{2 \times 2}$, we have
  \begin{displaymath}
    \b| \w(\mF; q_f) - \w(\mG; q_f)\b| \leq 
    \eps \b(M_\a + M_\i\b) {\rm width}(\Omi) \b| \mF - \mG \b|,
 \end{displaymath}
 where ${\rm width}(\Omi)$ is defined in \eqref{eq:2d:defn_widthOmi}.
\end{lemma}
\begin{proof}
  Fix some $f' \in \Edgext$, $f' \subset \Oma$; then, for any
  connecting path $\gamma \in \Gamma_{f',f}$ we have
  \begin{align*}
     \w(\mF; q_f) - \w(\mG; q_f) =~& \int_{\gamma} \B( \D \w(\mF) - \D
     \w(\mG) \B) \cdot \dx \\
     =~& \int_{\gamma} \B(
     \b[ \pp W(\mF) - \Sac(\yF) \b] - \b[\pp W(\mG) - \Sac(\yG) \b]
     \B) \cdot \dx.
  \end{align*}
  Since $\Sac(\yB; T) = \pp W(\mB)$ for all $T \in \Tepsc \cup \Tepsa$,
  $\mB \in \R^{2 \times 2}$, the integrand vanishes in $\Oma \cup
  \Omc$. Hence, it follows that
  \begin{displaymath}
    \b| \w(\mF; q_f) - \w(\mG; q_f) \b| \leq
    {\rm length}(\gamma \cap \Omi)  \max_{T \in \Tepsi} 
    \B| \b[\Sac(\yF; T)  - \Sac(\yG; T)\b] - \b[ \pp W(\mF) - \pp W(\mG)
    \b] \B|.
  \end{displaymath}
  Following closely the calculations in Section \ref{sec:2d:Racest_Ti}
  we can deduce that 
  \begin{displaymath}
    \b| \w(\mF; q_f) - \w(\mG; q_f) \b| \leq
    {\rm length}(\gamma \cap \Omi) \,
    \b(M_\a + M_\i\b) \,  \b| \mF - \mG \b|.
  \end{displaymath}

  Since we are free to choose the path $\gamma$, we can choose it so
  that ${\rm length}(\gamma \cap \Omi)$ is minimized, which yields the
  stated result.
\end{proof}

Let $T \in \Tepsi \cup \Tepsc$, let $\mF = \D y(T)$, and recall that
$\zeta_f$ are the Crouzeix--Raviart nodal basis functions associated
with edge midpoints $q_f$; then, using Lemma \ref{th:2d:Lipest_wh} we obtain
\begin{align*}
  \b| \D \modw(y; T) - \D \w(\mF; T) \b|
  \leq~& \sum_{\substack{f \in \Edgext \\ f \subset \pp T}} 
  \b| \w( \mF_f(y); q_f) - \w(\mF; q_f) \b|
  \b| \D \zeta_f(T) \b| \\
  \leq~& \b(M_\a + M_\i\b) \, {\rm width}(\Omi) \B\{ \max_{\substack{f \in
      \Edgext \\ f\subset \pp T}} \b| \mF_f(y) - \mF \b| \B\}
  \, \bg\{ \eps \sum_{\substack{f \in \Edgext \\ f\subset \pp T}} \b| \D
  \zeta_f(T) \b| \bg\}.
\end{align*}
A direct calculation yields
\begin{displaymath}
   \eps \sum_{\substack{f \in \Edgext \\ f\subset \pp T}} \b| \D
  \zeta_f(T) \b| = 2 + 2 + 2 \sqrt{2} \leq 7.
\end{displaymath}
From the definitions of $\mF_f(y)$ and $\mF$ it follows that
\begin{equation}
  \label{eq:2d:Lipest_modwh}
  \b| \D \modw(y; T) - \D \w(\mF; T) \b| \leq \eps
  7 \b(M_\a + M_\i\b)  {\rm width}(\Omi) \, {\rm osc}(\D y; \omega_T)
  \quad \forall T \in \Tepsi \cup \Tepsc.
\end{equation}

\medskip
\begin{proof}[Proof of Lemma \ref{th:2d:main_estimate}]
  Combining \eqref{eq:2d:Lipest_modwh} with \eqref{eq:2d:Rac1_est_Tc}
  and noting that $\D \w(\mF) = 0$ for $T \in \Tepsc$
  (cf. \eqref{eq:2d:prop_Sacm:TiTc} and the fact that $\Sac(\yF) = \pp
  W(\mF)$ in $\Omc$), and also combining \eqref{eq:2d:Lipest_modwh}
  with \eqref{eq:2d:Racest_Ti_pre:c}, we finally arrive at the result
  of Lemma \ref{th:2d:main_estimate}.
\end{proof}

\subsection{Remarks on the conditions of Theorem \ref{th:2d:main_result}} \quad
\label{sec:2d:remarks}
In this section, we construct simple examples to discuss the various
assumptions of Theorem \ref{th:2d:main_result}. We will show that most
assumptions are also necessary.

\paragraph{Technical conditions}
The assumption that $\Tha \cup \Thi \subset \Teps$, and the assumption
\eqref{eq:intro:noint_La_Omc}, were made for the sake of convenience
of the analysis and simplicity of presentation. Dropping this
assumption is not straightforward, but it is reasonable to expect that
a careful analysis should allow to do so.

The same statement applies to the assumptions made on the interaction
potential; this was already discussed in \S\ref{sec:intro:assV}.

\paragraph{Connectedness of $\Oma$}
\label{sec:2d:connOma}
The assumption that ${\rm int}(\Oma)$ is connected is more
problematic; at this point it is unclear whether or not it can be
removed in general. A more detailed analysis of the functions $\w(\mF,
\cdot)$, $\mF \in \R^{2 \times 2}$, defined in \S\ref{sec:2d:cons_pt}
is required to understand this issue. There are, however, at least two
special cases where one can attempt to remove it with relatively
little effort:
\begin{itemize}
\item {\it Well separated components: } If $\Oma$ has several
  connected components, which are separated by an $O(1)$ distance,
  then one can localize the consistency error estimate to each of the
  components and obtain a qualitatively similar result as Theorem
  \ref{th:2d:main_result}.
\item {\it Specific a/c methods: } Suppose that the GCC method
  described in \S\ref{sec:intro:int_corr} is used to construct a patch
  test consistent coupling scheme, with parameters $C_{x, r,
    s}$. Suppose, moreover, that $\Oma$ has two connected components,
  $\Om_1$ and $\Om_2$, each of which have a portion of the boundary
  with the same orientation (say, normal $e_1$), as displayed in
  Figure \ref{fig:two_components}.

  \begin{figure}
    \includegraphics[height=4cm]{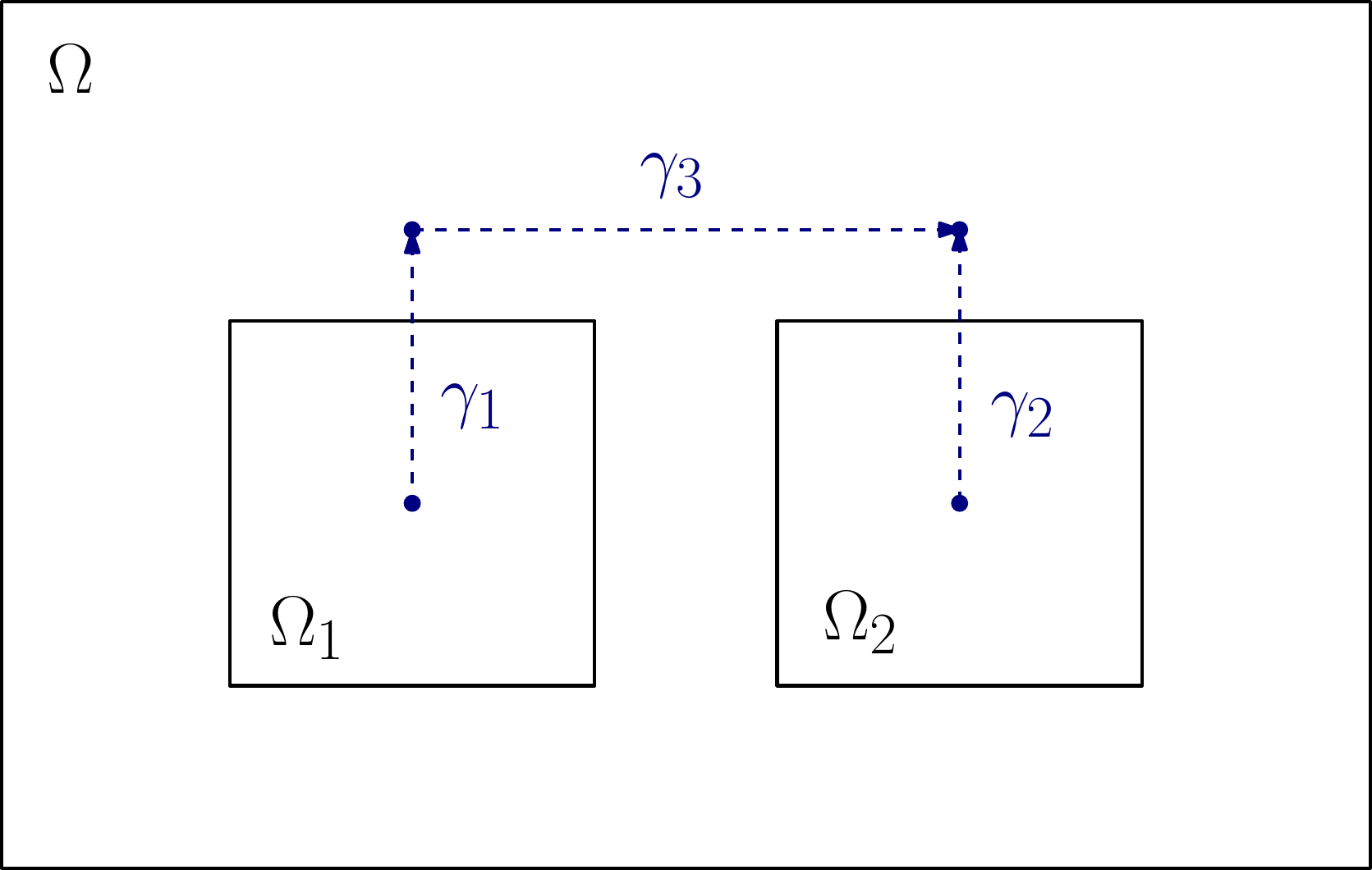}
    \caption{\label{fig:two_components} Atomistic region with two
      components to visualise the argument given in
      \S\ref{sec:2d:connOma}.}
  \end{figure}
  
  It is then reasonable to assume that the parameters $C_{x,r,s}$ have
  the same value in those parts of the interface surrounding $\Om_1$
  and $\Om_2$, which would imply that
  \begin{displaymath}
    \int_{\gamma_1} \Sac(\yF) \cdot \dx = \int_{\gamma_2} \Sac(\yF) \cdot \dx,
  \end{displaymath}
  Moreover, since $\Sac(\yF) = \pp W(\mF)$ in the continuum region, we
  would obtain that
  \begin{displaymath}
    \int_{\gamma_1 \cup \gamma_3 \cup \gamma_2'} \b( \Sac(\yF) - \pp
    W(\mF) \b) \cdot \dx = 0,
  \end{displaymath}
  where $\gamma_2'$ denotes the curve $\gamma_2$ with reversed
  orientation. 

  This shows that it is possible to choose $\psi(\mF; \cdot) = 0$ in
  both components of $\Oma$, and as a consequence, Theorem
  \ref{th:2d:main_result} would remain true.
\end{itemize}

A related issue is the dependence of the modelling error estimate
\eqref{eq:2d:main_result} on ${\rm width}(\Omi)$, which comes solely
from the Lipschitz estimate on $\mF \mapsto \w(\mF; \cdot)$;
cf. \S\ref{sec:2d:Est_wh}. Hence, a better understanding of this
function may also allow a finer analysis of this undesirable
dependence.

\paragraph{The global energy consistency condition}
\label{sec:2d:rem_econs}
Global energy consistency is a natural and convenient condition that yields
the important intermediate result (Corollary \ref{th:2d:econs_cor})
that
\begin{equation}
  \label{eq:2d:cons_econs_disc}
  \mint_\Om \Sac(\yF) \dx = \pp W(\mF) \qquad \forall \, \mF \in \R^{2
    \times 2}.
\end{equation}
Note also that \eqref{eq:2d:cons_econs_disc} implies $\Eac(\yF) =
\Ea(\yF) + c$ for all $\mF \in \R^{2 \times 2}$, where $c$ is a fixed
constant that is independent of $\mF$; that is,
\eqref{eq:2d:cons_econs_disc} is practically equivalent to global
energy consistency.

In some important situations patch test consistency already implies
\eqref{eq:2d:cons_econs_disc}. The following result gives such a
result for finite atomistic regions.

\begin{proposition}
  Suppose that $\Oma \cup \Omi \subset {\rm int}(\Om)$; then patch
  test consistency \eqref{eq:intro:patchtest} of $\Eac$ implies
  \eqref{eq:2d:cons_econs_disc}.
\end{proposition}
\begin{proof}
  According to Lemmas~\ref{th:2d:Sac} and \ref{th:2d:cons_of_pt} we
  have
  \begin{displaymath}
    \b\< \del\Eac(\yF), u \b\> = \int_{\Omc} \pp W(\mF) : \D u \dx
    + \int_{\Oma \cup \Omi} \Sac(\yF) : \D u \dx.
 \end{displaymath}
 Let $\mG \in \R^{2 \times 2}$ and choose any $u \in \Us$ such that
 $\D u = \mG$ in $\Omi \cup \Oma$; this is possible due to the
 assumption that $\Oma \cup \Omi \subset {\rm int}(\Om)$. Integrating
 by parts twice, letting $\nu$ denote the unit outward normal to $\Omi
 \cup \Oma$, and noting that the portions of the surface integrals
 along $\pp \Om$ cancel each other out, yields
 \begin{align*}
   \b\< \del\Eac(\yF), u \b\> =~& \int_{\Omc} \pp W(\mF) : \D u \dx
   + \int_{\Oma \cup \Omi} \Sac(\yF) : \mG \dx \\ 
   =~& - \int_{\pp (\Oma \cup \Omi)} \pp W(\mF) : (u \otimes \nu) \ds
   + \int_{\Oma \cup \Omi} \Sac(\yF) : \mG \ds \\
   =~& - \int_{\Oma \cup \Omi} \pp W(\mF) : \D u \dx + \int_{\Oma \cup
     \Omi} \Sac(\yF) : \mG \ds \\
   =~& \int_{\Oma \cup \Omi} \b[ \Sac(\yF) - \pp W(\mF)\b] : \mG \dx.
 \end{align*}
 Since $\Eac$ is patch test consistent, the last term vanishes, and
 hence the result follows.
\end{proof}

It turned out to be difficult to devise a counterexample, which
clearly demonstrates that absence \eqref{eq:2d:cons_econs_disc} can
yield an inconsistent method. A more thorough investigation of this
condition is still required.

\paragraph{The locality condition}
\label{sec:2d:rem_locality}
\def\Ilp{\mathscr{L}^\i_{-, +}}
\def\Ilm{\mathscr{L}^\i_{-, -}}
\def\Irp{\mathscr{L}^\i_{+, +}}
\def\Irm{\mathscr{L}^\i_{+, -}}
To show that the locality condition \eqref{eq:intro:Fi_loc} (or a
variant thereof) is necessary we assume, without loss of generality,
that $N$ is even and define a functional $\mathscr{J} \in \CC^2(\Ys)$,
$\mathscr{J} = \eps^2 J$, 
\begin{displaymath}
  J(y) = 
  \bg|\sum_{x \in \Ilp} \!\!\Da{e_1} y(x) \bg|^2
  + \bg| \sum_{x \in \Irp} \!\!\Da{e_1} y(x) \bg|^2
  - \bg| \sum_{x \in \Ilm} \!\!\Da{e_1} y(x) \bg|^2
  - \bg| \sum_{x \in \Irm} \!\!\Da{e_1} y(x) \bg|^2,
\end{displaymath}
where, 
\begin{align*}
  \mathscr{L}^\i_{-, \pm} =~& \b\{ x \in \L \bsep x_1 \leq 0, x_2
  = \pm 1/2 \b\},  \quad \text{and} \\
  \mathscr{L}^\i_{+, \pm} =~& \b\{ x \in \L \bsep x_1 > 0, x_2 =
  \pm 1/2 \b\}.
\end{align*}

From the definition it is obvious that $J(\yF) = 0$ for all $\mF \in
\R^{2 \times 2}$. Moreover, using summation by parts along the two
lines $\Ilp \cup \Irp$ and $\Ilm \cup \Irm$, it is easy to check that
the patch test \eqref{eq:intro:patchtest} holds. Finally, $\mathscr{J}$
satisfies the scaling condition,
\begin{displaymath}
  \pp_{(x, x+\eps r)} \pp_{(x', x' + \eps r)} J(y) = \cases{2\, \mI,
    & \text{if } r = e_1 \text{ and } 
    x, x' \in \mathscr{L}^\i_{a, b},\quad a,b \in\{+,-\}, \\
    0, & \text{otherwise}.
  }
\end{displaymath}
However, $\mathscr{J}$ clearly violates the locality condition.

We may think of $\mathscr{J}$ as an a/c functional for $\Ea = 0$, or,
alternatively, as an additional contribution that can be added to any
a/c functional whose interface satisfies $\Ilp \cup \Ilm \cup \Irp
\cup \Irm \subset {\rm int}(\Omi^\per)$.

Let $y \in \YsA$ be ``smooth'' but not affine in the upper half plane
$\{ x \in \L \sep x_2 \geq 0 \}$, and let $y = \yA$ on the lower
interface $\Ilm \cup \Irm$; then, testing $\del \mathscr{J}(y)$ with
the unique displacement $u \in \Us$ such that
\begin{align*}
  u(x) = y(x) - \mA x &\qquad \text{for } x \in \{ (0, 1/2), (1, 1/2) \}; \\
   x_1 \mapsto u(x_1, 1/2) &\qquad \text{is affine in $[-1, 0]$ and in
     $[0, 1]$; and} \\
  \Da{e_2} u(x) = 0 &\qquad \text{for all } x \in \L;
\end{align*}
then we obtain, after a brief computation,
\begin{displaymath}
  \label{eq:2d:rem_loc:mainest}
  \b\| \del \mathscr{J}(y) \b\|_{\WW^{-1,2}_\eps} \geq \frac{\b\< \del \mathscr{J} (y), u
    \b\>}{\| \D u \|_{\LL^2(\Om)}} = \bg(\B|\eps\!\! \sum_{x \in \Ilp} \!\!\b(\Da{e_1}
  y(x) - \mA e_1\b) \B|^2
  + \B|\eps \!\!\sum_{x \in \Irp} \!\!\b(\Da{e_1}
  y(x) - \mA e_1\b) \B|^2\bg)^{1/2}.
\end{displaymath}
This final estimate is scaled like a surface integral and is clearly a
zeroth order term if $y$ is smooth but $y(0, 1/2) \neq y(1,
1/2)$. This shows that the locality condition \eqref{eq:intro:Fi_loc}
(or a variant thereof) is indeed necessary to obtain a first-order
consistency estimate.

\paragraph{The scaling condition}
\label{sec:2d:rem_scal}
\def\Ip{\mathscr{L}^\i_+}
\def\Im{\mathscr{L}^\i_-}
\def\Ipm{\mathscr{L}^\i_\pm}
It is fairly clear that the modelling error estimate can be
arbitrarily large without the scaling condition
\eqref{eq:intro:Fi_scal}. We nevertheless briefly discuss a simple
example with a natural interpretation.

Using a similar argument as in the previous paragraph, we define a
functional $\mathscr{J} \in \CC^2(\Ys)$, $\mathscr{J} = \eps^2 J$,
\begin{displaymath}
  J(y) = \beta \sum_{x \in \Ip} \b| \Da{e_1} y(x) \b|^2 - \beta
  \sum_{x \in \Im} \b|\Da{e_1} y(x) \b|^2,
\end{displaymath}
where $\beta > 0$ is a constant, and where
\begin{displaymath}
  \Ipm = \b\{ x \in \L \bsep x_2 = \pm 1/2 \b\}.
\end{displaymath}
It is easy to see that $J$ is patch test consistent, that $J(\yF)= 0$
for all $\mF \in \R^{2 \times 2}$, and that it satisfies the locality
condition.

Let $y \in \Ys$ such that $y(x) = \mA x$ for $x_2 = -1/2$, then
testing $\del \mathscr{J}(y)$ with the unique displacement $u \in
\Us$ such that
\begin{align*}
  u(x) = y(x) - \mA x, & \quad \text{for } x \in \Ip, \\
  \Da{e_2} u(x) = 0, & \quad \text{for all } x \in \L,
\end{align*}
we obtain
\begin{equation}
  \label{eq:2d:rem_scal:mainest}
  \b\| \del \mathscr{J}(y) \b\|_{\WW^{-1,2}_\eps} \geq \frac{\b\< \del \mathscr{J}(y), u
    \b\>}{\| \D u \|_{\LL^2}} = \beta \eps \bg[ \eps \sum_{x \in \Ip}
  \b| \Da{e_1} y(x) - \mA e_1 \b|^2 \bg]^{1/2}.
\end{equation}
If $y$ is ``smooth'' but not affine, then the term in square brackets
is of the order $O(1)$. By choosing $\beta$ arbitrarily large, the
modelling error can be made arbitrarily large as well. In particular, the
choice $\beta = 1/\eps$ would give a seemingly natural surface scaling
to the interface functional, and in this case we would obtain an
$O(1)$ modelling error.








\section{Conclusion}
\label{sec:conclusion}
A fairly complete consistency analysis of general patch test
consistent a/c coupling methods in (one and) two space dimensions was
developed in this paper. The main result is the first order modelling
error estimate, Theorem \ref{th:2d:main_result}. The main undesirable
condition is the assumption that ${\rm int}(\Oma)$ is connected. To
remove this assumption a finer analysis of the corrector functions
$\w(\mF, \cdot)$ defined in \S\ref{sec:2d:cons_pt} is required. At
this point one cannot exclude the possibility that a/c methods exist
for which this assumption is in fact necessary.

Many open problems remain to be answered. First and foremost, one
ought to answer the question whether a/c methods satisfying all the
conditions of Theorem \ref{th:2d:main_result} always exist. In
\cite{OrtZha:2011b}, we present a general construction (a variant on
the geometrically consistent coupling method \cite{E:2006}) that
appears to work in practise, however, we have no proof of this fact so
far. Indeed, if it should turn out that in certain cases the ``ghost
forces'' cannot be completely removed, then an extension of Theorem
\ref{th:2d:main_result} estimating the contribution of the ``ghost
force'' to the modelling error is highly desirable since such a result
would provide the correct quantity that needs to be minimized. It is
by no means clear that minimizing the ``ghost force'' itself is the
best possible target. A similar analysis would also be useful for
estimating the modelling error of blending methods \cite{BvK:blend1d}.

It should be conceptually straightforward, though technically more
demanding, to generalize all results to higher order finite element
methods in the continuum region, however, it would then also be
desirable to obtain the second order modelling error estimate in the
continuum region. Such a result seems difficult to obtain without a
more detailed understanding of the corrector functions $\w(\mF,
\cdot)$.

An immediate question is whether a variant of the main result is still
valid in 3D. This is by no means clear at this point. From a technical
point of view, we require generalizations of the two main technical
tools: the bond density lemma (\S\ref{sec:aux:bdl}) and the
characterisation of discrete divergence-free $\PO$-tensor fields
(\S\ref{sec:aux:helmh}). While the bond density lemma as stated in
this paper is false in 3D, one can establish variants that are
potentially usefull for a 3D analysis (work in progress). Generalising
the explicit construction of \S\ref{sec:aux:helmh} is  entirely open
at this point.

Another important and difficult question is the extension to
multi-lattices where the Cauchy--Born model is obtained through a
homogenization procedure \cite{AbLiSh, DoElLuTa:2007}.

Finally, it should be stressed, that Theorem \ref{th:2d:main_result}
is a general abstract result, and as such can undoubtedly be improved
upon when a specific coupling method is analyzed. It may be possible
for specific methods to obtain more information about the corrector
functions $\psi(\mF, \cdot)$, and hence obtain a better estimate on
the dependence of the modelling error on the interface width. For
example, the consistency analysis in \cite{OrtShap:2011a} requires no
corrector functions at all, and in the consistency analysis of
nearest-neighbour interactions \cite{OrtZha:2011b} the corrector
function vanishes in the continuum region. The proof of Theorem
\ref{th:2d:main_result} may, however, serve as a general guidance for
modelling error estimates in specific cases.

Finally, the stability of a/c methods in 2D/3D is largely open at this
point.

\section*{Acknowledgements}
I thank B.~Langwallner, X.~H.~Li, M.~Luskin, E.~S\"{u}li, A.~Shapeev,
and L.~Zhang for their comments on a draft of this manuscript, which
have greatly helped to improve its quality. E. S\"{u}li pointed out to
me the literature on the patch test. Early sketches of some of the
technical results presented in \S\ref{sec:fram} were developed in
discussion with A. Shapeev during our work on
\cite{OrtShap:2011a}. The representation of discrete divergence-free
$\PO$-tensor fields discussed in \S\ref{sec:aux:helmh} was pointed out
to me by L. Zhang. We have used a variant in our explicit construction
of consistent a/c methods in \cite{OrtZha:2011a}.


\appendix

\section{Proofs of \S \ref{sec:fram}}
\label{sec:app_proofs}

\begin{proof}[Proof of Lemma \ref{th:fram:Dbarvh_Dvh}]
  For $d = 1$, since $I_\eps y_h = y_h$, the result is trivial; hence
  assume that $d = 2$. Assume also that $p < \infty$. Since the norms
  involved are effectively weighted $\ell^p$-norms, one can obtain the
  case $p = \infty$ as the limit $p \nearrow \infty$.

  In this proof we will in fact use a periodic version of
  \eqref{eq:aux:bdl}, which is a simple consequence of
  \eqref{eq:aux:bdl} (see also \cite{OrtShap:2011a}):
  \begin{displaymath}
    |T| = \eps^2 \sum_{x \in \L} \mint_{x}^{x + \eps r} 
    \chi_{T^\per} \db.
  \end{displaymath}

  With our definition of the $\LL^p$-norms for matrix-valued
  functions, we have
  \begin{displaymath}
    \b\| \D I_\eps y_h \b\|_{\LL^p(\Om)}^p 
    = \sum_{j = 1}^2 \int_\Om \b| \D I_\eps y_h e_j \b|_p^p \dx
    = \sum_{j = 1}^2 \int_\Om \b| \Dc{e_j} I_\eps y_h \b|_p^p \dx.
  \end{displaymath}

  Using the periodic bond density lemma, and the fact that
  $\{\chi_{T^\per} \sep T \in \Teps \}$ is a partition of unity for
  $\R^2$, we have
  \begin{align}
    \notag
    \int_\Om \b| \Dc{e_j} I_\eps y_h \b|_p^p \dx =~& \sum_{T \in \Teps}
    |T| \b| \Dc{e_j} I_\eps y_h(T) \b|_p^p \\
    \notag
    =~& \sum_{T \in \Teps}  \b| \Dc{e_j} I_\eps y_h(T) \b|_p^p
    \eps^2 \sum_{x \in \L} \mint_{x}^{x + \eps e_j} \chi_{T^\per} \db \\
    \notag
    =~& \eps^2 \sum_{x \in \L} \sum_{T \in \Teps}
    \mint_{x}^{x + \eps e_j}  \b| \Dc{e_j} I_\eps y_h \b|_p^p \chi_{T^\per} \db.
    \\
    \label{eq:aux:Dbarvh_Dvh:15}
    =~& \eps^2 \sum_{x \in \L} \mint_{x}^{x + \eps e_j} 
    \b| \Dc{e_j} I_\eps y_h \b|_p^p \db.
 \end{align}
  We have also used the fact that $\Dc{e_j} I_\eps y_h$ is
  continuous across edges that have direction $e_j$.

  Due to the specific choice of the triangulation $\Teps$ it follows
  that $\Dc{e_j} I_\eps y_h$ is constant along each bond $(x, x+\eps
  e_j)$, and hence
  \begin{displaymath}
    \mint_{x}^{x + \eps e_j} 
    \b| \Dc{e_j} I_\eps y_h \b|_p^p \db = \b| \Da{e_j} I_\eps y_h(x) \b|_p^p
    = \bg| \mint_{x}^{x+\eps e_j} \Dc{e_j} y_h \db \bg|_p^p
    \leq \mint_{x}^{x+\eps e_j} \b| \Dc{e_j} y_h \b|_p^p \db,
  \end{displaymath}
  where we employed Jensen's inequality in the last step.

  Inserting this estimate into \eqref{eq:aux:Dbarvh_Dvh:15}, and
  reversing the argument in \eqref{eq:aux:Dbarvh_Dvh:15}, we arrive at
  \begin{align*}
    \int_\Om \b| \Dc{e_j} I_\eps y_h \b|_p^p \dx 
    \leq~& \eps^2 \sum_{x \in \L} \mint_{x}^{x+\eps e_j} \b| \Dc{e_j}
    y_h \b|_p^p \db \\
    =~& \sum_{T \in \Th} \eps^2 \sum_{x \in \L} \mint_x^{x+\eps e_j}
    \b| \Dc{e_j} y_h \b|_p^p \chi_{T^\per} \db \\
    =~& \sum_{T \in \Th} |T| \,  \b| \Dc{e_j} y_h(T) \b|_p^p = \b\|
    \Dc{e_j} y_h \b\|_{\LL^p(\Om)}^p. \qedhere
  \end{align*}
\end{proof}

\begin{remark}
  \label{rem:app:cts_discr_norms}
  From the foregoing proof, it follows that
  \begin{displaymath}
    \b\| \D y \b\|_{\LL^p(\Om)} = \bg(\eps^2 \sum_{j = 1}^2 \sum_{x \in
      \L} \b| \Da{e_j} y(x) \b|_p^p \bg)^{1/p}
    \qquad \text{for } y \in \Ys. \qedhere
  \end{displaymath}
\end{remark}

\begin{proof}[Proof of Lemma \ref{th:fram:interp_err}]
  To simplify the notation, we define the scalar function $z = y_i$
  for some fixed $i$. Moreover, we prove the result only for $d = 2$;
  for $d = 1$ the result follows from the interpolation error
  estimates established in \cite{OrtnerSuli:2008a}.

  {\it Step 1. $\WW^{2,\infty}$-interpolant: } We first define a
  $\WW^{2,\infty}$-interpolant $\tilz$ of $z$, using the
  $\CC^1$-conforming Hsieh--Clough--Tocher (HCT) element
  \cite{Ciarlet:1978}; see Figure \ref{fig:cloughtocher}.

  \begin{figure}
    \includegraphics[width=3cm]{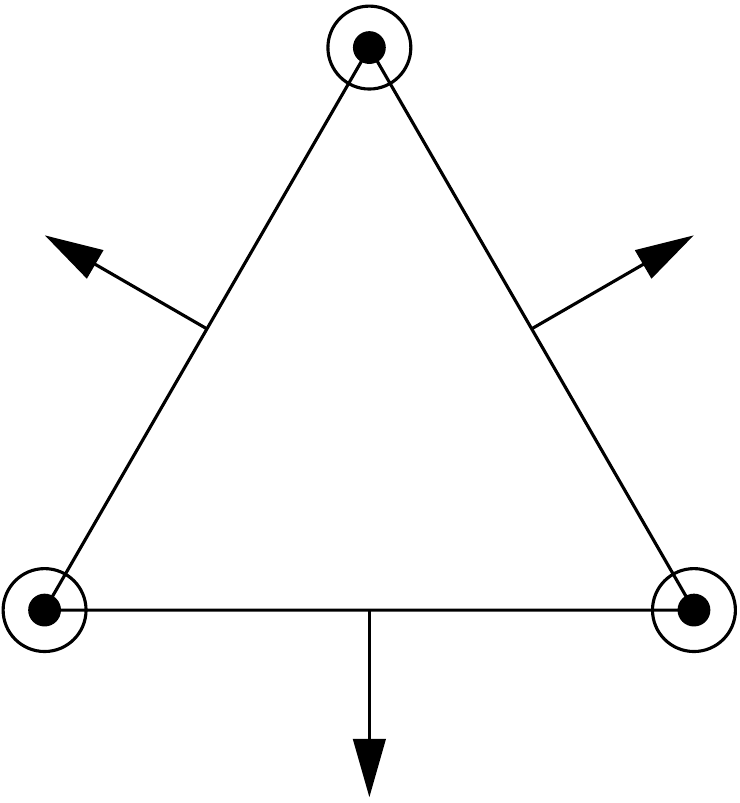}
    \caption{\label{fig:cloughtocher} Illustration of the degrees of
      freedom in the $\CC^1$-conforming Clough--Tocher element: black
      dots denote point values, circles denote gradient values, arrows
      denote directional derivatives.}
  \end{figure}

  Let $T \in \Tepsc$ and let $Q_T$ denote the set of vertices of $T$,
  and $F_T$ the set of edges of $T$.

  For each vertex $q \in Q_T$, we define the point value $\tilz(q) =
  z(q)$, and the gradient value by $\D\tilz(q) = \mint_{\omega_q^\c}
  \D z \dx$, where
  \begin{displaymath}
    \omega_q^\c = \bigcup \b\{ T' \in (\Tepsc)^\per \bsep q \in T' \b\}.
  \end{displaymath}
  Similarly, for each edge $f \in F_T$, $f = T \cap T'$, with midpoint
  $q_f$, we define the patch $\omega_f^\c = \Omc^\per \cap (T \cup
  T')$, and the directional derivative $\Dc{\nu}\tilz(q_f) =
  \mint_{\omega_f^\c} \Dc{\nu} z \dx$.

  Let $\varphi_q$ be the nodal basis function associated with the
  point value at a vertex $q$, $\phi_f$ the nodal basis function
  associated with the normal derivative at an edge $f$, and let
  $\Phi_{q,\alpha}$ be the nodal basis function associated with the
  $\alpha$-component of the derivative $\D z(q)$. 

  {\it Step 2. Estimating $z - \tilz$: } Fix $T \in \Tepsc$, $x \in
  T$, and define $\mF = \D z(T)$, then we have
  \begin{align*}
    \b| \D z(T) - \D \tilz(x) \b| =~&  \b| \mF - \D \tilz(x) \b| \\
    \leq~& \bg| \sum_{f \in F_T} \b( \mF\cdot\nu_f - \Dc{\nu_f} \tilz(q_f) \b)
    \otimes \D \varphi_f(x) \bg|
    + \bg| \sum_{\substack{q \in Q_T \\ \alpha \in \{1,2\}}}
      \b( \mF_\alpha - \pp_{x_\alpha} \tilz(q) \b) \otimes 
      \D \Phi_{q,\alpha} \bg|.
  \end{align*}

  Since all elements $T \in \Teps$ are translated, scaled, and
  possibly reflected, copies of the reference triangle $\hat{T} = {\rm
    conv}\{(0,0), (1,0), (0, 1) \}$, it follows that the HCT nodal basis
  functions are given (up to translations and reflections) by
  \begin{displaymath}
    \varphi_f(x) = \eps \hat{\varphi}_f\b( \eps^{-1} x \b),\quad
    \text{and} \quad
    \hat{\Phi}_{q, \alpha}(x) = \eps \Phi_{q, \alpha}\b( \eps^{-1} x \b).
  \end{displaymath}
  Note in particular, that the gradients of these nodal basis
  functions are scale invariant, that is, 
  \begin{displaymath}
    \| \D \varphi_f \|_{\LL^\infty} \leq C \quad \text{and} \quad
    \| \D \Phi_{q, \alpha} \|_{\LL^\infty} \leq C,
  \end{displaymath}
  where $C$ is a fixed constant that is independent of $\eps$.
  
  From the construction of $\tilz$ it is easy to see that, for $f \in
  F_q, q \in Q_T, \alpha \in \{1,2\}$,
  \begin{displaymath}
    \b| \mF \nu_f - \Dc{\nu_f} \tilz(q_f) \b| \leq \eps\,{\rm osc}(\D
    z; \omega_T^\c), \quad \text{and} \quad
    \b| \mF_\alpha - \pp_{x_\alpha} \tilz(q)
    \b|  \leq \eps\,{\rm osc}(\D
    z; \omega_T^\c);
  \end{displaymath}
  and hence we obtain
  \begin{equation}
    \label{eq:fram:interr:20}
    \b\| \D z(T) - \D \tilz \b\|_{\LL^p(T)} \leq C_1 \eps \, |T|^{1/p}\, {\rm
      osc}(\D z; \omega_T^\c),
  \end{equation}
  for some generic constant $C_1$.

  {\it Step 3. Interpolation error: } Using standard interpolation
  error estimates \cite{Ciarlet:1978}, we obtain
  \begin{displaymath}
    \b\| \D \tilz - \D I_h \tilz \b\|_{\LL^p(\Omc)}
    \leq C_I' \b\| h \D^2 \tilz \b\|_{\LL^p(\Omc)}.
  \end{displaymath}
  
  Let $T \in \Teps$ and $\mF = \D z(T)$, then application of an
  inverse inequality, and \eqref{eq:fram:interr:20} yield
  \begin{align}
    \b\| \D^2 \tilz \b\|_{\LL^p(T)} =~& \b\| \D^2 (\tilz - z)
    \b\|_{\LL^p(T)} 
    %
    \leq C_2 \eps^{-1} \b\| \D \tilz - \mF \b\|_{\LL^p(T)} 
    \label{eq:fram:interr:25}
    \leq  C_2 \, |T|^{1/p} \, {\rm osc}(\D z; \omega_T^\c).
  \end{align}

  Finally, since $\tilz(x) = z(x)$ for all $x \in \Lext$, it follows
  that $I_h z = I_h \tilz$, and hence we can estimate
  \begin{align*}
    \b\| \D z - \D I_h z \b\|_{\LL^p(\Om)} =~& \b\| \D z - \D I_h z
    \b\|_{\LL^p(\Omc)} \\
    \leq~& \b\| \D z - \D \tilz \b\|_{\LL^p(\Omc)} + \b\| \D \tilz -
    \D I_h \tilz \b\|_{\LL^p(\Omc)}.
  \end{align*}
  Employing \eqref{eq:fram:interr:20} and \eqref{eq:fram:interr:25},
  we obtain the stated result.
\end{proof}

\begin{proof}[Proof of Lemma \ref{th:fram:cons_reduceEmodel}]
  For each $f \in \Edg$, let $f = T_{-} \cap T_{+}$, $T_\pm \in \Teps$,
  let $\nu_\pm$ denote the corresponding unit outward normals, and
  $\omega_f' = T_+ \cup T_-$.

  We integrate by parts in each element $T \in \Teps$ and use the fact
  that $I_\eps u_h = u_h$ in $\Omi \cup \Oma$ to obtain
  \begin{align}
    \notag
    \B| \b\< \Phi, I_\eps u_h \b\> - \b\< \Phi_h, u_h \b\> \B|
    =~& \bg| \sum_{T \in \Teps} \int_T  \sigma(T) : \D \b(I_\eps u_h -
    u_h\b) \dx \bg| \\
    \notag
    =~& \bg| \sum_{\substack{f \in \Edg \\ f \not\subset \Omi\cup\Oma}}
    \int_f \b( \sigma(T_+) \nu_+ + \sigma(T_-) \nu_- \b) \cdot
    \b(I_\eps u_h - u_h \b) \ds \bg| \\
    \label{eq:fram:modelerrproof:1}
    \leq~& \sum_{\substack{f \in \Edg \\ f \not\subset \Omi\cup\Oma}}
    \eps {\rm osc}(\sigma; \omega_f')\,
    \int_f \b| I_\eps u_h - u_h \b| \ds.
  \end{align}

  Let $v := I_\eps u_h - u_h$. An application of \cite[Lemma
  6.6]{Ortner:Thesis} yields the trace inequality
  \begin{equation}
    \label{eq:fram:L1_trace}
    \| v \|_{\LL^1(f)} \leq \eps^{-1} \| v \|_{\LL^1(\omega_f')} +
    \| \D v \|_{\LL^1(\omega_f')}.
  \end{equation}
  Furthermore, since $v$ is Lipschitz continuous and $v(p) = 0$ on
  every vertex of the triangulation $\Teps$, we can use \cite[Lemma
  6.8]{Ortner:Thesis} to deduce that
  \begin{equation}
    \label{eq:fram:L1_pointFriedrich}
    \| v \|_{\LL^1(\omega_f')} \leq
    \sqrt{2} \eps \| \D v \|_{\LL^1(\omega_f')}.
  \end{equation}

  Combining \eqref{eq:fram:L1_pointFriedrich},
  \eqref{eq:fram:L1_trace}, and \eqref{eq:fram:modelerrproof:1},
  applying two H\"{o}lder inequalities, and estimating the overlaps
  between the patches $\omega_f'$, we deduce that
  \begin{align*}
    \B| \b\< \Phi, I_\eps u_h \b\> - \b\< \Phi_h, u_h \b\> \B| 
    \leq~& (1+\sqrt{2}) \eps \sum_{\substack{f \in \Edg \\ f \not\subset \Omi\cup\Oma}}
    {\rm osc}(\sigma; \omega_f') |\omega_f'|^{1/p} 
    \b\| \D v \b\|_{\LL^{p'}(\omega_f')} \\
    \leq~& C_1 \eps \bg(\sum_{T \in \Thc} |T| {\rm osc}(\sigma;
    \omega_T^\c)^p \bg)^{1/p} \b\| \D v \b\|_{\LL^{p'}(\Omc)}.
  \end{align*}
  An application of Lemma \ref{th:fram:Dbarvh_Dvh} yields the stated
  result.
\end{proof}

\section{List of Symbols}
\label{sec:notation}
\begin{longtable}{rcl}
  $a \cdot b$, $a \otimes b$ && vector dot product and tensor product;
  \S\ref{sec:intro:basic_notation}  \\
  $|\cdot|, |\cdot|_p$ && $\ell^p$-norms;
  \S\ref{sec:intro:basic_notation} \\
  $\| \cdot \|_{\ell^p_\eps}$ && weighted $\ell^p$-norms;
  \S\ref{sec:intro:basic_notation} \\
  $\L, \Lext$ && Lattice and lattice domain;
  \S\ref{sec:intro:at_model:perdef} \\
  $\yA$ && homogeneous deformation;
  \S\ref{sec:intro:at_model:perdef} \\
  $\Us, \Ys, \YsA$ && spaces of periodic displacements and deformations
  \S\ref{sec:intro:at_model:perdef} \\
  $A^\per, \mathscr{A}^\per$ && periodic extension of a set or family
  of sets
  \S\ref{sec:intro:at_model:perdef} \\
  $\Rg$ && interaction range, \S\ref{sec:intro:defn_Ea} \\
  $\Da{r}, \Da{\Rg}$ && finite difference operator and stencil 
  \S\ref{sec:intro:defn_Ea}  \\
  $\Dc{r}, \D$ && directional derivative, deformation or displacement
  gradient, 
  \S\ref{sec:intro:basic_notation} \\
  $\pp v$ && Jacobi matrix of vector valued function, 
  \S\ref{sec:intro:basic_notation} \\
  $\Ea$ && atomistic energy, \S\ref{sec:intro:defn_Ea} \\
  $V$ && atomistic interaction potential, \S\ref{sec:intro:defn_Ea} \\
  $\Pa$ && external potential in atomistic model,
  \S\ref{sec:intro:minA} \\
  $\del \E$, $\ddel \E$, $\< \cdot, \cdot \>$ && first and second
  variations, abstract duality pairing, 
  \S\ref{sec:intro:minA} \\
  $\pp_r V, \pp_{r,s} V$ && first and second partial derivatives of
  $V$, \S\ref{sec:intro:assV} \\
  $M_{r,s}^\a, M^\a$ && bounds on $\pp_{r,s} V$ and Lipschitz constant
  for $\del\Ea$, \S\ref{sec:intro:assV} \\ 
  $\Th, \Thext, \Edgh, \Edghext$ && triangulations, and edge sets, 
  \S\ref{sec:intro:galerkin} \\
  $h_T, h(x)$ && mesh size functions,
  \S\ref{sec:intro:galerkin} \\
  $\PI, \PO, \PIper, \POper$ && finite element spaces, 
  \S\ref{sec:intro:galerkin} \\
  $\Ush, \Ysh, \YshA$ && finite element spaces, 
  \S\ref{sec:intro:galerkin} \\
  $I_h$ && nodal interpolation operator for $\PI(\Th)$,
  \S\ref{sec:intro:galerkin} \\
  $W$ && Cauchy--Born stored energy function,
  \S\ref{sec:intro:qce} \\
  $\Eac$, $\Pac$ && a/c energy and external potential,
  \S\ref{sec:intro:gen_int_corr} \\
  $\Oma, \Omc, \Omi$ && atomistic, continuum, and interface region,
  \S\ref{sec:intro:gen_int_corr} \\
  $\Tha, \Thc, \Thi$ && atomistic, continuum, and interface
  triangulations,
  \S\ref{sec:intro:gen_int_corr} \\
  $\La$ && set of atomistic sites in a/c method, 
  \S\ref{sec:intro:gen_int_corr} \\
 $(x, x'), (x, x+\eps r)$ && bonds, 
  \S\ref{sec:intro:gen_int_corr} \\
  $\Bi$, $\Bi^\per$ && set of interface bonds,
  \S\ref{sec:intro:gen_int_corr} \\
  $\Ei$, $\Fi$ && interface functional, 
  \S\ref{sec:intro:gen_int_corr} \\
  $\pp_b \Fi, \pp_{(x,x+\eps r)} \Fi$ && scaled first partial derivatives of
  $\Fi$, 
  \S\ref{sec:intro:loc_scal} \\
  $M_{r,s}^\i$, $M^\i$ && bounds on second partial derivatives of $\Fi$,
  \S\ref{sec:intro:loc_scal} \\
  $\Teps, \Tepsext, \Edg, \Edgext$ && atomistic triangulation and edge
  sets,
  \S\ref{sec:fram:disc_cts} \\
  ${\rm osc}$ && oscillation operator, 
  \S\ref{sec:fram:disc_cts} \\
  $I_\eps$ && nodal interpolation operator for $\PI(\Teps)$,
  \S\ref{sec:fram:disc_cts} \\
  $\omega_T^\c$ && patch used in the interpolation error estimate,
  \S\ref{sec:fram:interp} \\
  ${\rm h}_T$ && modified mesh size function
  \S\ref{sec:fram:interp} \\
  $\|\cdot\|_{\WW^{-1,p}_h}$ && $\WW^{1,p'}$-dual norm on $\PI(\Th)^*$, 
  \S\ref{sec:fram:err} \\
  $\|\cdot\|_{\WW^{-1,p}_\eps}$ && $\WW^{1,p'}$-dual norm on $\PI(\Teps)^*$, 
  \S\ref{sec:fram:cons} \\
  $\phi_1, \phi_2$ && first and second neighbour potential,
  \S\ref{sec:e1d} \\
  $x_n, v_n, v_n', v_n'', v_n'''$ && notation for 1D grid functions,
  \S\ref{sec:e1d} \\
  $\chi_T$ && characteristic function used in bond density lemma,
  \S\ref{sec:aux:bdl} \\
  $\chi_T^\i$ && characteristic function used to define $\Sac$, Prop. \ref{th:2d:Sac}, p.\pageref{th:2d:Sac} \\
  $\mint_{x}^{x'} f \db, \mint_{x}^{x+\eps r} f \db$ && bond integrals, 
  \S\ref{sec:aux:bdl} \\
  $\CR$, $\CRper$ && Crouzeix--Raviart finite element spaces, 
  \S\ref{sec:aux:CR_space} \\
  $q_f$, $\zeta_f$, $\zeta_f^\per$ && midpoint of an edge $f$ and
  associated nodal basis, 
  \S\ref{sec:aux:CR_space} \\
  $\int_\gamma \sigma \cdot \dx$ && path integral,
  \S\ref{sec:aux:pathints}  \\
  $\mJ$ && rotation about $\pi/2$,
  Lemma~\ref{th:aux:divfree}, p.\pageref{th:aux:divfree} \\
  $\omega_T^\a, \omega_T$ && atomistic interaction neighbourhoods,
  \S\ref{sec:2d} \\
  $M_T$ && prefactors in modelling error estimate,
  Eq. \eqref{eq:2d:defn_barM}, p.\pageref{eq:2d:defn_barM} \\
  ${\rm width}(\Omi)$ && width of $\Omi$,
  Eq. \eqref{eq:2d:defn_widthOmi}, p.\pageref{eq:2d:defn_widthOmi} \\
  $\Sa$ && atomistic stress function, Eq. \eqref{eq:2d:defn_Sa}, p.\pageref{eq:2d:defn_Sa}  \\
  $V_{x,r}, V_{\mF, r}$ && alternative notation for $\pp_r V(\Da{\Rg}
  y(x))$ and for  $\pp_r V(\mF \Rg)$  \\
  $\Sac$ && a/c stress function, Prop. \ref{th:2d:Sac}, p.\pageref{th:2d:Sac} \\
  $\w(\mF; \cdot)$ && corrector function for $\Sac(\yF)$, Lemma
  \ref{th:2d:cons_of_pt}, p.\pageref{th:2d:cons_of_pt} \\
  $\modw(y; \cdot)$ && corrector function for $\Sac(y)$,
  Eq. \eqref{eq:2d:defn_modwh}, p.\pageref{eq:2d:defn_modwh} \\
  $\Sacm$ && modified a/c stress function,
  Eq. \eqref{eq:2d:defn_Sacm}, p. \pageref{eq:2d:defn_Sacm}
  \\
  $\Rac(y; T)$ && stress error, Eq. \eqref{eq:2d:defn:Rac}, p. \pageref{eq:2d:defn:Rac} \\
\end{longtable}

\bibliographystyle{plain}
\bibliography{qc}

\end{document}

%% file: notation.tex
\renewcommand{\cases}[1]{\left\{ \begin{array}{rl} #1 \end{array} \right.}
\newcommand{\smfrac}[2]{{\textstyle \frac{#1}{#2}}}

\newcommand{\mymat}[1]{\left[ \begin{matrix} #1 \end{matrix} \right]}

\def\Xint#1{\mathchoice
{\XXint\displaystyle\textstyle{#1}}%
{\XXint\textstyle\scriptstyle{#1}}%
{\XXint\scriptstyle\scriptscriptstyle{#1}}%
{\XXint\scriptscriptstyle\scriptscriptstyle{#1}}%
\!\int}
\def\XXint#1#2#3{{\setbox0=\hbox{$#1{#2#3}{\int}$ }
\vcenter{\hbox{$#2#3$ }}\kern-.6\wd0}}
\def\mint{\Xint-}

\def\b{\big}
\def\B{\Big}
\def\bg{\bigg}

\def\sep{\,|\,}
\def\bsep{\,\b|\,}
\def\Bsep{\,\B|\,}

\def\diam{{\rm diam}}

\def\argmin{{\rm argmin}}

\def\R{\mathbb{R}}
\def\N{\mathbb{N}}
\def\Z{\mathbb{Z}}

\def\WW{{\rm W}}
\def\CC{{\rm C}}

\def\LL{{\rm L}}

\def\dx{\,{\rm d}x}

\def\dt{\,{\rm d}t}
\def\ds{\,{\rm d}s}
\def\dd{{\rm d}}
\def\pp{\partial}

\def\db{\,{\rm db}}

\def\<{\langle}
\def\>{\rangle}

\def\per{\#}

\def\mA{{\sf A}}
\def\mB{{\sf B}}
\def\mF{{\sf F}}
\def\mG{{\sf G}}

\def\mI{{\sf I}}
\def\mJ{{\sf J}}

\def\mR{{\sf R}}

\def\eps{\varepsilon}
\def\tot{{\rm tot}}

\def\ac{{\rm ac}}
\def\a{{\rm a}}
\def\c{{\rm c}}
\def\i{{\rm i}}
\def\qce{{\rm qce}}

\newcommand{\Da}[1]{D_{\!#1}}
\newcommand{\Dc}[1]{\nabla_{\!\!#1}}
\def\D{\nabla}
\def\del{\delta}
\def\ddel{\delta^2}

\def\jl{[\![}
\def\jr{]\!]}

\def\L{\mathscr{L}}
\def\Lext{\L^\per}

\def\Rg{\mathscr{R}}
\def\Us{\mathscr{U}}
\def\Ys{\mathscr{Y}}
\def\YsA{\mathscr{Y}_\mA}

\def\yF{y_\mF}
\def\yG{y_\mG}
\def\yB{y_\mB}
\def\yA{y_\mA}

\def\E{\mathscr{E}}
\def\Ea{\E_\a}

\def\Pa{\mathscr{P}_\a}

\def\Eatot{\E_\a^\tot}
\def\Eactot{\E_\ac^\tot}

\def\Om{\Omega}
\def\Oma{\Om_\a}
\def\Omc{\Om_\c}

\def\T{\mathscr{T}}
\def\Th{\T_h}
\def\Thext{\T_h^\#}
\def\Ush{\mathscr{U}_h}
\def\Ysh{\mathscr{Y}_h}
\def\YshA{\Ys_{\mA,h}}
\def\PO{{\rm P}_0}
\def\POper{{\rm P}_0^\per}
\def\PI{{\rm P}_1}
\def\PIper{{\rm P}_1^\per}
\def\CR{{\rm N}_1}
\def\CRper{{\rm N}_1^\per}
\def\Edgh{\mathscr{F}_h}
\def\Edghext{\mathscr{F}_h^\per}
\def\Sh{\PI(\Thext)}

\def\Teps{\T_\eps}
\def\Tepsext{\T_\eps^\per}
\def\Tepsa{\T_\eps^\a}
\def\Tepsc{\T_\eps^\c}
\def\Tepsi{\T_\eps^\i}
\def\Edg{\mathscr{F}_\eps}
\def\Edgext{\mathscr{F}_\eps^\per}

\def\Laqce{\L_\a^\qce}

\def\Omcqce{\Om_\c^\qce}
\def\Eqce{\E_{\rm qce}}

\def\gc{{\rm gc}}
\def\Lagc{\L_\a^\gc}
\def\Ligc{\L_\i^\gc}

\def\Pac{\mathscr{P}_\ac}
\def\Eac{\E_\ac}

\def\Fi{E_\i}
\def\Ei{\E_\i}

\def\La{\L_\a}
\def\Laext{\L_\a^\per}

\def\Oma{\Om_\a}
\def\Omc{\Om_\c}
\def\Omi{\Om_\i}
\def\Tha{\T_h^\a}
\def\Thc{\T^\c_h}
\def\Thi{\T^\i_h}
\def\tily{\tilde y}
\def\tilz{\tilde z}
\def\modV{\widetilde{V}}

\def\w{\psi}
\def\modw{\hat{\psi}}

\def\Bi{\mathscr{B}_\i}
\def\Biext{\mathscr{B}_\i^\per}

\def\Emodel{\mathcal{E}^{\rm model}_h}
\def\Emodeleps{\mathcal{E}^{\rm model}_\eps}
\def\Ecoarse{\mathcal{E}^{\rm coarse}_h}
\def\Eext{\mathcal{E}^{\rm ext}_h}

\def\Eqnl{\E_{\rm qnl}}
\def\Laqnl{\mathscr{N}_\a}
\def\Lcqnl{\mathscr{N}_\c}

\def\Sa{\Sigma_\a}

\def\Sac{\Sigma_\ac}
\def\Sacm{\widehat{\Sigma}_\ac}
\def\Rac{{\sf R}}